\documentclass[10pt,reqno]{amsart}

\usepackage[utf8]{inputenc} 

\usepackage[margin=1in]{geometry} 
\usepackage{graphicx} 

\usepackage{booktabs} 
\usepackage{array} 
\usepackage{paralist} 
\usepackage{verbatim} 
\usepackage{subfig} 
\usepackage{mathrsfs}
\usepackage{amssymb}
\usepackage{amsthm}
\usepackage{amsmath,amsfonts,amssymb,esint}
\usepackage{graphics}
\usepackage{enumerate}
\usepackage{mathtools}
\usepackage{xfrac}
\usepackage{bbm}
\usepackage{enumitem}
\usepackage{soul}

\usepackage{fancyhdr} 
\usepackage{times}

\usepackage[usenames,dvipsnames]{color}
\usepackage[colorlinks=true, pdfstartview=FitV, linkcolor=blue, citecolor=blue, urlcolor=blue]{hyperref}

\definecolor{labelkey}{rgb}{0,0,1}

\numberwithin{equation}{section}

\newtheorem{theorem}{Theorem}[section]
\newtheorem{corollary}[theorem]{Corollary}
\newtheorem{proposition}[theorem]{Proposition}
\newtheorem{lemma}[theorem]{Lemma}
\newtheorem{definition}{Definition}[section]
\theoremstyle{definition}
\newtheorem{remark}[theorem]{Remark}

\newcommand{\norm}[1]{\left\|#1\right\|}
\newcommand{\abs}[1]{\left|#1\right|}

\newcommand*{\supp}{\ensuremath{\mathrm{supp\,}}}
\newcommand*{\Id}{\ensuremath{\mathrm{Id}}}
\newcommand*{\RSZ}{\ensuremath{\mathcal{R}}}
\renewcommand*{\div}{\ensuremath{\mathrm{div\,}}}

\newcommand*{\N}{\ensuremath{\mathbb{N}}}
\newcommand*{\T}{\ensuremath{\mathbb{T}}}

\newcommand{\eps}{\varepsilon}

\newcommand{\BB}{\mathcal B}
\newcommand{\GG}{\mathcal G}

\newcommand{\vo}{\overline v_q}
\newcommand{\RR}{\mathring R}
\newcommand{\RRO}{\mathring{\overline{R}}_q}

\renewcommand*{\tilde}{\widetilde}

\newcommand*{\curl}{\ensuremath{\mathrm{curl\,}}}

\newcommand{\les}{\lesssim}

\newcommand{\Proj}{\ensuremath{\mathbb{P}}}

\title{Wild solutions of the Navier-Stokes equations whose singular sets in time have Hausdorff dimension strictly less than $1$}

\author{Tristan Buckmaster}
\address{Department of Mathematics, Princeton University, Princeton, NJ 08544, USA}
\email{buckmaster@math.princeton.edu}

\author{Maria Colombo}
\address{EPFL SB, Station 8,  CH-1015 Lausanne, Switzerland}
\email{maria.colombo@epfl.ch}
 
\author{Vlad Vicol}
\address{Courant Institute for Mathematical Sciences, New York Univeristy, New York, NY 10012, USA}
\email{vicol@cims.nyu.edu}

\begin{document}


\begin{abstract}
We prove non-uniqueness for a class of weak solutions to the Navier-Stokes equations which have bounded kinetic energy,  integrable vorticity, and are smooth outside a fractal set of singular times with Hausdorff dimension strictly less than $1$.
\end{abstract}

\maketitle


\section{Introduction}
\label{sec:intro}

Throughout this paper we consider the \emph{incompressible three dimensional Navier-Stokes equations}:
\begin{subequations}\label{eq:NSE:*}
\begin{align}
\partial_t v+ \div (v \otimes  v ) +\nabla p - \Delta v &= 0 \label{eq:NSE:*:a} \\
\div v &= 0 \label{eq:NSE:*:b} \\
v|_{t=0} &= v_0
\end{align}
\end{subequations}
posed on the torus $\T^3=[-\pi,\pi]^3$. We consider solutions of zero mean, i.e.\ $\int_{\T^3} v(x,t) dx = 0$ for all $t \in [0,T]$. 
The notion of \emph{weak solution} of \eqref{eq:NSE:*} that we work with in this paper is that of a \emph{distributional solution}, which has bounded kinetic energy, and is {strongly} continuous in time:
\begin{definition}[\bf Weak solution]
\label{def:weak}
Given any zero mean initial datum $v_0 \in L^2$, we say $v \in C^0([0,T);L^2(\T^3))$ is a \emph{weak} solution of the Cauchy problem for the Navier-Stokes equations \eqref{eq:NSE} if the vector field $v(\cdot,t)$ is weakly divergence-free for all $t \in [0,T)$, has zero mean, and
\begin{align*}
\int_{\T^3}  v_0 \cdot \varphi(\cdot,0) dx + \int_{0}^T \! \int_{\T^3} v \cdot (\partial_t \varphi + (v \cdot \nabla) \varphi +   \Delta \varphi ) dx dt = 0
\end{align*}
holds for any test function $\varphi \in C_0^\infty(\T^3 \times [0,T))$ such that $ \varphi(\cdot,t) $ is divergence-free for all $t$.
\end{definition}
In view of the $C^0([0,T);L^2(\T^3))$ regularity, by~\cite{FabesJonesRiviere72} we have that the above defined weak solutions are also  {\em mild} or \emph{Oseen} solutions (see also~\cite[Chapter 6]{Lemarie16}) of the Navier-Stokes equations. That is, for $t\in [0,T)$ we have
\begin{align}
v(\cdot,t) = e^{  t  \Delta} v_0 + \int_0^t  e^{  (t-s)\Delta}\Proj_H \div(v(\cdot,s)\otimes v(\cdot,s)) ds \, .
\label{eq:mild}
\end{align}
Here $\Proj_H$ is the Helmholtz projector and $e^{t\Delta}f $ is the heat extension of $f$. Our main result of this paper is as follows.

\begin{theorem}[\bf Main result]
\label{thm:RT}
There exists a $\beta>0$ such that the following holds.
For $T>0$, let $u^{(1)}, u^{(2)} \in C^0([0,T];\dot{H^3}(\T^3))$ be two strong solutions  of the Navier-Stokes equations \eqref{eq:NSE:*:a}--\eqref{eq:NSE:*:b} on $[0,T]$, with data $u^{(1)}(0,x)$ and $u^{(2)}(0,x)$ of zero mean. There exists a weak solution $v$ of the Cauchy problem to \eqref{eq:NSE:*} on $[0,T]$  with initial datum $v|_{t=0} = u^{(1)}|_{t=0}$, which has the additional regularity
\[ 
v \in C^0([0,T];H^\beta(\T^3) \cap W^{1,1+\beta}(\T^3)) \, ,
\] 
and such that 
\[
v \equiv u^{(1)} \quad \mbox{on} \quad [0,\sfrac{T}{3}], \qquad \mbox{and} \qquad  v \equiv u^{(2)} \quad \mbox{on} \quad [\sfrac{2T}{3},T]\, .
\]
Moreover, for every such $v$ there exists a zero Lebesgue measure set of times $\Sigma_T \subset (0,T]$ with Hausdorff (in fact box-counting) dimension less than $1-\beta$, such that 
\[
v \in C^\infty \bigl( ((0,T]\setminus \Sigma_T) \times \T^3\bigr). 
\]
In particular, the weak solution  $v$ is almost everywhere smooth.
\end{theorem}

The outline of the proof of Theorem~\ref{thm:RT} is given in Section~\ref{sec:outline}, while the detailed estimates are done in Sections~\ref{sec:gluing}--\ref{sec:stress}.

\begin{remark}[\bf Non-uniqueness of weak solutions for strong initial datum]
\label{rem:nonunique:1}
Theorem~\ref{thm:RT} immediately implies that weak solutions of the Cauchy problem for the Navier-Stokes equation~\eqref{eq:NSE:*}, cf.~Definition~\ref{def:weak}, are not unique.

The cheap way to see this is to take any $T>0$, $u^{(1)} \equiv 0$, and $u^{(2)}$ to be any nontrivial mean-free solution of the Navier-Stokes equation on $[0,T]$ (e.g.\ a shear flow). Then the weak solution $v$ given by Theorem~\ref{thm:RT} is nontrivial on $[0,T]$, and thus $0$ is not the only weak solution with $0$ initial datum. Conversely, we note that taking $u^{(1)}$ to be any nontrivial solution to the Navier-Stokes equation, and $u^{(2)} \equiv 0$, Theorem~\ref{thm:RT} gives a counterexample to backwards (in time) uniqueness for weak solutions of \eqref{eq:NSE:*}, in the sense of Definition~\ref{def:weak}.

More generally, we emphasize that Theorem~\ref{thm:RT} proves the non-uniqueness of weak solutions to the Cauchy problem for the Navier-Stokes equation~\eqref{eq:NSE:*} for \emph{any strong initial datum}. To see this, consider any $v_0 \in \dot{H}^3$ and take $T = c \norm{v_0}_{H^3}^{-1}$, where $c>0$ is a sufficiently small universal constant (cf.~Proposition~\ref{prop:local:existence}). Then there exists a unique solution $u^{(1)} \in C^0([0,T];H^3)$ to the Cauchy problem \eqref{eq:NSE} with datum $v_0$. Moreover $\norm{u^{(1)}(T)}_{L^2} \leq \norm{v_0}_{L^2}$. However, using Theorem~\ref{thm:RT}  one can glue to this solution the shear flow $u^{(2)}(x_1,x_2,x_3,t)  = (A e^{-t} \sin( x_2), 0,0)$. Then if $A$ is chosen such that $A e^{-T } > 2 \norm{v_0}_{L^2}$, we have $\norm{v(T)}_{L^2} = \norm{u^{(2)}(T)}_{L^2} > \norm{v_0}_{L^2} \geq \norm{u^{(1)}(T)}_{L^2}$. Therefore $v$ is a weak solution to \eqref{eq:NSE} with datum $v_0$, but $v$ is not equal to the smooth solution $u^{(1)}$ at time $T$. 

While for the above argument we have considered $v_0 \in \dot{H^3}$, it is clear that Theorem~\ref{thm:RT} also implies the non-uniqueness of weak solutions to the Cauchy problem for \eqref{eq:NSE:*} for any initial datum for which one has unique local in time solvability of \eqref{eq:NSE:*} (examples include $v_0 \in \dot{H}^{1/2}$ cf.~\cite{FujitaKato64}, $v_0 \in L^3$ with zero mean cf.~\cite{Kato84}; $v_0 \in BMO^{-1}$ which is small and has zero mean cf.~\cite{KochTataru01}; see~\cite{LemarieRieusset02} for further details). Indeed, for any such initial datum the unique local in time solution $u^{(1)}$ is smooth in positive time, and hence for any $\eps>0$ we have $u^{(1)}(\cdot,\eps) \in \dot{H}^3$. We then apply Theorem~\ref{thm:RT} on the time interval $[\eps,T]$, rather than $[0,T]$, in order to glue the strong solution to a shear flow with  kinetic energy which is either strictly larger, or strictly less at time $T$. 
\end{remark}

\subsection{Background}

We make a few comments concerning different notions of solutions to the Navier-Stokes equation, other than  from those in Definition~\ref{def:weak} (see~\cite{Lemarie16} for a more detailed discussion). The weakest notion of solution to the Cauchy problem for \eqref{eq:NSE:*} is that of a \emph{very weak solution}: these are distributional solutions of \eqref{eq:NSE:*} which only lie in $C^0_{\rm weak}(0,T;L^2)$, and are weakly divergence free. However, one typically proves the existence of solutions which are stronger than this. 

Indeed, for any $L^2$ initial datum $v_0$, Leray~\cite{Leray34} constructed a distributional solution $v \in C^0_{\rm weak}(0,\infty;L^2) \cap L^2(0,\infty;\dot{H}^1)$, and obeys the energy inequality $\norm{v(t)}_{L^2}^2 + 2 \int_s^t \norm{\nabla v(\tau)}_{L^2}^2 d\tau \leq \norm{v(s)}_{L^2}^2$ for a.e.\ $s\geq 0$, and all $t > s$. See also the work of Hopf~\cite{Hopf51} on bounded domains. These are the \emph{Leray-Hopf weak solutions}. One nice feature of Leray-Hopf weak solutions is that they possess epochs of regularity, i.e.\ many time intervals on which they are smooth. In fact, already Leray~\cite{Leray34} made the observation that these weak solutions are almost everywhere in time smooth, since the  putative \emph{singular set of times} $\Sigma_T$ has Hausdorff dimension $\leq \sfrac 12$. This fact follows directly from two ingredients: the fact that for $v_0 \in H^1$ the maximal time of existence of a unique smooth solution is bounded from below by $c \norm{v_0}_{H^1}^{-4}$, and a Vitali-type covering lemma which may be combined with the $L^2_t H^1_x$ information provided by the energy inequality. Scheffer~\cite{Scheffer76} went further to prove that the $\sfrac 12$-dimensional Hausdorff measure of $\Sigma_T$ is $0$. These results were strengthened to bounds on the box-counting dimension for $\Sigma_T$, cf.~\cite{RobinsonSadowski07,Kukavica09}. See~\cite{Lemarie16,RSR16} for further references. 

\begin{remark}[\bf Weak solutions with partial regularity in time]
We note that while the weak solutions constructed in Theorem~\ref{thm:RT} are not Leray-Hopf, they give the {first example} of a mild/weak solution to the Navier-Stokes equation whose singular set of times $\Sigma_T \subset (0,T]$ is both \emph{nonempty}, and has Hausdorff (in fact, box-counting) dimension {\em strictly less than $1$}. This is in contrast with the prior work~\cite{BV}, where $\Sigma_T$ has dimension $1$. It is in an interesting open problem to construct weak solutions to \eqref{eq:NSE:*}, in the sense of Definition~\ref{def:weak}, where the $\sfrac 12$-dimensional Hausdorff measure of the nonempty set of singular times is $0$.
\end{remark}

A fundamental step towards understanding the uniqueness and smoothness of weak solutions was to introduce the concept of a \emph{suitable weak solution}, by Scheffer~\cite{Scheffer76} and Caffarelli-Kohn-Nirenberg~\cite{CaffarelliKohnNirenberg82}.  Suitable weak solutions obey a localized in space-time version of the energy inequality, and they have partial regularity in space and time: the putative singular set of points in space-time has $1$-dimensional parabolic Hausdorff measure is $0$. See the reviews~\cite{RSR16,Lemarie16} for more recent extensions and further references.

The uniqueness of suitable weak solutions or of Leray-Hopf weak solutions is an outstanding open problem. The \emph{weak-strong uniqueness} result of Prodi-Serrin~\cite{Prodi59,Serrin62} states that if there exists a weak/mild solution $v \in L^\infty_t L^2_x \cap L^2_t \dot{H}^1_x \cap L^p_t L^q_x$ of the Cauchy problem for \eqref{eq:NSE:*}, with $\sfrac 2p + \sfrac 3q \leq 1$\footnote{The $L^p_t L^q_x$ norm, for $\sfrac 2p + \sfrac 3q = 1$,  is invariant under the Navier-Stokes scaling map $v(x,t) \mapsto v_\lambda(x,t) = \lambda v(\lambda x,\lambda^2 t)$. Spaces that obey these properties are called \emph{scaling critical} spaces. Since the Leray-Hopf energy space $L^\infty_t L^2_x \cap L^2_t \dot{H}^1_x$ obeys $\sfrac 2\infty + \sfrac 32 = \sfrac 22 + \sfrac 36 = \sfrac 32 > 1$, we may call the system \eqref{eq:NSE:*} \emph{energy supercritical}. } and $p<\infty$, and if $u$ is a Leray-Hopf weak solution with the same initial datum, then $u\equiv v$. This is a conditional uniqueness result within the class of Leray-Hopf weak solutions. Moreover, the solutions are smooth in positive time~\cite{Lady67}. The $L^\infty_t L^3_x$ endpoint was established in~\cite{EscauriazaSerginSverak03}. Similar weak-strong uniqueness results hold within the class of mild solutions, except the $q=3$ endpoint which requires continuity in time~\cite{FabesJonesRiviere72,FurioliLemarieRieussetTerraneo00,LionsMasmoudi01}. See~\cite[Chapter 12]{Lemarie16} for further references. A very interesting conjecture of Jia-\v Sver\'ak~\cite{JiaSverak14,JiaSverak15} essentially states that the Prodi-Serrin uniqueness criteria are sharp, and that the non-uniqueness of Leray-Hopf weak solutions may already be expected in the regularity class $L^\infty_t L^{3,\infty}_x$. Compelling numerical evidence in support of this conjecture was recently provided by Guillod-\v Sver\'ak~\cite{GuillodSverak17}.  A related interesting open problem is to establish the non-uniqueness of mild/weak solutions to \eqref{eq:NSE:*} in the regularity class $C^0_t L^{q}_x \cap L^2_t H^1_x$, for any $q \in [2, 3)$.

We conclude this subsection by revisiting the non-uniqueness result of Remark~\ref{rem:nonunique:1}, for rough initial datum:

\begin{remark}[\bf Non-uniqueness of very weak solutions for any $L^2$ initial datum]
\label{rem:nonunique:2}
If instead of the weak solutions of Definition~\ref{def:weak} we consider \emph{very weak solutions} of \eqref{eq:NSE:*}, so they only lie in $C^0_{\rm weak}(0,T;L^2)$, then Theorem~\ref{thm:RT} implies that the non-uniqueness for the Cauchy problem holds for any $L^2$ initial datum of zero mean, within the class of very weak solutions. Indeed, for any such datum, by the work of Leray there exists at least one very weak solution $u$ to the Cauchy problem for \eqref{eq:NSE:*}, which in fact is smooth most of the time. Pick any regular time $t_0>0$ of $u$, and let $v_0 = u(t_0) \in \dot{H}^3$. We then apply the argument of Remark~\ref{rem:nonunique:1} on the time interval $[t_0,t_0+T]$, with $u^{(1)}$ being the unique local in time smooth solution of \eqref{eq:NSE:*} with initial datum $v_0$ at time $t_0$. Note that by weak-strong uniqueness we in fact have that the Leray solution $u$ is equal to $u^{(1)}$ on $[t_0,t_0+T]$. In view of Theorem~\ref{thm:RT} we can construct a very weak solution $v$ which is equal to $u$ on $[0,t_0 + \sfrac T3]$, and equal to a shear flow of our choice on $[t_0+\sfrac{2T}{3},T]$. This solution $v$ is smooth except for a set of times of Hausdorff dimension $<1$, and is different from the Leray solution $u$.
\end{remark}

\subsection{The energy supercritical hyperdissipative Navier-Stokes equation}
The proof of Theorem~\ref{thm:RT} uses essentially that the kinetic energy space is supercritical with respect to the natural scaling invariance associated to \eqref{eq:NSE:*}. In fact, the proof applies \emph{mutatis mutandis} to the energy supercritical $\alpha$-hyperdissipative Navier-Stokes equation
\begin{subequations}\label{eq:NSE}
\begin{align}
\partial_t v+ \div (v \otimes  v ) +\nabla p + (- \Delta)^\alpha v &= 0\\
\div v &= 0 \\
v|_{t=0} &= v_0\,.
\end{align}
\end{subequations}
Here we consider the \emph{energy supercritical regime} $\alpha \in [1,5/4)$. Indeed, \eqref{eq:NSE} is invariant under the scaling map $v(x,t) \mapsto v_\lambda(x,t) = \lambda^{2\alpha-1} v(\lambda x,\lambda^{2\alpha} t)$, and the energy norm $L^\infty_t L^2_x$ is invariant under this map for $\alpha = \sfrac 54$. Definition~\ref{def:weak}, with $\Delta \varphi$ replaced by $- (-\Delta)^{\alpha} \varphi$, gives the notion of a weak solution for  \eqref{eq:NSE}. Our result is:

\begin{theorem}[\bf The hyperdissipative problem]
\label{thm:main}
For $\alpha \in [1,\sfrac 54)$ there exists $\beta = \beta(\alpha) > 0$ such that Theorem~\ref{thm:RT}, and thus also Remark~\ref{rem:nonunique:1}, holds with system~\eqref{eq:NSE:*} replaced by the more general system~\eqref{eq:NSE}.
\end{theorem}

The system \eqref{eq:NSE} was first considered by Lions in \cite{Lions59,LionsJ2} for $\alpha$ in the critical and subcritical regime $\alpha \geq \sfrac 54$. Lions proved the existence and uniqueness of Leray-weak solutions, for any $L^2$ initial datum. These solutions are regular in positive time. In~\cite{Tao} it was proven that slightly below the critical threshold $\alpha= \sfrac 54$ the existence of a globally regular solution still holds when the right-hand side of the first equation in \eqref{eq:NSE} is replaced by a logarithmically supercritical operator.  For $\alpha \in [\sfrac 34,1)$ and $(1,\sfrac 54)$ partial regularity results \`a la Caffarelli-Kohn-Nirenberg were established in~\cite{KatzPavlovic02,TangYu} and \cite{CDLM17}. These works show the existence of a weak solution whose putative singular set (in space-time) has $(5-4\alpha)$-dimensional Hausdorff measure $0$. In the opposite direction, the recent works~\cite{CDLDR17,DR08} prove the non-uniqueness of Leray-weak solutions to \eqref{eq:NSE} in the parameter ranges $\alpha < \sfrac 15$, respectively $\alpha < \sfrac 13$. The non-uniqueness of weak solutions in the sense of Definition~\ref{def:weak} is also shown to hold for $\alpha < \sfrac 12$.

We note that very recently, by adapting the arguments in \cite{BV}, Luo and Titi~\cite{LuoTiti18} demonstrated the non-uniqueness of very weak solutions for \eqref{eq:NSE} in the parameter range $\alpha \in (1,\sfrac 54)$. When compared to~\cite{LuoTiti18} the weak solutions constructed in this paper have the additional property that their set of singular times has Hausdorff dimension strictly less than $1$. Together, the uniqueness  result~\cite{LionsJ2}, the non-uniqueness results of~\cite{LuoTiti18}, and of this work, confirm the well-posedness \emph{criticality} of the exponent $\alpha = \sfrac 54$, within the class of weak solutions defined in Definition~\ref{def:weak}.

We give the proof of Theorem~\ref{thm:main} for general values of $\alpha < \sfrac 54$. Theorem~\ref{thm:RT} follows by restricting to $\alpha=1$.

\section{Outline of the proof}
\label{sec:outline}

The proof of Theorem \ref{thm:main} proceeds via a convex integration scheme based on the scheme introduced  in \cite{BV}, which is itself built on a long line of work initiated by De Lellis and Sz\'ekelyhidi Jr.~\cite{DeLellisSzekelyhidi13}, culminating in  the eventual resolution of Onsager's conjecture by Isett \cite{Isett16} (cf.\ \cite{DeLellisSzekelyhidi12a, BDLISZ15, Buckmaster15,BDLSZ16,DaneriSzekelyhidi17,BDLSV17,Isett17}). Such a scheme is used to inductively define a sequence of approximate solutions, converging to a weak solution of \eqref{eq:NSE}. The principal new idea of this paper is to create \emph{good regions} in time where the approximate solutions are strong solutions to  \eqref{eq:NSE} and are untouched in later inductive steps. This is achieved by employing the method of gluing introduced by Isett~\cite{Isett16} (cf.\ \cite{BDLSV17}). Taking the countable union of the good regions over each inductive step, one forms a fractal set, whose complement has Hausdorff dimension strictly less than $1$. This is explained in detail in Section \ref{sec:inductive} and \ref{ss:gluing_stage} below. The concept of good regions is partially inspired by similar concepts introduced in \cite{Buckmaster15} (cf.\ \cite{BDLSZ16}). An additional novelty of the present work is the introduction of \emph{intermittent jets} which replace the intermittent Beltrami flows of \cite{BV} as the fundamental building blocks on which the convex integration scheme in based (see Section \ref{ss:convex_integration} and \ref{s:intermittent}).

\subsection{Inductive estimates and main proposition}
\label{sec:inductive}

For every index $q \in \mathbb N$ we will construct a solution $(v_q,\RR_q)$ to the Navier-Stokes-Reynolds system
\begin{subequations}
\label{eq:NSE:Reynolds}
\begin{align}
\partial_t v_q + \div (v_q\otimes v_q)+\nabla p_q + (- \Delta)^\alpha v_q &= \div \RR_q\\
\div v_q &= 0\,,
\end{align}
\end{subequations}
where $\RR_q$ is a trace-free symmetric matrix. The pressure $p_q$ is normalized to have zero mean on $\T^3$ and is explicitly given by the formula 
\begin{align}
p_q = \div \Delta^{-1} \div (\RR_q- v_q\otimes v_q)\,.
\label{eq:pq:def}
\end{align}
Here we use the convention that for a $2$ tensor $S = (S^{ij})_{i,j=1}^{3}$ the divergence contracts on the second component, i.e.\ $(\div S)^i = \partial_j S^{ij}$. The summation convention on repeated indices is used throughout.
 
Fix a sufficiently large integer $b = b(\alpha)>0$.\footnote{For instance, it is sufficient to take $ b (5-4\alpha) \geq 1000$, which verifies \eqref{eq:new:beta:b}.} Depending on this choice of $b$, fix a sufficiently small parameter $\beta = \beta(\alpha,b)>0$.\footnote{For instance, it is sufficient to require that $200 \beta b^2 \leq 5-4\alpha$. This verifies both~\eqref{eq:new:beta:b} and~\eqref{eq:main:constraint:b}.} In particular, $\beta b \ll 1$.

The size of the Reynolds stress $\RR_q$ will be measured in terms of a size parameter
\begin{align}
\delta_q = \lambda_1^{3\beta} \lambda_q^{-2\beta}
\label{eq:delta:q:def}
\end{align}
where $\lambda_q$ is a frequency parameter defined by 
\begin{align}\notag
\lambda_q = a^{(b^q)}
\end{align}
where  $a \gg 1$ is an  large real number to be chosen later. Note that $\delta_1 = \lambda_1^\beta = a^{\beta b}$ is large if $a$ is sufficiently large. 

For every $q\geq 0$ we assume that $\RR_q$ obeys the estimates  
\begin{subequations}
\label{eq:Rq:inductive}
\begin{align}
\norm{\RR_q}_{L^1(\T^3)} &\leq \lambda_{q}^{-\eps_R} \delta_{q+1} 
\label{eq:Rq:L1}\\
\norm{ \RR_q}_{H^3(\T^3)}  & \leq  \lambda_q^{7} 
\label{eq:Rq:H3:x}
\end{align}
\end{subequations}
for some $\eps_R>0$ to be chosen later, which depends only on the values of $\alpha$, $\beta$, and $b$.
For the approximate velocity field $v_q$, we assume that it obeys the estimates  
\begin{subequations}
\begin{align}
\norm{v_q}_{L^2(\T^3)} &\leq 2 \delta_0^{\sfrac 12} - \delta_q^{\sfrac 12}
\label{eq:v:L2:x} 
\\
\norm{v_q}_{H^3(\T^3)} &\leq  \lambda_q^{4}  \, .
\label{eq:v:H3:x} 
\end{align}
\end{subequations}
These inductive estimates will ensure that the approximate solutions $v_q$ converge strongly in $C^0(0,T;L^2)$ to a weak solution $v$ of the Navier-Stokes equation \eqref{eq:NSE}.

Consider $T>0$ and fix the parameter sequences $\{ \tau_q \}_{q\geq 0}$ and $\{ \vartheta_q \}_{q\geq 1}$  defined in \eqref{eq:theta:q:def} and \eqref{eq:tau:q:def} below, which obey the heuristic bounds
 \begin{align}
 \vartheta_{q+1} \ll \tau_q \ll \vartheta_q \ll 1 \, .
 \label{eq:tau:q:mu:q:assume:0}
 \end{align}
In particular, for $q\geq 1$ we make the choices 
\begin{align}
\vartheta_{q} =  \lambda_{q-1}^{-7} \delta_{q}^{\sfrac 12}
\label{eq:theta:q:def}
\end{align}
and 
\begin{align}
\tau_{q} = \vartheta_{q} \lambda_{q-1}^{- \sfrac{\eps_R}{4}} = \lambda_{q-1}^{-7 - \sfrac{\eps_R}{4}} \delta_q^{\sfrac 12}\, .
\label{eq:tau:q:def}
\end{align}
For the special case $q=0$ we set 
\[\tau_0:=\frac{T}{15}\,.\]
For $\vartheta_0$ we do not need to assign a value.

In order to ensure that the singular set of times has Hausdorff dimension strictly less than $1$, at every $q\geq 0$ we split the interval $[0,T]$ into a closed {\em good set} $\GG^{(q)}$ and an open {\em bad set} 
\[
\BB^{(q)} = [0,T] \setminus \GG^{(q)}
\] 
which obey the following properties:
\begin{enumerate}[label=(\roman*)]
 \item \label{eq:cond:i} $\GG^{(0)} = [0,\sfrac T3] \cup [\sfrac{2T}{3} ,T]$.
 \item \label{eq:cond:ii} $\GG^{(q-1)} \subset \GG^{(q)}$ for every $q\geq 1$.
 \item \label{eq:cond:iii} $\BB^{(q)}$ is a finite union of disjoint open intervals of length $5 \tau_q$.\footnote{Observe that this condition is consistent with property \ref{eq:cond:i} and the definition $\tau_0=\sfrac{T}{15}$.}
 \item \label{eq:cond:iv} For $q\geq 1$, the bad sets have measures which obey
 \begin{align}
|\BB^{(q)}| \leq |\BB^{(q-1)}| \frac{10 \tau_q}{\vartheta_q}\, .
\label{eq:mead:2}
\end{align}
\item \label{eq:cond:v} The velocity fields obey
\begin{align}
\mbox{if} \quad t \in \GG^{(q')} \quad \mbox{for some} \quad q' < q, \quad \mbox{then} \quad v_q(t) = v_{q'}(t)\,.
\label{eq:vq:unchanged}
\end{align}
\item \label{eq:cond:vi} The residual Reynolds stress obeys
 \begin{align}
\RR_q(t) = 0 \quad \mbox{for all} \quad t\in [0,T] \quad \mbox{such that} \quad {\rm dist}(t, \GG^{(q)}) \leq \tau_q\,.
\label{eq:Rq:time:support}
\end{align} 
\end{enumerate}

Due to \eqref{eq:Rq:time:support} and the parabolic regularization of the Navier-Stokes equation (cf.~\eqref{eq:v:H:3:N} below) we have that $v_q$ is a $C^\infty$ smooth exact solution of the Navier-Stokes equation on $\GG^{(q)}$.
In addition, \eqref{eq:vq:unchanged} implies that $v = v_q$ on $\GG^{(q)}\setminus \{0\}$, and thus the limiting solution $v$ is $C^\infty$ smooth on $(\GG^{(q)}\setminus \{0\}) \times \T^3$. This justifies that the {\em singular set of times} $\Sigma_T$ obeys
\begin{align}
\Sigma_T \subset \bigcap_{q\geq 0} \BB^{(q)}.
\label{eq:mead:1}
\end{align}
It thus follows from \eqref{eq:mead:2} and the definitions of $\tau_q$ and $\vartheta_q$ in \eqref{eq:theta:q:def} and \eqref{eq:tau:q:def} that 
\begin{align}
|\BB^{(q)}| \leq |\BB^{(0)}| \prod_{q'=1}^{q}  \frac{10 \tau_{q'}}{\vartheta_{q'}} \leq 10^q T \prod_{q'=0}^{q-1}    \lambda_{q'}^{-\sfrac{\eps_R}{4}}    \leq T 10^q  a^{- \frac{\eps_R (b^q - 1)}{4 (b-1)}} \leq T 10^q  \lambda_q^{- \frac{\eps_R}{8 (b-1)}} \, .
\label{eq:mead:22}
\end{align}
Here we have also used the definition of $\lambda_q$, and the fact that $b>2$. 
To estimate the box-counting (Minkowski) dimension of $\Sigma_T$, we note that for every $q\geq0$, the set $\Sigma_T$ is covered by $\BB^{(q)}$, which itself consists of disjoint intervals of length $5 \tau_q$. Due to \eqref{eq:mead:22}, the number of such intervals is at most
\begin{align}\notag
T 10^q  \lambda_q^{-\frac{\eps_R}{8 (b-1)}} (5\tau_q)^{-1}
\end{align}
By \eqref{eq:mead:1}, and the super-exponential growth of $\lambda_q$, we conclude that 
\begin{align}
{\rm dim}_{\rm box} (\Sigma_T) 
&\leq  \lim_{q \to \infty} \frac{\log (T) + q \log(10) - \frac{\eps_R }{8 (b-1)} \log\left( \lambda_{q} \right) - \log(5\tau_q)}{- \log \left(5 \tau_q \right)} 
\notag\\
&= 1 - \lim_{q\to \infty}\frac{\frac{\eps_R b}{8 (b-1)} \log\left( \lambda_{q-1} \right) }{- \log \left(\tau_q \right)} \notag\\
&= 1 - \frac{\eps_R b}{8 (b-1)(7 + \sfrac{\eps_R}{4} + \beta b)} < 1 - \frac{\eps_R}{64}   < 1  \,.\label{e:box_count}
\end{align} 
This implies that $\Sigma_T$ also has  box-counting dimension strictly less than $1$.

\begin{proposition}[Main Iteration Proposition]
\label{prop:main}
There exists  a sufficiently small parameter $\eps_R = \eps_R(\alpha, b,\beta) \in (0,1)$ and a sufficiently large parameter $a_0 = a_0(\alpha, b, \beta, \eps_R) \geq 1$
such that for any $a \geq a_0$ satisfying the technical condition \eqref{e:integer_constraint}, the following holds: Let $(v_q,\mathring R_q)$ be a pair solving the Navier-Stokes-Reynolds system \eqref{eq:NSE:Reynolds} in $\mathbb T^3\times [0,T]$ satisfying the inductive estimates \eqref{eq:Rq:L1}--\eqref{eq:v:H3:x}, and a corresponding set $\GG^{(q)}$ with the properties~\ref{eq:cond:i}--\ref{eq:cond:vi} listed above. Then there exists a second pair $(v_{q+1},\mathring R_{q+1})$  solving   \eqref{eq:NSE:Reynolds} and a set $\GG^{(q+1)}$ which satisfy \eqref{eq:Rq:L1}--\eqref{eq:v:H3:x} and \ref{eq:cond:i}--\ref{eq:cond:vi} with $q$ replaced by $q+1$. In addition we have that
\begin{equation}\label{e:interative_v}
\norm{v_{q+1}-v_q}_{L^{2}}\leq  \delta_{q+1}^{\sfrac12}\,.
\end{equation}
\end{proposition}

\subsection{Gluing stage}\label{ss:gluing_stage}
The first stage of proving Proposition~\ref{prop:main} is to start from the approximate solution $(v_q,\RR_q)$ which obeys \eqref{eq:Rq:L1}--\eqref{eq:v:H3:x} and \eqref{eq:Rq:time:support}, and construct a new \emph{glued} pair $(\overline v_q, {\mathring {\bar R}}_{q})$, which solves \eqref{eq:NSE:Reynolds}, obeys bounds which are the same as \eqref{eq:Rq:L1}--\eqref{eq:v:H3:x} up to a factor of $2$, but which has the advantage that ${\mathring {\bar R}}_{q} \equiv 0$ on $\T^3 \times \BB^{(q+1)}$.  

Specifically, the  new velocity field $\overline v_q$ is defined as
\[
\overline v_q (x,t)= \sum_{i} \eta_i(t) v_i(x,t)\,,
\]
where the $\eta_i$ are certain cutoff functions with support in the intervals $[t_i, t_{i+1}+\tau_{q+1}]$ (with $t_i = \vartheta_{q+1} i$) that form a partition of unity (see \eqref{e:chi:partition} below), and the $v_i$ are exact solutions of the Navier-Stokes equation with initial datum given by $v_i(t_{i-1})=v_q(t_{i-1})$. Due to parabolic regularization, these exact solutions $v_i$ are  $ C^\infty$ smooth in space and time on the support of $\eta_i$, so that $\overline v_q$ inherits this $C^\infty$ regularity.  This is in contrast to $(v_q,\RR_q)$, which is only assumed to be $ H^3$ smooth. Trivially, in the regions where a cut-off $\eta_i$ is identically $1$, $\overline v_q$ is an exact solution to  \eqref{eq:NSE}.

Observe that property \eqref{eq:Rq:time:support} ensures that $v_q$ is already an exact solution of \eqref{eq:NSE} on a large subset of $[0,T]$, namely the $\tau_q$ neighborhood of $\GG^{(q)}$. In particular if $t_{i-1}$ and $t_i$ both lie within this neighborhood, then by uniqueness  of the Navier-Stokes equation in $C^0_t H^3_x$, we have $v_i=v_{i+1}=v_q$ on the overlapping region $\supp \eta_i\eta_{i+1}$. 
Hence $\overline v_q = v_q$  is an exact solution here. In order to single out overlapping regions where $\overline v_q$ is not necessarily an exact solution of \eqref{eq:NSE} we introduce the index set
\begin{align}
{\mathcal C} = \left\{ i\in \{ 1, \ldots , n_{q+1} \}  \colon    \mbox{there exists } t \in [t_{i-1},t_{i+1}+\tau_{q+1}] \mbox{ with } \RR_q(t) \neq 0 \right\} \, .
\label{eq:C:def}
\end{align}
We then define
\begin{align}
\BB^{(q+1)} = \bigcup_{i \in {\mathcal C}~\mbox{or}~i-1 \in {\mathcal C}} (t_i - 2 \tau_{q+1}, t_{i} + 3 \tau_{q+1}) \, .
\label{eq:B:q+1:def}
\end{align}
By the discussion above, it will follow that $\overline v_q$ is an exact solution on the complement of $\BB^{(q+1)}$, namely $\GG^{(q+1)}$. We prove in Section~\ref{sec:gluing} below that the above defined good and bad sets at level $q+1$ obey the postulated properties \ref{eq:cond:i}--\ref{eq:cond:iv}.

In Section~\ref{sec:gluing} we prove the following proposition:
\begin{proposition}\label{prop:bar:vq:Rq}
There exists a solution $(\overline v_q, {\mathring {\bar R}}_{q})$ of \eqref{eq:NSE:Reynolds}, such that 
\begin{align}
\overline v_q \equiv v_q \qquad \mbox{on} \qquad   \T^3 \times \GG^{(q)} \,,
\label{eq:vq:same}
\end{align} 
and moreover the velocity field $\overline v_q$ satisfies:
\begin{subequations}
\begin{align}
 \norm{\bar v_q}_{L^2} &\leq 2 \delta_0^{\sfrac 12} -\delta_q^{\sfrac 12} 
 \label{eq:vq:global:1} \\
 \norm{\bar v_q}_{H^3} &\leq 2 \lambda_q^{4}  
 \label{eq:vq:global:2} \\
 \norm{\bar v_q - v_{q}}_{L^2} &\leq  \vartheta_{q+1}\lambda_q^{6}   \leq \frac 14 \delta_{q+1}^{\sfrac 12}
 \label{eq:vq:vell} \\ 
\| \partial_t^M D^N \bar v_q\|_{L^\infty(\sfrac{T}{3},\sfrac{2T}{3} ; H^3)} &\les  \tau_{q+1}^{-M} \vartheta_{q+1}^{-\frac{N}{2\alpha}} \lambda_q^{4} \les  \tau_{q+1}^{-M-N} \lambda_q^{4}
\label{eq:vq:1}
\end{align}
\end{subequations}
and the stress tensor $\mathring{\overline R_q}$ satisfies:
\begin{subequations}
\begin{align}
\mathring{\overline{R}}_q(t) &= 0  \quad \mbox{for all} \quad t \in [0,T] \quad \mbox{such that} \quad {\rm dist}(t, \GG^{(q+1)}) \leq 2 \tau_{q+1} \label{eq:bar:Rq:supp}\\
\norm{\mathring{\overline R_q}}_{ L^1} &\leq \tau_{q+1}^{-1}\vartheta_{q+1}\lambda_q^{-\sfrac{\eps_R}{2}} \delta_{q+1}  \leq \lambda_q^{-\sfrac{\eps_R}{4}} \delta_{q+1}  \label{eq:Rq:1}\\
\norm{\partial_t^M D^N\mathring{\overline R_q}}_{H^3} &\lesssim \tau_{q+1}^{-M-1}\vartheta_{q+1}^{-\frac{N}{2\alpha}}\lambda_q^{4} \les \tau_{q+1}^{-M-N-1} \lambda_q^{4} \label{eq:Rq:2}
\end{align}
\end{subequations}
for all $M,N \geq 0$.
\end{proposition}

\subsection{Convex integration stage}\label{ss:convex_integration}
In this step we start from the pair $(\overline v_q, {\mathring {\bar R}}_{q})$, and construct a new pair $(v_{q+1}, \RR_{q+1})$ with $\RR_{q+1}$  obeying \eqref{eq:Rq:time:support} at level $q+1$, and which obeys the  bounds \eqref{eq:Rq:L1}--\eqref{eq:v:H3:x} at level $q+1$. 

The perturbation $w_{q+1}:=v_{q+1}-\overline v_q$ will be constructed to correct  for ${\mathring {\bar R}}_{q}$. Moreover, $w_{q+1}$ will be designed to have support outside a $\tau_{q+1}$ neighborhood of $\mathcal G^{(q+1)}$ -- this ensures properties \ref{eq:cond:v} and \ref{eq:cond:vi} in Section \ref{sec:inductive} will be satisfied.  As in \cite{BV}, the perturbation $w_{q+1}$ will consist of three parts: the principal part $w_{q+1}^{(p)}$, the divergence corrector $w_{q+1}^{(c)}$, and the temporal corrector $w_{q+1}^{(c)}$. 

The principal part $w_{q+1}^{(p)}$ will be constructed as a sum of \emph{intermittent jets} $W_{(\xi)}$ (defined in \eqref{eq:jet}, Section \ref{s:intermittent}). The use of intermittent jets replaces the use of \emph{intermittent Beltrami waves} in \cite{BV}. The principal difference of intermittent jets from  intermittent Beltrami waves is that their definition is in physical space rather than frequency space. Consequently,  intermittent jets are comparatively simpler to define and they can be designed to have disjoint support, mimicking the advantageous support properties of Mikado flows, as introduced in~\cite{DaneriSzekelyhidi17}. We note that the intermittent variants of the  $d-1$ dimensional Mikado flows found in \cite{ModenaSZ17,Luo18}, lying in $d$-dimensional space, are insufficiently intermittent to be used as building block for a 3-D Navier-Stokes convex integration scheme.\footnote{For Navier-Stokes in dimensions greater than $3$, they are however applicable, as demonstrated in \cite{Luo18}.} Intermittent jets are inherently $3$-dimensional (in space), with the trade-off that they are time dependent. We note in passing that utilizing intermittent jets, it is likely that the convex integration results \cite{ModenaSZ17,ModenaSZ18} on the transport equation may be improved. 

In the definition of $w_{q+1}^{(p)}$, the intermittent jets $W_{(\xi)}$ will be weighted by functions $a_{(\xi)}$:
\[
w_{q+1}^{(p)} = \sum_\xi a_{(\xi)}W_{(\xi)} \,,
\]
where $a_{(\xi)}$ are constructed such that
\begin{equation}\label{e:cancellation}
\div \left(w_{q+1}^{(p)}\otimes w_{q+1}^{(p)}+{\mathring {\bar R}}_{q}\right)\sim  \frac{1}{\mu} \partial_t\mathbb P_H\mathbb P_{\neq 0} \left(\sum_\xi a_{(\xi)}^2 |W_{(\xi)}|^2 \xi \right)+  \mbox{(pressure gradient)} + \mbox{(high frequency error)}
\end{equation}
for some large parameter $\mu$. As is typical in convex integration schemes, the high frequency error can be ignored since its contribution to $\RR_{q+1}$ can be bounded using the gain associated with solving the divergence equation. The temporal corrector $w_{q+1}^{(t)}$ is then defined to be 
\[w_{q+1}^{(t)}:= -  \frac{1}{\mu} \mathbb P_H\mathbb P_{\neq 0} \left(\sum_\xi a_{(\xi)}^2 |W_{(\xi)}|^2 \xi\right)\,,\]
where ${\mathbb P}_{H}$ is the Helmholtz projection, and ${\mathbb P}_{\neq 0}$ is the projection onto functions with mean zero. That is, $\Proj_H f = f -\nabla (\Delta^{-1} \div f)$ and $\Proj_{\neq 0} f = f - \fint_{\T^3} f$. Hence we have
\[\div \left(w_{q+1}^{(p)}\otimes w_{q+1}^{(p)}+{\mathring {\bar R}}_{q}\right)+\partial_t w_{q+1}^{(t)}\sim ~ \mbox{(pressure gradient) }+\mbox{ (high frequency error)}\,.
\]
Finally, the divergence corrector $w_{q+1}^{(c)}$ is designed such that $\div \left(w_{q+1}^{(p)}+w_{q+1}^{(c)}\right)\equiv 0$, and hence the perturbation
\[w_{q+1}^{(c)}:=w_{q+1}^{(p)}+w_{q+1}^{(c)}+w_{q+1}^{(t)}\]
is divergence free. 

The intermittent jets will be defined to have support confined to $\sim (\ell_{\perp}\lambda_{q+1})^3$ many cylinders of diameter $\sim \frac{1}{\lambda_{q+1}}$ and length $\sim \frac{\ell_{\|}}{\ell_{\perp}\lambda_{q+1}}$. In particular, the support of $w_{q+1}^{(p)}$ has measure $\sim \ell_{\|}\ell_{\perp}^2$. Using the heuristic that $\norm{w_{q+1}^{(p)}}_{L^2}$ should be roughly the size $\norm{\mathring {\bar R}_q}_{L^1}^{\sfrac 12 }$, by the $L^p$ de-correlation result in Lemma~\ref{lem:Lp:independence} below, one would expect an $L^p$ estimate on $w_{q+1}^{(p)}$ of the form
\begin{equation}\label{e:w_heurs}
\norm{w_{q+1}^{(p)}}_{L^p}\sim \delta_{q+1}^{\sfrac12}\ell_{\perp}^{\sfrac{2}{p}-1}\ell_{\|}^{\sfrac{1}{p}-\sfrac12}\,.
\end{equation}
Indeed we will prove  estimate \eqref{e:w_heurs} for $p=2$ and prove a slightly weaker estimate for $1<p<2$ (see Proposition \ref{prop:perturbation}). Utilizing \eqref{e:w_heurs}, one may heuristically estimate the contribution of $(-\Delta)^{\alpha} w_{q+1}^{(p)}$ to the new Reynolds stress $\RR_{q+1}$
\begin{align*}
\norm{\abs{\nabla}^{-1}(-\Delta)^{\alpha} (w_{q+1}^{(p)})}_{L^1} \sim  \norm{w_{q+1}^{(p)}}_{W^{2\alpha-1,p}}\sim  \delta_{q+1}^{\sfrac12}\ell_{\perp}^{\sfrac{2}{p}-1}\ell_{\|}^{\sfrac{1}{p}-\sfrac12}\lambda_{q+1}^{2\alpha-1}\,,
\end{align*}
with $p>1$ arbitrarily close to $1$.
Here we see the necessity of the 3-dimensionality of the intermittent jets. 

In order to ensure that an identity of the form \eqref{e:cancellation} holds, the cylinder supports of the intermittent jets will be shifting at a speed $\ell_{\perp}\lambda_{q+1}\mu$. Heuristically, one would then expect that in order to ensure that the contribution of $\partial_t w_{q+1}^{(p)}$ to $\RR_{q+1}$ is small, one would need to impose an upper bound on the choice of $\mu$. One then needs to choose $\mu$ carefully in order to balance different contributions to the Reynolds stress error. Explicitly, we will define the parameters $\mu$, $\ell_{\perp}$ and $\ell_{\|}$ by
\begin{align} \label{e:param_def}
\mu=\frac{\lambda_{q+1}^{2\alpha -1}\ell_{\|}}{\ell_{\perp}}, \qquad \ell_{\perp}:=\lambda_{q+1}^{-\frac{20\alpha-1}{24}}\qquad \mbox{and}\qquad\ell_{\|}:=\lambda_{q+1}^{-\frac{20\alpha-13}{12}}\,.
\end{align}
With these choices, we have
\[\ell_{\|}^{-1}\ll \ell_{\perp}^{-1}\ll \lambda_{q+1}\,\]
since $\alpha < \sfrac 54$. For technical reasons, we will require that $\lambda_{q+1}\ell_{\perp}\in\mathbb N$. This may be achieved by assuming that 
\begin{equation}\label{e:integer_constraint}
a^{\frac{25-20\alpha}{24}}\in\mathbb N\,,
\end{equation}
where we recall that we have previously assumed that $b\in\mathbb N$.
  
\subsection{Proof of Theorem~\ref{thm:main}}

Let $u^{(1)}$ and $u^{(2)}$ be two zero mean solutions of the Navier-Stokes equations (with different, zero-mean  initial data),  as in the statement of the theorem. Also, let $b$, $\beta$, $\epsilon_R$, and $a_0$ be as in Proposition~\ref{prop:main}.
Let $\eta \colon [0,T] \to [0,1]$ be a smooth cutoff function such that $\eta=1$ on $[0,\sfrac{2T}{5}]$ and $\eta=0$ on $[\sfrac{3T}{5},T]$.

Define
\begin{align}\notag
v_0 (x,t) = \eta(t) u^{(1)}(x,t) + (1-\eta(t)) u^{(2)}(x,t)\,.
\end{align}
and 
\begin{align}\label{e:R_0_def}
\RR_0 =  \partial_t\eta \; \RSZ(u^{(1)}-u^{(2)})-    \eta(1-\eta)(u^{(1)}-u^{(2)})\mathring{\otimes} (u^{(1)}-u^{(2)}) \, ,
\end{align}
where $a \mathring \otimes b$ denotes the traceless part of the tensor $a\otimes b$, and $\RSZ$ is a standard  inverse divergence operator acting on vector fields $v$ which have zero mean on $\T^3$ as
\begin{align}
(\RSZ v)^{k\ell} = (\partial_k \Delta^{-1} v^{\ell} + \partial_\ell \Delta^{-1} v^k)  - \frac{1}{2} \left( \delta_{k \ell} + \partial_k \partial_\ell \Delta^{-1}\right)\div \Delta^{-1} v  
\label{eq:RSZ}
\end{align} 
for $k,\ell \in \{1,2,3\}$. The above inverse divergence operator has the property that  $\RSZ v(x)$ is a symmetric trace-free matrix for each $x \in \T^3$, and $\RSZ$ is an right inverse of the $\div$ operator, i.e.\ $\div (\RSZ v) = v$. When $v$ does not obey $\int_{\T^3} v dx = 0$, we overload notation and denote $\RSZ v := \RSZ( v -\int_{\T^3} v dx)$. Note that that  $\nabla \RSZ$ is a Calder\'on-Zygmund operator, and  that $\RSZ$ obeys the same elliptic regularity estimates as $|\nabla|^{-1}$.

Observe that the pair $(v_0, \RR_0 )$ obeys the Navier-Stokes-Reynolds system \eqref{eq:NSE:Reynolds}, for a suitable $0$-mean pressure scalar $p_0$ which may be computed by solving a Poisson equation. Moreover, let $a_0$, $\beta$ and  $b$ be as in Proposition \ref{prop:main}.  Then choosing $a\geq a_0$ sufficiently large, the pair $(v_0, \RR_0 )$ satisfies \eqref{eq:Rq:L1}-\eqref{eq:v:H3:x}. From the definition \eqref{e:R_0_def}, it follows that $\RR_0$ is supported on the interval $[\sfrac{2T}{5},\sfrac{3T}{5}]$. Since by definition $\GG^{(0)} = [0,\sfrac T3] \cup [\sfrac{2T}{3} ,T]$ and $\tau_0=\sfrac{T}{15}$, we obtain property \eqref{eq:Rq:time:support}.

For $q \geq 1$ we inductively apply Proposition~\ref{prop:main}. 
The bound \eqref{eq:v:H3:x} and \eqref{e:interative_v} and interpolation yields
\begin{align*}
\sum_{q=0}^{\infty} \norm{v_{q+1}-v_q}_{\dot{H}^{\beta'}} 
&\les \sum_{q=0}^{\infty} \norm{v_{q+1}-v_q}_{L^{2}}^{1-\sfrac{\beta'}3}(\norm{v_{q+1}}_{\dot{H}^3} + \norm{v_q}_{\dot{H}^{3}})^{\sfrac {\beta'} 3} \\
&\les  \; \sum_{q=0}^{\infty} \lambda_{q+1}^{- \beta \frac{3-\beta'}{6}}\lambda_{q+1}^{\frac{4 \beta'}3} \\
&\les 1 \, ,
\end{align*}
for $0\leq \beta'<\frac{3\beta}{8+\beta}$, where the implicit constant is  universal (independent of $a$). Hence there exists 
\begin{align*}
 v:= \lim_{q\to \infty} v_q   \in H^{\beta'}\,. 
\end{align*}
Since $\norm{\mathring R_q}_{L^1} \to 0$ as $q \to \infty$, and since $v_q \to v$ also in $L^\infty_t L^{2+\beta'''}$ for some $\beta'''>0$, it is straightforward to show that $v$ is a weak solution of the Navier-Stokes equation. Moreover, as a consequence of properties \ref{eq:cond:i} and \ref{eq:cond:v} from Section \ref{sec:inductive} and the definition of $v_0$ we have
\[v \equiv u^{(1)} \quad \mbox{on} \quad [0,\sfrac{T}{3}], \qquad \mbox{and} \qquad  v \equiv u^{(2)} \quad \mbox{on} \quad [\sfrac{2T}{3},T]\, .\]
The argument leading to \eqref{e:box_count} implies that the singular set of times of $v$ has  box-counting dimension (and hence Hausdorff dimension)  less than $\sfrac{\eps_R}{64}$. Finally, the claimed $C^0_tW^{1,1+\beta''}_x$ regularity on $v$, for some $\beta''>0$, follows from the maximal regularity of the heat equation (fractional heat equation if $\alpha>1$), once we note that $\norm{\Proj_H(v\otimes v)}_{L^{1+\beta''}}\les \norm{v}_{H^{\beta'}}^2$ if $\beta''$ is chosen suitably small. The theorem then holds with $\bar \beta = \min\{ \beta'', \beta', \sfrac{\eps_R}{64} \} > 0$.

\section{Gluing step}
\label{sec:gluing}
  
\subsection{Local in time estimates}
\label{sec:local:in:time}
It is well-known Navier-Stokes equations are locally (in time) well-posed in $H^3$, which is a scaling subcritical space. Moreover, away from the initial time, parabolic regularization takes place. We summarize these facts, in version that is suitable for the applications in this paper.
\begin{proposition}
\label{prop:local:existence}
Let $v_0 = v|_{t=t_0} \in H^{3}(\T^3)$ have zero mean on $\T^3$, and consider the Cauchy problem for \eqref{eq:NSE} with this initial condition. 
There exists a universal constant $c\in (0,1]$ such that if $t_1 > t_0$ is such that
\begin{align}
0< t_1-t_0 \leq \frac{c}{\norm{v_0}_{H^3}} \, ,
\label{eq:v:diff:IC}
\end{align}
then there exists a unique strong solution to \eqref{eq:NSE} on $[t_0,t_1)$, and it obeys the estimates 
\begin{subequations}
\begin{align}
\sup_{t\in [t_0,t_1]} \norm{v(t)}_{L^2}^2 +2 \int_{t_0}^{t_1} \norm{v(t)}_{\dot{H}^\alpha}^2 dt &\leq \norm{v_0}_{L^2}^2. 
\label{eq:v:L2:1}\\
\sup_{t\in [t_0,t_1]} \norm{v(t)}_{H^3} &\leq 2 \norm{v_0}_{H^3}.
\label{eq:v:H3:1}
\end{align}
\end{subequations}
Moreover, 
assuming that 
\begin{align}
 0< t_1-t_0 \leq  \frac{ c}{ \norm{v_0}_{H^3}  (1+ \norm{v_0}_{L^2})^{\frac{1}{2\alpha-1}}}   \, ,
\label{eq:t:weird}
\end{align}
we have that 
\begin{align}
\sup_{t\in (t_0,t_1]}  |t-t_0|^{\frac{N}{2\alpha}+M} \norm{\partial_t^M D^N v(t)}_{H^3}   \les  \norm{v_0}_{H^3}\,, 
\label{eq:v:H:3:N}
\end{align}
for any $N\geq 0$ and $M \in\{0,1\}$. The implicit constant may depend on $\alpha, N, M$.
\end{proposition}
\begin{proof}[Proof of Proposition~\ref{prop:local:existence}]
The energy inequality gives a global in time control on $\norm{v(t)}_{L^2}$:
\[
\frac 12 \frac{d}{dt} \norm{v}_{L^2}^2 \leq - \norm{v}_{\dot{H}^\alpha}^2 .
\]
From the Gagliardo-Nirenberg-Sobolev and the Poincar\'e inequalities, and using that $\nabla \cdot v = 0$ we obtain
\begin{align}
\frac 12 \frac{d}{dt} \norm{v}_{\dot{H}^3}^2  + \norm{v}_{\dot{H}^{3+\alpha}}^2
\les  \norm{v}_{\dot{H}^3}^2 \norm{\nabla v}_{L^\infty} + \norm{v}_{\dot{H}^3} \norm{\Delta v}_{L^4}^2 
\les \norm{v}_{\dot{H}^3}^3\notag
\end{align}
which gives the bound \eqref{eq:v:H3:1} for a time interval $[t_0,t_1]$ with $t_1$ that obeys \eqref{eq:v:diff:IC}. 
The bound \eqref{eq:v:H3:1} is subcritical, in the sense that an $L^\infty_t H^3_x$ a priori estimate is sufficient to establish the  uniqueness of the solution. The higher regularity claimed in \eqref{eq:v:H:3:N} follows from the mild form of the solution 
\begin{align}
v(t) = e^{-(t-t_0) (-\Delta)^\alpha} v_0 + \int_{t_0}^t e^{-(t-s) (-\Delta)^\alpha} {\mathbb P}_H \div (v(s) \otimes v(s)) ds,
\label{eq:v:mild}
\end{align} 
and properties of the fractional heat equation which may be derived from Plancherel. 

Let us first focus on the case $M=0$. For $\alpha=1$, estimate \eqref{eq:v:H:3:N} is well-known, and follows from the instantaneous gain of analyticity of the solution~\cite{FoiasTemam89}, or a small modification of the below argument. For $\alpha >1 $ we briefly sketch the argument. Using Gallilean invariance, let us only consider the case $t_0=0$. From the inequality
\begin{align}
\norm{uv}_{H^3} \les \norm{u}_{H^3} \norm{v}_{L^\infty} + \norm{u}_{L^\infty} \norm{v}_{H^3} \les \norm{u}_{H^3} \norm{v}_{L^2}^{1/2} \norm{v}_{H^3}^{1/2} + \norm{u}_{L^2}^{1/2} \norm{u}_{H^3}^{1/2} \norm{v}_{H^3},
\label{eq:H3:algebra}
\end{align}
the formulation \eqref{eq:v:mild} and the boundedness of the Leray projector ${\mathbb P}_H$ on $L^2$, we obtain
\begin{align}
t^{\frac{1}{2\alpha}} \norm{D v(t)}_{H^3} 
&\leq t^{\frac{1}{2\alpha}} \norm{D e^{-t(-\Delta)^\alpha}}_{L^2\to L^2} \norm{v_0}_{H^3}   + t^{\frac{1}{2\alpha}} \int_{0}^t \norm{D^2 e^{-(t-s)(-\Delta)^\alpha}}_{L^2\to L^2} \norm{v(s)}_{H^3}^{\sfrac 32} \norm{v(s)}_{L^2}^{\sfrac 12} ds \notag\\
&\les \norm{v_0}_{H^3}+ t^{\frac{1}{2\alpha}}\norm{v_0}_{H^3}^{\sfrac 32} \norm{v_0}_{L^2}^{\sfrac 12} \int_{0}^t \frac{ds}{(t-s)^{\frac{1}{\alpha}}} \notag\\
&\les \norm{v_0}_{H^3} \left( 1 + t^{1 - \frac{1}{2\alpha}}\norm{v_0}_{H^3}^{\sfrac 12} \norm{v_0}_{L^2}^{\sfrac 12} \right) 
\notag\\
&\les \norm{v_0}_{H^3} \left( 1 + t^{ \frac{2\alpha - 1}{2\alpha}}\norm{v_0}_{H^3}^{\frac{2\alpha-1}{2\alpha}} \norm{v_0}_{L^2}^{\frac{1}{2\alpha}} \right) \notag
\end{align}
from which \eqref{eq:v:H:3:N} with $N=1$ and $M=0$ follows in view of \eqref{eq:t:weird}.
In order to treat the case $N\geq 2$ and $M=0$, we first note that for $1 \leq n\leq N-1$ by induction on $N$ we have
\begin{align}
\norm{D^{n} (v \otimes v)}_{H^3} 
&\les \sum_{j=0}^{n}  \norm{D^j v  \otimes D^{n-j} v }_{H^3} \notag\\
&\les \sum_{j=0}^{n}  \norm{D^j v }_{H^3} \norm{D^{n-j} v }_{H^3}^{\sfrac 12} \norm{D^{n-j} v}_{L^2}^{\sfrac 12} \notag\\
&\les \sum_{j=0}^{n-3}  \norm{D^j v }_{H^3} \norm{D^{n-j} v }_{H^3}^{\sfrac 12}  \norm{D^{n-j-3} v}_{H^3}^{\sfrac 12}   
+ \norm{D^{n-2} v }_{H^3} \norm{D^{2} v }_{H^3}^{\sfrac 12} \norm{v}_{H^3}^{\sfrac 13}  \norm{v}_{L^2}^{\sfrac 16}\notag\\
&\quad 
+ \norm{D^{n-1} v }_{H^3} \norm{D  v }_{H^3}^{\sfrac 12} \norm{v}_{H^3}^{\sfrac 16}  \norm{v}_{L^2}^{\sfrac 13}
+ \norm{D^n v }_{H^3} \norm{v }_{H^3}^{\sfrac 12} \norm{v}_{L^2}^{\sfrac 12}\notag\\
&\les \norm{v_0}_{H^3}^2 t^{-\frac{n}{2\alpha} + \frac{3}{4\alpha}} + \norm{v_0}_{H^3}^{\sfrac{11}{6}} t^{-\frac{n}{2\alpha} + \frac{1}{2\alpha}} \norm{v_0}_{L^2}^{\sfrac 16}\notag\\
&\quad + \norm{v_0}_{H^3}^{\sfrac{5}{3}} t^{-\frac{n}{2\alpha} + \frac{1}{4 \alpha}} \norm{v_0}_{L^2}^{\sfrac 13}+ \norm{v_0}_{H^3}^{\sfrac 32} t^{- \frac{n}{2\alpha}} \norm{v_0}_{L^2}^{\sfrac 12} \notag\\
&\les  \norm{v_0}_{H^3}^{\sfrac 32} t^{- \frac{n}{2\alpha}}  \left( \norm{v_0}_{H^3}^{\sfrac 12} t^{\frac{3}{4\alpha}} + \norm{v_0}_{L^2}^{\sfrac 12} \right).\notag
\end{align}
Using the above estimate with $n = N-1$ we obtain that
\begin{align}
t^{\frac{N}{2\alpha}} \norm{D^N v(t)}_{H^3}   
&\leq t^{\frac{N}{2\alpha}} \norm{D^N e^{-t(-\Delta)^\alpha}}_{L^2\to L^2} \norm{v_0}_{H^3} \notag\\
&  + t^{\frac{N}{2\alpha}} \int_{\sfrac{t}{2}}^t \norm{D^2 e^{-(t-s)(-\Delta)^\alpha}}_{L^2\to L^2}  \norm{D^{N-1}( v(s)\otimes v(s))}_{H^3}  \notag\\
&  + t^{\frac{N}{2\alpha}} \int_{0}^{\sfrac{t}{2}}  \norm{D^{N+1} e^{-(t-s)(-\Delta)^\alpha}}_{L^2\to L^2} \norm{v(s) \otimes v(s)}_{H^3} ds \notag\\
&\les \norm{v_0}_{H^3}+ t^{\frac{N}{2\alpha}}\norm{v_0}_{H^3}^{\sfrac 32} \int_{\sfrac{t}{2}}^t \frac{  \norm{v_0}_{H^3}^{\sfrac 12} s^{\frac{3}{4\alpha}} + \norm{v_0}_{L^2}^{\sfrac 12}  }{(t-s)^{\frac{1}{\alpha}} s^{\frac{N-1}{2\alpha}}}  ds + t^{\frac{N}{2\alpha}} \norm{v_0}_{H^3}^{\sfrac 32} \norm{v_0}_{L^2}^{\sfrac 12} \int_{0}^{\sfrac{t}{2}} \frac{ds}{(t-s)^{\frac{N+1}{2\alpha}}} \notag\\
&\les \norm{v_0}_{H^3} \left( 1 + t^{1 + \frac{5}{4\alpha}} \norm{v_0}_{H^3} + t^{1 - \frac{1}{2\alpha}}\norm{v_0}_{H^3}^{\sfrac 12} \norm{v_0}_{L^2}^{\sfrac 12} \right)  \notag
\end{align}
from which \eqref{eq:v:H:3:N} follows in view of \eqref{eq:t:weird}.

To obtain the desired bounds for $M = 1$, let us consider the case $N=0$ first. 
Using the equation, the already established bounds for $M=0$ and $N\geq 0$, the Gagliardo-Nirenberg-Sobolev inequalities, and the fact that the Leray projector is bounded on $L^2$, we have that
\begin{align}
t \norm{\partial_t v(t)}_{H^3} 
&\leq t \norm{(-\Delta)^\alpha v(t)}_{H^3} + t  \norm{\nabla v(t)}_{H^3} \norm{v(t)}_{H^3}^{\sfrac 12} \norm{v(t)}_{L^2}^{\sfrac 12} +t  \norm{v(t)}_{H^3}^{\sfrac 56} \norm{v(t)}_{L^2}^{\sfrac 16}   \norm{v(t)}_{H^3}   \notag\\
&\les \norm{v_0}_{H^3} + t^{1-\frac{1}{2\alpha}} \norm{v_0}_{H^3}^{\sfrac 32} \norm{v_0}_{L^2}^{\sfrac 12} + t \norm{v_0}_{H^3}^{\sfrac{11}{6}} \norm{v_0}_{L^2}^{\sfrac 16}\,,\notag
\end{align}
and the desired bound follows from the assumption \eqref{eq:t:weird}. The remaining cases $N\geq 1$ are treated in a similar manner, using the Leibniz rule. We omit these details.
\end{proof}

\subsection{Stability estimates}
\label{sec:stability}

In this section we estimate the difference between an approximate solution $v_q$ and an exact solution of the Navier-Stokes equation. Let $\RSZ$ be the inverse divergence operator defined in \eqref{eq:RSZ}. The main result is:  
\begin{proposition}
\label{prop:stability}
Fix $\alpha \in [1, \sfrac 54)$ and an integrability index $p_0 \in (1,\sfrac 54)$. Assuming the parameter $\delta_0$ is sufficiently large, depending on $p_0$, the following holds.

For $q\geq 0$, assume that $(v_q,\RR_q)$ is a $C^0_t H^3_x$ smooth solution of \eqref{eq:NSE:Reynolds} which obeys the estimates \eqref{eq:Rq:L1}--\eqref{eq:v:H3:x}. Let $t_0 \in [0,T]$ and define
\begin{align}\notag
v_0 := v_q\vert_{t=t_0}.
\end{align}
Assume that $t_1>t_0$ is such that $[t_0,t_1] \subset [0,T]$ and
\begin{align}
0< t_1 - t_0 \leq  \delta_0^{-1} \lambda_q^{-4} \, .
\label{eq:t1:assume}
\end{align}
Then, in view of \eqref{eq:v:H3:x} and  Proposition~\ref{prop:local:existence} there exists a unique $C^0_t H^{3}_x$ smooth zero-mean solution $v$ of the Cauchy problem for \eqref{eq:NSE} on $[t_0,t_1]$, with initial datum $v_0$. 
We claim that there exists a constant $C = C(p_0,\alpha) > 0$ such that for any $p \in [p_0,2]$,   we have that 
\begin{subequations}
\begin{align}
\norm{v(t) - v_q(t)}_{L^p} &\leq C |t -t_0| \norm{|\nabla|\RR_q}_{L^\infty([t_0,t_1];L^p)}
\label{eq:v:diff:Lp}
\\
\norm{\RSZ v(t) - \RSZ v_q(t)}_{L^p} &\leq C  |t  -t_0| \norm{  \RR_q}_{L^\infty([t_0,t_1];L^p)}
\label{eq:Rv:diff:Lp}
\end{align}
\end{subequations}
holds for all $t \in (t_0,t_1]$. 
\end{proposition}

In particular, letting 
\begin{align}
p_0  =  1 +  \frac{\eps_R}{32} \in (1,\sfrac 54)\,,
\label{eq:p0:cond:1}
\end{align}
from the bounds \eqref{eq:v:diff:Lp}--\eqref{eq:Rv:diff:Lp} we obtain the following stability estimate:
\begin{corollary}
\label{cor:stability}
Fix $\alpha  \in [1, \sfrac 54)$. Assuming that $a \geq 1$ is sufficiently large, depending only on $\eps_R$, and if $t_1 \in (t_0,T]$ obeys \eqref{eq:t1:assume}, then we have the bounds
\begin{subequations}
\begin{align}
\norm{v - v_q }_{L^\infty([t_0,t_1];L^2)} & \leq  |t_1 -t_0| \lambda_q^{5}
\label{eq:v:diff:L2}
\\ 
\norm{\RSZ(v-v_q) }_{L^\infty([t_0,t_1];L^1)} & \leq   |t_1 -t_0| \lambda_q^{-  \frac{3}{4} \eps_R} \delta_{q+1} \, .
\label{eq:Rv:diff:L1}  
\end{align}
\end{subequations}
\end{corollary}
\begin{proof}[Proof of Corollary~\ref{cor:stability}]
We show that estimates \eqref{eq:v:diff:Lp}--\eqref{eq:Rv:diff:Lp} imply the bounds \eqref{eq:v:diff:L2}--\eqref{eq:Rv:diff:L1}. Recall the stress $\RR_q$ has zero mean. For $p\in (1,2]$ and $\delta \in [0,1]$ by interpolation we have the inequalities $\norm{|\nabla|^\delta f}_{L^p} \les \norm{f}_{L^p}^{1 - \frac{\delta}{3}} \norm{f}_{\dot W^{3,p}}^{\frac{\delta}{3}}$ and $\norm{f}_{L^p} \les \norm{f}_{L^1}^{\frac{1}{p}} \norm{f}_{L^\infty}^{1 - \frac 1p}$. Moreover, since $H^3 \subset L^\infty$ and $H^3 \subset W^{3,p}$, we obtain the Gagliardo-Nirenberg-type inequality
\begin{align}
\norm{ |\nabla|^\delta \RR_q}_{L^p} 
\les \norm{\RR_q}_{L^1}^{\gamma} \norm{\RR_q}_{H^3}^{1-\gamma}, \qquad \mbox{with} \qquad \gamma = \frac{1}{p} - \frac{\delta}{3p} \, .
\label{eq:GNS:1}
\end{align}
The implicit constant depends only on $p$ and $\delta$.

In order to prove \eqref{eq:v:diff:L2}, we use \eqref{eq:v:diff:Lp}
and apply estimate \eqref{eq:GNS:1} with $\delta = 1$ and $p=2$. We obtain from \eqref{eq:Rq:L1}--\eqref{eq:Rq:H3:x} that
\begin{align}\notag
\norm{|\nabla| \RR_q}_{L^2} \les (\lambda_q^{-\eps_R} \delta_{q+1})^{\frac 12 - \frac{1}{6}} \lambda_q^{7  \left( \frac12 + \frac{1}{6}\right)} \, ,
\end{align}
from which estimate \eqref{eq:v:diff:L2} follows, since $\delta_{q+1}\leq \lambda_1^{\beta}$, and $\beta$ is sufficiently small. The leftover power of $\lambda_q$ may be used to absorb any constants. 

Similarly, in order to prove \eqref{eq:Rv:diff:L1}, we use \eqref{eq:Rv:diff:Lp}, the bound \eqref{eq:GNS:1} with $\delta = 0$  and $p=p_0$, and the embedding $L^{p_0} \subset L^1$, to obtain 
\begin{align}
\norm{\RR_q}_{L^1} \les \norm{\RR_q}_{L^{p_0}} 
&\les (\lambda_q^{-\eps_R} \delta_{q+1})^{\frac{1}{p_0}} (\lambda_q^{7})^{\frac{p_0-1}{p_0}} \notag\\
&\les (\lambda_q^{-\eps_R} \delta_{q+1})^{1 - \frac{(p_0-1)}{p_0}} \lambda_q^{ \frac{ 7 (p_0-1)}{p_0}} \notag\\
&= \lambda_q^{-\frac 34 \eps_R} \delta_{q+1} \left( \lambda_q^{-\frac{\eps_R}{4}}   \left( \delta_{q+1}^{-1} \lambda_q^{\eps_R + 7} \right)^{\frac{p_0-1}{p_0}} \right) \notag \\
&\leq \lambda_q^{-\frac 34 \eps_R} \delta_{q+1}  \lambda_q^{-\frac{\eps_R }{4} + (p_0-1) (\eps_R  +  7+ 2 \beta b)} \notag\\
&\leq \lambda_q^{-\frac 34 \eps_R} \delta_{q+1}  \lambda_q^{-\frac{\eps_R }{4} + 8 (p_0-1) }\, .\notag
\end{align}
In the last inequality above  we have used the definitions of $\delta_{q+1}$ and  $\lambda_q$, and the fact that $p_0 \geq 1$. Estimate \eqref{eq:Rv:diff:L1} follows from  the assumption \eqref{eq:p0:cond:1} on $p_0$, upon using the leftover power of $\lambda_q$ to absorb the implicit constants.
\end{proof}

\begin{proof}[Proof of Proposition~\ref{prop:stability}]
For simplicity, by temporal translation invariance it is sufficient to consider the case $t_0=0$. 

In order to prove \eqref{eq:v:diff:Lp} we let $u = v_q - v$ and $q = p_q - p$. Then $\div u =0$, $u\vert_{t=0} = 0$, and $u$ obeys the  equation
\begin{align}
\partial_t u + (-\Delta)^\alpha u = {\mathbb P} \div \RR_q - {\mathbb P} \div(v \otimes u + u \otimes v_q)\,,
\label{eq:u:diff:evo:2}
\end{align} 
where ${\mathbb P}$ is the Leray projector.  Then, since $u(0) = 0$ the solution of \eqref{eq:u:diff:evo:2} may be written in integral form as
\begin{align}
u(t) = \int_0^t e^{-(t-s) (-\Delta)^\alpha} {\mathbb P} \div \left( \RR_q -  v \otimes u - u \otimes v_q\right)(s) ds\,.
\label{eq:u:model:mild}
\end{align}

Next, we use that that for $p\in [1,2]$, $t>0$,  and any periodic function $\phi$ of zero mean we have that
\begin{subequations}
\begin{align}
\norm{e^{- t(-\Delta)^\alpha} \phi}_{L^p} &\les  \norm{\phi}_{L^p}
\label{eq:z:model:heat:1} \\
\norm{\nabla e^{- t(-\Delta)^\alpha} \phi}_{L^p} &\les \frac{1}{t^{\frac{1}{2\alpha}}} \norm{\phi}_{L^p},
\label{eq:z:model:heat:2}
\end{align}
\end{subequations}
where the implicit constant only depends on $\alpha$.
These estimates follow from $L^1$ bounds for the Green's function of the fractional heat equation. We will also frequently use the Gagliardo-Nirenberg estimates 
\begin{subequations}
\begin{align}
\norm{\nabla \phi}_{L^\infty} &\les  \norm{\phi}_{L^2}^{\sfrac 16} \norm{\phi}_{\dot{H}^3}^{\sfrac 56}
\label{eq:GNS:2} \\
\norm{\phi}_{L^\infty} &\les  \norm{\phi}_{L^2}^{\sfrac 12} \norm{\phi}_{\dot{H}^3}^{\sfrac 12}
\label{eq:GNS:3}
\end{align}
\end{subequations}
which hold for zero-mean periodic functions $\phi$.

We return to \eqref{eq:u:model:mild} and obtain that 
\begin{align}
\norm{u(t)}_{L^p} 
&\leq \int_0^t \norm{ e^{-(t-s) (-\Delta)^\alpha} {\mathbb P} \div \left( \RR_q -  v \otimes u - u \otimes v_q\right)(s)}_{L^p} ds 
\notag\\
&\les \int_0^t \norm{|\nabla| \RR_q(s)}_{L^p} + \frac{1}{(t-s)^{\frac{1}{2\alpha}}} \norm{ \left(  v \otimes u + u \otimes v_q\right)(s)}_{L^p} ds 
\notag\\
&\leq C_1 \int_0^t   \norm{|\nabla| \RR_q(s)}_{L^p} + \frac{1}{(t-s)^{\frac{1}{2\alpha}}} \left( \norm{v(s)}_{L^\infty} + \norm{v_q(s)}_{L^\infty} \right) \norm{u(s)}_{L^p}  ds,
\label{eq:u:model:evo:1}
\end{align}
for a suitable constant $C_1>0$ which only depends on $p_0$, since $p \in [p_0,2]$ and $\alpha \in [1,\sfrac 54]$. 
Next, we claim that if $t_1>0$ is chosen sufficiently small, depending on $\norm{v}_{L^\infty}$ and $\norm{v_q}_{L^\infty}$, then we have 
\begin{align}
\norm{u(t)}_{L^p} \leq 2 C_1 t \norm{|\nabla| \RR_q(s)}_{L^\infty([0,t_1];L^p)} \qquad \mbox{for all} \qquad t\in (0,t_1].
\label{eq:u:model:want}
\end{align}
This estimate follows from Gr\"{o}nwall's inequality, using the following bootstrap argument. Assuming that the bound \eqref{eq:u:model:want} holds, we claim that the same estimate holds with the constant $2C_1$ replaced by the smaller constant $3C_1/2$. Indeed, inserting \eqref{eq:u:model:want} in \eqref{eq:u:model:evo:1} we obtain
\begin{align}
\frac{\norm{u(t)}_{L^p}}{2 C_1 t   \norm{|\nabla| \RR_q(s)}_{L^\infty([0,t_1];L^p)}}
&\leq \frac{1}{2} + \frac{1}{t} \left( \norm{v}_{L^\infty} + \norm{v_q}_{L^\infty} \right) \int_0^t \frac{s\; ds}{(t-s)^{\frac{1}{2\alpha}}}
\notag\\
&\leq \frac{1}{2} + \frac{2\alpha}{2\alpha - 1} t^{1- \frac{1}{2\alpha}} \left( \norm{v}_{L^\infty} + \norm{v_q}_{L^\infty} \right).
\label{eq:u:model:evo:2}
\end{align}
Thus if we ensure that 
\begin{align}
4   t_1^{\frac{1}{2} + \frac{\alpha-1}{2\alpha}} \left( \norm{v}_{L^\infty} + \norm{v_q}_{L^\infty} \right) \leq \frac 14\,,
\label{eq:u:model:t:cond:1}
\end{align} 
then \eqref{eq:u:model:evo:2} shows that \eqref{eq:u:model:want} holds with constant $3C_1/2$, as desired.
However, by \eqref{eq:GNS:3} we know that  
\begin{align}
  \norm{v}_{L^\infty} + \norm{v_q}_{L^\infty}  &\leq 
 C_1  \left( \norm{v}_{L^2}^{\sfrac 12}\norm{v}_{H^3}^{\sfrac 12} + \norm{v_q}_{L^2}^{\sfrac 12}\norm{v_q}_{H^3}^{\sfrac 12} \right)\notag
\end{align}
for some universal constant $C_1>0$, and further, using \eqref{eq:v:L2:x}, \eqref{eq:v:H3:x}, \eqref{eq:v:L2:1} and \eqref{eq:v:H3:1}, we obtain that 
\begin{align}
  \norm{v}_{L^\infty} + \norm{v_q}_{L^\infty}  
  \leq  C_1  \left( \norm{v_0}_{L^2}^{\sfrac 12} (2 \norm{v_0}_{H^3})^{\sfrac 12} + \norm{v_q}_{L^2}^{\sfrac 12}\norm{v_q}_{H^3}^{\sfrac 12} \right) 
  \leq 4 C_1 \delta_0^{\sfrac 14} \lambda_q^{2}\,.\notag
\end{align}
To conclude, we use \eqref{eq:t1:assume}, which shows that the left side of \eqref{eq:u:model:t:cond:1} is bounded from above by
\begin{align}
4  (\delta_0^{-1} \lambda_q^{-4})^{\frac 12 + \frac{\alpha-1}{2\alpha}} 4 C_1 \delta_0^{\sfrac 14} \lambda_q^{2} = 16  C_1  \delta_0^{-\sfrac 14} (\delta_0 \lambda_q^{4})^{- \frac{\alpha-1}{2\alpha}} \leq 16  C_1  \delta_0^{-\sfrac 14}  \leq \frac 14\,,\notag
\end{align}
by letting $a$, and hence $\delta_0$, be sufficiently large. Here we have used that $\alpha \geq 1$, and that $\delta_0, \lambda_q \geq 1$. Thus, we have shown that \eqref{eq:u:model:want} holds.

In order to prove \eqref{eq:Rv:diff:Lp} we denote
\begin{align}
z = \Delta^{-1} \curl u\,.\notag
\end{align}
Note that since $\div u =0$ we have $\curl z = - u$, and using the Calder\'on-Zygmund inequality we have $\norm{\RSZ u(t)}_{L^p} \les \norm{z(t)}_{L^p}$. Thus our goal is to obtain $L^p$ estimates for $z(t)$.  We apply $\Delta^{-1} \curl$ to the equation obeyed by $u$ (it is convenient to rewrite \eqref{eq:u:diff:evo:2} without Leray projectors, and add a pressure gradient term, which is then annihilated by the $\curl$ operator) and obtain
\begin{align}
&\partial_t z + v \cdot \nabla z  +(- \Delta)^\alpha  z  \notag\\
&= \Delta^{-1} \curl \div \RR_q + [ \Delta^{-1} \curl , v \cdot \nabla ] \curl z + \Delta^{-1} \curl (\curl z \cdot \nabla v_q) \notag\\
&= \Delta^{-1} \curl \div \RR_q +  \Delta^{-1} \curl \div \left(  (z\times \nabla) v  \right) +  \Delta^{-1} \nabla \div  \left( ( z \cdot \nabla) v\right)    + \Delta^{-1} \curl \div \left( ((z\times \nabla) v_q)^T \right).
\label{eq:z:diff:evo}
\end{align} 
For the last term on the right side of in \eqref{eq:z:diff:evo} we have used the identity 
\begin{align}
(\curl z  \cdot \nabla) v_q = \div \left( ((z \times \nabla) v_q)^T \right)\,, \notag
\end{align}
which written for the $i^{th}$ component is
\begin{align}
((\curl z  \cdot \nabla v_q)^i = \epsilon_{jkl} \partial_k z^l \partial_j v_q^i = \partial_k (\epsilon_{jkl} z^l \partial_j v^i_q) - \epsilon_{jkl} z^l \partial_{j}\partial_{k} v_q^i = \partial_k (\epsilon_{klj} z^l \partial_j v^i_q) =: \partial_k ((z \times \nabla) v_q )^{ki}.\notag
\end{align} 
Here we used that the transposition of two indices in $\epsilon_{jkl}$ results in a $(-1)$ factor. 
Moreover, we have also spelled out the commutator term on the right side of \eqref{eq:z:diff:evo}  as 
\begin{align}
[\Delta^{-1} \curl, v\cdot \nabla] \curl z = \Delta^{-1} \curl \div \left( (z\times \nabla) v \right) + \Delta^{-1} \nabla \div \left( ( z \cdot \nabla) v\right)   \,,\notag
\end{align}
which written for the $i^{th}$ component is
\begin{align}
\left([\Delta^{-1} \curl, v\cdot \nabla] \curl z \right)^i 
&= \epsilon_{ijk} \Delta^{-1} \partial_j \left( v^m \partial_m (\curl z)^k \right) + v^m \partial_m z^i \notag\\
&= \epsilon_{ijk} \epsilon_{kln} \Delta^{-1} \partial_j \left( v^m \partial_{m}\partial_l z^n \right) + v^m \partial_m z^i \notag\\
&= - \epsilon_{ijk} \epsilon_{kln} \Delta^{-1} \partial_j \partial_{m} \left( \partial_l v^m  z^n \right) + \epsilon_{ijk} \epsilon_{kln} \Delta^{-1} \partial_j \partial_l \left( v^m \partial_{m} z^n \right) + v^m \partial_m z^i \notag\\
&=  \Delta^{-1} \epsilon_{ijk} \partial_j \left( \epsilon_{knl}   \partial_{m} \left( \partial_l v^m  z^n \right) \right) -  \epsilon_{ijk} \epsilon_{nlk}   \Delta^{-1} \partial_j \partial_l \left( v^m \partial_{m} z^n \right) + v^m \partial_m z^i \notag\\
&= \Delta^{-1} \epsilon_{ijk} \partial_j \left( \partial_m \left( \epsilon_{knl}   z^n \partial_l v^m    \right) \right) + \Delta^{-1} \partial_i \partial_n  (v^m \partial_m z^n) \notag\\
&=  \Delta^{-1} \epsilon_{ijk} \partial_j \left( \partial_m \left( \epsilon_{knl}   z^n \partial_l v^m    \right) \right)  + \Delta^{-1}  \partial_i \partial_m  ( z^n \partial_n v^m )\,.\notag
\end{align}
Here we have also used that $\epsilon_{ijk} = 0$ if two of the indices $i,j$, or $k$ repeat, and that $ \epsilon_{ijk} \epsilon_{nlk}  =   \delta_{in} \delta_{jl} - \delta_{il} \delta_{jn} $, where the $\delta$'s refer to the Kronecker symbol. 
 
Using \eqref{eq:z:diff:evo}, upon placing the $v\cdot \nabla z = \div(v \otimes z)$ term on the right side, and using that  $z(t_0) = 0$, the solution to \eqref{eq:z:diff:evo} may be written in integral form as
\begin{align}
z(t) 
&= \int_0^t e^{-(t-s)(-\Delta)^\alpha} \left( \Delta^{-1} \curl \div \RR_q  +   \Delta^{-1} \curl \div \left( ((z\times \nabla) v_q)^T \right)  - \div (v \otimes z)\right)(s) ds \notag\\
&\qquad + \int_0^t e^{-(t-s)(-\Delta)^\alpha} \left(   \Delta^{-1} \curl \div \left( (z\times \nabla) v  \right) + \Delta^{-1} \nabla \div  \left( ( z \cdot \nabla) v\right)   \right)(s) ds\,.
\label{eq:z:model:mild}
\end{align}
From \eqref{eq:z:model:heat:1}--\eqref{eq:z:model:heat:2} and the boundedness of Calder\'on-Zygmund operators on $L^p$, similarly to \eqref{eq:u:model:evo:1} we conclude that
\begin{align}
\norm{z(t)}_{L^p} &\les \int_0^t \norm{\RR_q(s)}_{L^p} +   \norm{ \left( (z\times \nabla) v_q\right) (s)}_{L^p} + \frac{1}{(t-s)^{\frac{1}{2\alpha}}} \norm{(v \otimes z)(s)}_{L^p}  +   \norm{ \left( (z\times \nabla) v\right)(s)}_{L^p} \notag\\
&\qquad+  \norm{ \left( (z \cdot \nabla) v\right)(s)}_{L^p} ds \notag\\
&\leq C_1 t \norm{\RR_q}_{L^\infty([0,t_1];L^p)}  +  C_1 \left(  \norm{\nabla v_q}_{L^\infty} 
+   \norm{\nabla v}_{L^\infty} \right) \int_0^t \norm{z(s)}_{L^p} ds + C_1 \norm{v}_{L^\infty} \int_0^t \frac{\norm{z(s)}_{L^p}}{(t-s)^{\frac{1}{2\alpha}}}   ds
\label{eq:z:model:have:1}
\end{align}
where $C_1$ depends only on $p_0$ and $\alpha$, since $p\in [p_0,\alpha]$.
Next we claim that if $t_1$ is chosen sufficiently small, then 
\begin{align}
\norm{z(t)}_{L^p} \leq 2 C_1 t \norm{\RR_q}_{L^\infty(0,t_1;L^p)} \qquad \mbox{for all} \qquad t\in (0,t_1]\,.
\label{eq:z:model:want}
\end{align}
The argument is similar to the one for the bound for $u(t)$, so we only sketch the details. Let us assume that \eqref{eq:z:model:want} holds. Then from \eqref{eq:z:model:have:1}  we obtain  
\begin{align}
\frac{\norm{z(t)}_{L^p}}{2 C_1 t \norm{\RR_q}_{L^\infty(0,t_1;L^p)}}
&\leq \frac 12  +    t \left(  \norm{\nabla v_q}_{L^\infty} 
+   \norm{\nabla v}_{L^\infty} \right)   +2   t^{1 - \frac{1}{2\alpha}} \norm{v}_{L^\infty} .
\label{eq:z:model:have:2}
\end{align}
Therefore, if we ensure that $t_1$ is small enough so that 
\begin{align}
t_1  \left(  \norm{\nabla v_q}_{L^\infty} 
+   \norm{\nabla v}_{L^\infty} \right) + t_1^{1 - \frac{1}{2\alpha}} \norm{v}_{L^\infty} \leq \frac{1}{5},
\label{eq:z:model:T:cond}
\end{align}
then \eqref{eq:z:model:have:2} implies that 
\begin{align}
\frac{\norm{z(t)}_{L^p}}{2 C_1 t \norm{\RR_q}_{L^\infty(0,t_1;L^p)}} \leq \frac 12 + \frac 25 < 1\notag
\end{align}
which shows that the bootstrap assumption was justified, and thus \eqref{eq:z:model:want} holds on $[0,T]$.
Denote by $C_1$ the universal constant in the Gagliardo-Nirenberg inequalities \eqref{eq:GNS:2}--\eqref{eq:GNS:3}. By also appealing to \eqref{eq:v:L2:x}--\eqref{eq:v:H3:x}, \eqref{eq:v:L2:1}--\eqref{eq:v:H3:1}, and our assumption \eqref{eq:t1:assume} for $t_1$, we obtain that the left side of \eqref{eq:z:model:T:cond} is bounded from above by 
\begin{align}
&C_1 t_1 \left(\norm{v_q}_{L^2}^{\frac 16} \norm{v_q}_{H^3}^{\frac 56} + \norm{v}_{L^2}^{\frac 16} \norm{v}_{H^3}^{\frac 56} \right) + C_1 t_1^{\frac 12 + \frac{(\alpha-1)}{2\alpha}} \norm{v}_{L^2}^{\frac 12} \norm{v}_{H^3}^{\frac 12} 
\notag\\
&\qquad \leq 4 C_1 t_1  \delta_0^{\frac{1}{12}} \lambda_q^{\frac{10}{3}} + 2 C_1 t_1^{\frac 12 + \frac{(\alpha-1)}{2\alpha}} \delta_0^{\frac 14}  \lambda_q^{2} \notag\\
&\qquad \leq 4 C_1 \delta_0^{-\frac{11}{12}} \lambda_q^{-\frac{2}{3}} + 2  C_1 \delta_0^{- \frac 14}  (\delta_0 \lambda_q^{4})^{- \frac{(\alpha-1) }{2\alpha}}
  \leq 6 C_1  \delta_0^{- \frac 14}  
  \leq \frac{1}{5}\notag
\end{align}
once we ensure that $a$, and hence $\delta_0$  is sufficiently large.  This concludes the proof of \eqref{eq:z:model:want}. 
\end{proof}

\subsection{Proof of Proposition~\ref{prop:bar:vq:Rq}}\label{ss:proof_prop}
We first define a $C^\infty$ smooth partition of unity $\{\eta_i\}_{i=0}^{n_{q+1}}$, such that $0 \leq \eta_i \leq 1$, and with 
\begin{align}
\sum_{i=0}^{n_{q+1}} \eta_i(t) = 1, \qquad \mbox{for every} \qquad  t \in [\sfrac{T}{3},\sfrac{2T}{3}]\, .
\label{e:chi:partition}
\end{align}
Denoting 
\begin{align}
t_i = \vartheta_{q+1}i \, ,\notag
\end{align}
this may be achieved by letting $\eta_i$ also have the following properties:
\begin{enumerate}[label=(\roman*)]
\item $\eta_i$ has support in $[t_i,t_{i+1} + \tau_{q+1}]$
\item $\eta_i$ is identically $1$ on $[t_i + \tau_{q+1},t_{i + 1}]$
\item $\eta_i$ satisfies the estimate
\begin{align}
\norm{\partial_t^M \eta_{i}}_{L^\infty} \les \tau_{q+1}^{-M} \, ,
\label{eq:dt:eta}
\end{align} where the implicit constant is independent of $\tau_{q+1}$  $\vartheta_{q+1}$, and $i$. 
\end{enumerate}  
As a consequence of the above properties, we have that $\eta_i \eta_j = 0$ whenever $|i-j| > 1$, and 
\begin{align}
\supp (\eta_i \eta_{i-1}) \subset [t_i,t_i+ \tau_{q+1}] \, .\notag
\end{align}
Having constructed the partition of unity $\{ \eta_i\}_{i=0}^{n_{q+1}}$, we next construct exact solutions $v_i$ of the Navier-Stokes equation for suitably defined datum.

For every $1 \leq i \leq n_{q+1}$ we define $v_i(x,t)$ to be the unique smooth solution of the Cauchy problem for Navier-Stokes equation \eqref{eq:NSE} with initial condition equal to $v_q$ at $t_{i-1}$:
\begin{subequations}
\label{eq:vi:evo}
\begin{align}
\partial_t v_i + \div (v_i \otimes v_i) + \nabla p_i  + (-\Delta)^{\alpha} v_i &= 0 \\
\div v_i &=0\\
v_i(t_{i-1}) &= v_q(t_{i-1}) \,.
\end{align}
\end{subequations}
In view of \eqref{eq:v:L2:x}--\eqref{eq:v:H3:x}, and Proposition~\ref{prop:local:existence}
this solution $v_i$ is uniquely defined  and obeys the estimates 
\begin{subequations}
\begin{align}
\norm{v_i(t)}_{L^2} \leq \norm{v_q(t_{i-1})}_{L^2} &\leq 2 \delta_0^{\sfrac 12} - \delta_q^{\sfrac 12}
\label{eq:vi:L2}\\
\norm{v_i(t)}_{H^3} \leq 2 \norm{v_q(t_{i-1})}_{H^3} &\leq 2 \lambda_q^{4}
\label{eq:vi:H3}\\
|t-t_{i-1}|^{\frac{N}{2\alpha}+M} \norm{\partial_t^M D^N v_i(t)}_{H^3} 
&\les  \lambda_q^{4}
\label{eq:vi:H3:N}
\end{align}
\end{subequations}
for all $N\geq 0$, $M \in \{0,1\}$ and all 
\begin{align}
t > t_{i-1} \qquad \mbox{such that} \qquad t - t_{i-1} \leq  \frac{c}{4 \lambda_q^{4} \delta_0^{\sfrac 12}} \leq \frac{c}{ \lambda_q^{4} (1+ 2 \delta_0^{\sfrac 12})^{\frac{1}{2\alpha -1}}}
\label{eq:ti:max}
\end{align}
where $c\in (0,1)$ is the universal constant from \eqref{eq:t:weird}, and $\alpha \geq 1$. Note that the definitions \eqref{eq:delta:q:def}, \eqref{eq:theta:q:def}, and the fact that $\beta \leq 1$,  imply that
\begin{align}
\vartheta_{q+1} = \frac{\delta_{q+1}^{\sfrac 12}}{\lambda_q^{7} } = \frac{1}{\lambda_q^{4} \delta_0} \frac{\delta_0 \delta_{q+1}^{\sfrac 12}}{\lambda_q^{3} } \leq  \frac{1}{\lambda_q^{4} \delta_0} \frac{\lambda_1^{\frac{3\beta}{2}}}{\lambda_q^{3} } \leq \frac{1}{\lambda_q^{4} \delta_0}   \, .
\label{eq:tau:theta:cond:1}
\end{align}
Therefore, assuming that $\delta_0 = \lambda_1^{3\beta} \lambda_0^{-2\beta} \geq \lambda_0^\beta$ is sufficiently large, depending on the universal constant $c$, by \eqref{eq:tau:theta:cond:1} we have that
\begin{align}\notag
3 \vartheta_{q+1} \leq \frac{c}{8 \lambda_q^{4} \delta_0^{\sfrac 12}} \, ,
\end{align}
which is consistent with \eqref{eq:ti:max}. 
Therefore for all $1\leq i \leq n_{q+1}$ the exact solutions $v_i(x,t)$ are smooth and well-defined for all $t \in (t_{i-1}, t_{i+2}] \supset \sup(\eta_i)$. Moreover, since
\begin{align}
t\in \supp (\eta_i) \qquad  \Rightarrow \qquad \vartheta_{q+1} \leq t - t_{i-1} \leq   3 \vartheta_{q+1}  \, ,\notag
\end{align}
from \eqref{eq:vi:H3:N} we obtain the bound 
\begin{align}
\sup_{t \in \supp(\eta_i) } \norm{\partial_t^M D^N v_i(t)}_{H^3} 
&\les \lambda_q^{4} \vartheta_{q+1}^{-\frac{N}{2\alpha}-M}, \qquad \mbox{for} \qquad 1\leq i \leq n_{q+1} \, ,
\label{eq:vi:high:deriv}  
\end{align}
where the implicit constant  depends only on $N\geq 0$ and $M \in \{0,1\}$.

At this stage we glue the solutions $v_i$ together in order to construct $(\overline v_q, {\mathring {\bar R}}_{q})$.   
We define the divergence-free (note that the cutoffs $\eta_i$ are only functions of time) velocity and the interpolated pressure as
\begin{subequations}
\begin{align}
\overline v_q(x,t)&=\sum_{i=1}^{n_{q+1}} \eta_i(t) v_i(x,t) \, , \quad \mbox{for all} \quad t \in [\sfrac{T/3}, \sfrac{2T}{3}] \, ,\label{eq:bar:vq:def} \\
\overline p_q^{(1)}(x,t)&=\sum_{i=1}^{n_{q+1}} \eta_i(t) p_i(x,t)\, , \quad \mbox{for all} \quad t \in [\sfrac{T/3}, \sfrac{2T}{3}] \, ,\notag
\end{align}
\end{subequations}
where $p_i$ is the pressure associated to the exact solution $v_i$.
Also we let
\begin{subequations}
\begin{align}
\overline v_q(x,t)&=v_q(x,t) = v_0(x,t) \, , \quad \mbox{for all} \quad t \in [0,\sfrac{T}{3}] \cup [\sfrac{2T}{3},T] \, ,
\label{eq:bar:vq:def:0} \\
\overline p_q^{(1)}(x,t)&=p_q(x,t) = p_0(x,t) \, , \quad \mbox{for all} \quad t \in [0,\sfrac{T}{3}] \cup [\sfrac{2T}{3},T] \, .\notag
\end{align}
\end{subequations}
Here we have used that $ [0,\sfrac{T}{3}] \cup [\sfrac{2T}{3},T] = \GG^{(0)}$, and the inductive assumption \eqref{eq:vq:unchanged}.

Having defined $\bar v_q$, we next prove that \eqref{eq:vq:same} holds. For $t \in \GG^{(0)}$, this holds by construction. In view of \eqref{e:chi:partition}, it suffices to show that if for some $i \in \{1,\ldots,n_{q+1}\}$ we have $t \in \supp(\eta_i) \cap \GG^{(q)}$, then $v_i(t) = v_q(t)$. For this purpose recall by \eqref{eq:v:H3:x} and \eqref{eq:Rq:time:support} that $v_q$ is a strong solution of the Navier-Stokes equation for all $t$ such that ${\rm dist}(t,\GG^{(q)}) \leq \tau_q$. Moreover, $v_i$ solves the Cauchy problem \eqref{eq:vi:evo}, so by the uniqueness of solutions in $C^0_t H^3_x$ of the Navier-Stokes equation, we only need to ensure that ${\rm dist}(t_{i-1},\GG^{(q)}) \leq \tau_q$. This follows from the fact that $t \in \GG^{(q)}$, and $0 < t - t_{i-1} \leq 3 \vartheta_{q+1} \leq \tau_q$. The last inequality trivially holds by \eqref{eq:theta:q:def} and \eqref{eq:tau:q:def} for $q\geq 1$, and by taking $a$ sufficiently large for $q=0$. Thus, we have proven that \eqref{eq:vq:same} holds. 

At this stage we show that the set  $\BB^{(q+1)}$ defined in \eqref{eq:B:q+1:def}, and hence implicitly $\GG^{(q+1)} = [0,T] \setminus \BB^{(q+1)}$, obey  properties \ref{eq:cond:ii}--\ref{eq:cond:iv} with $q$ replaced by $q+1$.  In order to prove \ref{eq:cond:ii}, assume that  $t \in \GG^{(q)} \cap (t_i - 2 \tau_{q+1}, t_{i} + 3 \tau_{q+1})$, for some $i \in \{1,\ldots, n_{q+1}\}$. Due to \eqref{eq:Rq:time:support} we know that $\RR_q(t') = 0$ for all $|t-t'| \leq \tau_q$. Since  $\tau_q \geq 2\vartheta_{q+1} + 3 \tau_{q+1}$, which holds by \eqref{eq:theta:q:def} and \eqref{eq:tau:q:def} $q\geq 1$, and by taking $a$ sufficiently large for $q=0$, we obtain that $\RR_q \equiv 0$ on $ [t_{i-2},t_{i+1}+\tau_{q+1}] $. Hence, by the definition \eqref{eq:C:def} we have that $i,i-1 \not \in {\mathcal C}$. Thus, $t \not \in \BB^{(q+1)}$ and so $t \in \GG^{(q+1)}$ as desired.  Property \ref{eq:cond:iii} holds by definition~\eqref{eq:B:q+1:def}, since $\tau_{q+1}$ is much smaller than $\vartheta_{q+1}$.  In order to prove property \ref{eq:cond:iv}, we need to estimate the cardinality of the set ${\mathcal C}$ defined in  \eqref{eq:C:def}. By definition, if $i \in {\mathcal C}$, there exists $t\in [t_{i-1},t_{i+1}+\tau_{q+1}]$ such that $\RR_{q}(t) \neq 0$, and thus by property \eqref{eq:Rq:time:support} we have ${\rm dist}(t, \GG^{(q)}) > \tau_q$. Therefore, $\BB^{(q)} \supset (t-\tau_q,t+\tau_q) \supset [t_i,t_{i+1}]$. By the pigeonhole principle we obtain that 
\begin{align}\notag
{\rm card}(\mathcal C) \leq \frac{|\BB^{(q)}|}{\vartheta_{q+1}}.
\end{align}
Estimate \eqref{eq:mead:2} at level $q+1$ then follows from \eqref{eq:B:q+1:def}. 

At this stage we remark that   property \ref{eq:cond:v} will also hold at the end of the convex integration stage. For this purpose, we remark that in the convex integration stage we do not add a perturbation to the solutions on the good set $\GG^{(q+1)} \supset \GG^{(q)}$, i.e.\ $v_{q+1}(t) = \bar v_q(t)$ for $t \in \GG^{(q+1)}\supset \GG^{(q)}$. Assuming for the moment this feature of our construction, property \eqref{eq:vq:same} established above and the inductive \eqref{eq:vq:unchanged}   show that \eqref{eq:vq:unchanged} holds at level $q+1$.

We now derive the formula for $\supp {\mathring {\bar R}}_{q}$. Note that on $[0,\sfrac{T}{3}] \supset [t_0,t_2]$ and on $[\sfrac{2T}{3},T] \supset [t_{n_{q+1}-1},t_{n_{q+1}}]$ we have that $\bar v_q = v_q$ is a smooth solution of the Navier-Stokes equation, and hence automatically 
\begin{align}\notag
{\mathring {\bar R}}_{q} = 0 \qquad \mbox{on} \qquad [t_0,t_2] \cup [t_{n_{q+1}-1},t_{n_{q+1}}]\, .
\end{align}
For $i\geq 2$, on the interval $[t_i,t_{i+1}]$ we have
\begin{align*}
\overline v_q&=(1-\eta_{i}) v_{i-1}+ \eta_{i} v_{i}\\
\overline p_q^{(1)}&=(1-\eta_{i}) p_{i-1}+ \eta_{i} p_{i}
\end{align*}
and similarly to \cite[Section 4.2]{BDLSV17}, we obtain
\begin{align}
&\partial_t\overline v_q+\div(\overline v_q\otimes \overline v_q)+ (-\Delta)^\alpha \overline v_q+\nabla\overline p_q^{(1)}  \notag\\
&\quad= (1-\eta_{i}) \partial_t v_{i-1}+ \eta_{i}\partial_t v_{i}+\partial_t\eta_{i}(v_{i}-v_{i-1}) \notag \\
&\quad \quad +(1-\eta_{i})^2 \div\left( v_{i-1}\otimes v_{i-1} \right)+ \eta_{i}^2 \div \left(v_{i}\otimes v_{i}\right)  +\eta_{i}(1-\eta_{i})\div(v_{i-1}\otimes v_{i}+v_{i}\otimes v_{{i-1}})) \notag \\
&\quad \quad + (1-\eta_{i}) (-\Delta)^\alpha v_{i-1}+ \eta_{i} (-\Delta)^\alpha v_{i} + (1-\eta_{i}) \nabla p_{i-1}+ \eta_{i}\nabla p_{i} \notag \\
&\quad =\partial_t\eta_i(v_{i}-v_{i-1}) -\eta_{i}(1-\eta_{i}) \div((v_{i}-v_{i-1})\otimes (v_{i}-v_{i-1})).
\label{eq:bar:vq:evo}
\end{align}
We observe that $v_i-v_{i-1}$ has zero mean because   the exact solutions of the Navier-Stokes equations $v_i, v_{i-1}$ preserve their average in time, and $v_q$ has zero mean by assumption. Hence we can apply the  inverse divergence operator $\mathcal{R}$ to $v_i-v_{i-1}$ and for $i \in \{2, \ldots, n_{q+1}-1\}$ define the symmetric traceless $2$-tensor  
\begin{align}
\mathring{\overline{R}}_q&= \partial_t\eta_i\mathcal{R}(v_i-v_{i-1})-    \eta_i(1-\eta_i)(v_i-v_{i-1})\mathring{\otimes} (v_i-v_{i-1}) \, , \qquad \mbox{for all} \qquad t \in [t_{i},t_{i+1}] \, ,
\label{eq:bar:Rq:def}
\end{align}
where we denote by $a \mathring \otimes b$ the traceless part of the tensor $a\otimes b$.  
We also define the scalar pressure
\begin{align*}
\overline{p}_q &=\overline{p}_q^{(1)} -   \eta_i(1-\eta_i)\left(|v_i-v_{i-1}|^2-\int_{\T^3}|v_i-v_{i-1}|^2\,dx\right)\, , \qquad \mbox{for all} \qquad t \in [t_{i},t_{i+1}]  \, .
\end{align*}
It follows from \eqref{eq:bar:vq:evo} that the pair $(\bar v_q, \mathring{\overline{R}}_q)$ defined by \eqref{eq:bar:vq:def} and \eqref{eq:bar:Rq:def} solves the Navier-Stokes-Reynolds system \eqref{eq:NSE:Reynolds} on $[0,T]$ with associated pressure $\bar p_q$.

Next, we prove that  \eqref{eq:bar:Rq:supp} holds. Note that by construction we have $\eta_i \equiv 1$ on $[t_i +\tau_{q+1},t_{i+1}]$ for all $i \in \{0,\ldots,n_{q+1}\}$, and thus on these sets we have $\partial_t \eta_i =  \eta_i (1-\eta_i) = 0$. Therefore, by \eqref{eq:bar:Rq:def} we have that $ \mathring{\overline{R}}_q(t) = 0$ whenever $t \in [t_i+\tau_{q+1},t_{i+1}]$ for some $i$. Thus it suffices to consider sets of times of the form $(t_i, t_i +\tau_{q+1})$. If $i\in\mathcal C$ or $i-1\in\mathcal C$, then there is nothing to prove since by definition \eqref{eq:B:q+1:def}, ${\rm dist}((t_i, t_i +\tau_{q+1}), \mathcal G^{(q+1)})>2\tau_{q+1}$. Hence, consider the case $i,i-1\notin\mathcal C$.  Thus by the definition of $\mathcal C$, $\RR_{q}(t)=0$ for all $t\in [t_{i-2},t_{t+1}+\tau_{q+1}]$. Since $v_{i-1}(t_{i-2})=v_q(t_{i-2})$, $v_{i}(t_{i-1})=v_q(t_{i-1})$, then since $\RR_{q}$ vanishes on $[t_{i-2},t_{t+1}+\tau_{q+1}]$, it follows by the bounds \eqref{eq:vi:H3} and \eqref{eq:v:H3:x} and the uniqueness of strong solutions to the Navier-Stokes equations  that $v_{i-1}=v_i=v_q$ on $(t_i, t_i +\tau_{q+1})$. Thus by \eqref{eq:bar:Rq:def} we have that $\RR_{q+1}(t)=0$ for $(t_i, t_i +\tau_{q+1})$.

Since in the convex integration stage we do not change the stress on the set $\{ t \colon {\rm dist}(t, \GG^{(q+1)}) \leq \tau_{q+1} \}$, it follows from  \eqref{eq:bar:Rq:supp}  that $\RR_{q+1}(t) = \mathring{\overline{R}}_q(t) = 0 $ for all $t$ such that ${\rm dist}(t, \GG^{(q+1)}) \leq \tau_{q+1}$. Thus \eqref{eq:Rq:time:support}, and hence property \ref{eq:cond:vi},  will automatically hold at the end of the convex integration step.

In order to conclude the proof of Proposition~\ref{prop:bar:vq:Rq}, it remains to prove estimates \eqref{eq:vq:global:1}--\eqref{eq:vq:1} for $\overline v_q$ and \eqref{eq:Rq:1}--\eqref{eq:Rq:2} for $\mathring{\overline{R}}_q$.  

By  \eqref{e:chi:partition}, \eqref{eq:vi:L2}, \eqref{eq:vi:H3}, and the definition of $\bar v_q$ in \eqref{eq:bar:vq:def}, it follows that \eqref{eq:vq:global:1} and \eqref{eq:vq:global:2} hold  for all $t \in  [\sfrac{T}{3},\sfrac{2T}{3}]$. By \eqref{e:chi:partition}, for all $t \in [\sfrac{T}{3},\sfrac{2T}{3}]$  we have that 
\begin{align}
\overline{v}_q(x,t)-v_q(x,t)=\sum_{i=0}^{n_{q+1}} \eta_i(t) (v_i(x,t)-v_q(x,t)) \, ,
\label{eq:vq:minus:bar:vq}
\end{align}
and at each time $t$ at most $2$ terms in the sum are nonzero. Since $v_i$ solves \eqref{eq:vi:evo}, and since $t \in \supp(\eta_i)$ implies that \eqref{eq:ti:max} holds, we may appeal to Corollary~\ref{cor:stability}, with $t_0$ replaced by $t_{i-1}$, and $t_1$ replaced by an arbitrary $t \in \supp \eta_i$. Here we note that condition \eqref{eq:t1:assume} is satisfied on $\supp(\eta_i)$ due to \eqref{eq:ti:max}.   By \eqref{eq:v:diff:L2}, we obtain that  
\begin{align}\notag
\sup_{t \in \supp(\eta_i)} \norm{v_i(t) - v_q(t)}_{L^2} \leq 4 \vartheta_{q+1}  \lambda_q^{5} \,.
\end{align}
Since at most two terms appear in \eqref{eq:vq:minus:bar:vq}, we may use the remaining power of $\lambda_q^{-1}$ to absorb any constants, and \eqref{eq:vq:vell} follows on $[\sfrac{T}{3},\sfrac{2T}{3}]$. 
Moreover,  estimates \eqref{eq:vq:global:1}--\eqref{eq:vq:vell} hold trivially on $[0,\sfrac{T}{3}] \cup [\sfrac{2T}{3},T]$ by the inductive assumptions and definition \eqref{eq:bar:vq:def:0}. Thus, we have proven \eqref{eq:vq:global:1}--\eqref{eq:vq:vell} on $[0,T]$.  

Lastly, \eqref{eq:vq:1} follows from the definition~\eqref{eq:bar:vq:def}, the Leibniz rule, estimate \eqref{eq:dt:eta} for the time derivatives landing on the cutoff functions $\eta_i$, and estimate \eqref{eq:vi:high:deriv} for the space and time derivatives landing on the $v_i$. Here we have used that  $\tau_{q+1}^{-1} > \vartheta_{q+1}^{-1}$. Thus we have established all the desired bounds for $\overline v_q$.

In order to prove the claimed $L^1$ estimate for $\overline \BB^{(q)}$, i.e.\ \eqref{eq:Rq:1}, we appeal to the definition \eqref{eq:bar:Rq:def}.  For the first term, we use \eqref{eq:dt:eta} and again appeal to Corollary~\ref{cor:stability}, this time to estimate \eqref{eq:Rv:diff:L1}, to obtain that  
\begin{align}
\norm{\partial_t \eta_i \; \RSZ(v_i- v_{i-1}) }_{L^1} 
&\leq \norm{\partial_t \eta_i}_{L^\infty} \left( \norm{\RSZ(v_i- v_q) }_{L^\infty(\supp(\eta_i);L^1)}  + \norm{\RSZ(v_{i-1}- v_q) }_{L^\infty(\supp(\eta_i);L^1)}  \right) \notag\\
&\les \tau_{q+1}^{-1} \vartheta_{q+1}  \lambda_q^{-\frac{3 \eps_R}{4}} \delta_{q+1} \notag\\
&\leq  \frac 12 \tau_{q+1}^{-1} \vartheta_{q+1}  \lambda_q^{-\frac{\eps_R}{2}} \delta_{q+1} \, ,
\label{eq:temp:**}
\end{align}
upon using the remaining power of $\lambda_q^{-\frac{\eps_R}{4}}$ to absorb any constants. For the second term in \eqref{eq:bar:Rq:def}, we use \eqref{eq:v:diff:L2} and obtain that
\begin{align}
\norm{\eta_i (1-\eta_i) \; (v_i- v_{i-1})\mathring\otimes(v_i - v_{i-1}) }_{L^1} 
&\leq \norm{v_i - v_{i-1}}_{L^\infty(\supp (\eta_{i-1}\eta_i); L^2)}^2  \notag\\
&\leq  4 \norm{v_i - v_q}_{L^\infty(\supp (\eta_i); L^2)}^2 \notag\\
&\les \left(\vartheta_{q+1} \lambda_q^{5} \right)^2 \notag\\
&\leq  \frac 12 \tau_{q+1}^{-1} \vartheta_{q+1}  \lambda_q^{-\frac{ \eps_R}{2}} \delta_{q+1} \, .
\label{eq:temp:*}
\end{align}
Here we have used that  by $\tau_{q+1} \leq \vartheta_{q+1}$, the definition \eqref{eq:theta:q:def}, and the fact that 
$ \eps_R \leq 1$, to conclude 
\begin{align}\notag
\vartheta_{q+1} \tau_{q+1} \leq \lambda_q^{-14} \delta_{q+1}\leq \lambda_q^{-1} \lambda_{q}^{- 10 - \frac{\eps_R}{2} } \delta_{q+1} \, ,
\end{align}
and using the leftover term  $\lambda_q^{-1}$ to absorb any implicit constants in \eqref{eq:temp:*}. 
Combined, \eqref{eq:temp:**} and \eqref{eq:temp:*} prove \eqref{eq:Rq:1}. 

It remains to prove \eqref{eq:Rq:2}. We return to \eqref{eq:bar:Rq:def}. For the first term we use \eqref{eq:vi:high:deriv} and \eqref{eq:dt:eta} to obtain
\begin{align}
&\norm{\partial_t^M D^N (\partial_t \eta_i \RSZ(v_i-v_{i-1})) }_{H^3} \notag\\
&\quad \les \sum_{M'=0}^M \norm{\partial_t^{M-M'+1} \eta_i}_{L^\infty} \left(\norm{\partial_t^{M'} D^N v_i}_{L^\infty(\supp (\eta_i);H^3)} + \norm{\partial_t^{M'} D^N v_{i-1}}_{L^\infty(\supp(\eta_{i-1});H^3)}\right) \notag\\
&\quad \les \sum_{M'=0}^M \tau_{q+1}^{-M+M'-1} \lambda_q^{4} \vartheta_{q+1}^{-\frac{N}{2\alpha} - M'}  \les  \tau_{q+1}^{-M-1} \lambda_q^{4} \vartheta_{q+1}^{-\frac{N}{2\alpha} }  \,,\notag
\end{align}
since $\tau_{q+1} \leq \vartheta_{q+1}$. This bound is consistent with \eqref{eq:Rq:2}. For the second term in \eqref{eq:bar:Rq:def}, since $H^3$ is an algebra we similarly obtain from \eqref{eq:vi:high:deriv}  and \eqref{eq:dt:eta} that
\begin{align}
&\norm{\partial_t^M D^N ( \chi_i (1-\chi_i) (v_i-v_{i-1}) \otimes(v_i-v_{i-1}))}_{H^3} \notag\\
&\quad \les \sum_{M'=0}^M   \tau_{q+1}^{-M+M'} \norm{\partial_t^{M'} D^N ( (v_i-v_{i-1}) \otimes(v_i-v_{i-1}))}_{L^\infty(\supp(\chi_{i-1} \chi_i);H^3)} \notag\\
&\quad \les \sum_{M'=0}^M  \tau_{q+1}^{-M+M'} \vartheta_{q+1}^{-M'} \lambda_q^{8} \vartheta_{q+1}^{-\frac{N}{2\alpha} }  \les  \tau_{q+1}^{-M-1} \lambda_q^{4} \vartheta_{q+1}^{-\frac{N}{2\alpha} }  \,,\notag
\end{align}
where we have additionally used that $\tau_{q+1} \leq \vartheta_{q+1} \leq \lambda_q^{-4}$, in view of \eqref{eq:tau:theta:cond:1}.

To conclude the proof of Proposition~\ref{prop:bar:vq:Rq}, we note that the second inequality in  \eqref{eq:vq:1} and \eqref{eq:Rq:2}, which bounds cost of a spatial derivative by $\tau_{q+1}^{-1}$, instead of $\vartheta_{q+1}^{-\sfrac{1}{(2\alpha)}}$, follows form the fact that $\alpha \in [1,\sfrac 54)$ and $1 \leq \vartheta_{q+1}^{-1} \leq \tau_{q+1}^{-1}$.

\section{Convex integration step: the perturbation}

\subsection{Intermittent jets}\label{s:intermittent}

Let us recall the following result from~\cite{DaneriSzekelyhidi17}:
\begin{lemma}\label{l:linear_algebra} For $\alpha=1,2$, there exists subsets $\Lambda_{\alpha} \subset\mathbb S^2\cap \mathbb Q^3$ and smooth  functions $\gamma_{\xi}: \mathcal N\rightarrow \mathbb R$ such that
\begin{align*}
R=\sum_{\xi\in \Lambda_{\alpha}} \gamma^2_{\xi}(R)(\xi\otimes \xi)
\end{align*}
for every symmetric matrix $R$ satisfying $\abs{R-\Id}\leq \sfrac12$.
\end{lemma}
For each $\xi\in \Lambda_\alpha$, let use define $A_{\xi}\in\mathbb S^2\cap \mathbb Q^3$ to be an orthogonal vector to $\xi$. Then for each $\xi\in \Lambda_\alpha$, we have that $\{\xi,A_{\xi},\xi\times A_{\xi}\}\subset \mathbb S^2\cap \mathbb Q^3$ form an orthonormal basis for $\mathbb R^3$. Furthermore, since the index sets $\{\Lambda_\alpha \}_{\alpha=1,2}$ are finite, there exists a universal natural number $N_{\Lambda}$ such that
\begin{equation}\label{e:Mignolet}
\left\{N_{\Lambda}\xi,~N_{\Lambda}A_{\xi},~N_{\Lambda}\xi\times A_{\xi}\right\}\subset N_{\Lambda}\mathbb S^2\cap \mathbb N^3\,
\end{equation}
for every $\xi \in \Lambda_{\alpha}$.

Let $\Phi:\mathbb R^2\rightarrow \mathbb R^2$ be a smooth function with support contained in a ball of radius $1$. Moreover, suppose $\Phi$ is normalized such that if $\phi=-\Delta \Phi$ then
\begin{equation}\label{e:phi_normalize}
 \frac{1}{4\pi^2}\int \phi^2(x,y)\,dxdy=1\,.
\end{equation}
We remark that by definition $\phi$ has mean zero. Define $\psi:\mathbb R\rightarrow \mathbb R$ to be a smooth, mean zero function with support in the ball of radius $1$ satisfying
\begin{equation}\label{e:psi_normalize}
\frac{1}{2\pi}\int \psi^2(z)\,dz=1\,.
\end{equation}
Let $\phi_{\ell_{\perp}}$, $\Phi_{\ell_{\perp}}$ and $\psi_{\ell_{\|}}$ be the rescalings
\begin{align*}
\phi_{\ell_{\perp}}(x,y):=\frac{\phi\left(\frac{x}{\ell_{\perp}},\frac{y}{\ell_{\perp}}\right)}{\ell_{\perp}},\quad \Phi_{\ell_{\perp}}(x,y):=\frac{\Phi\left(\frac{x}{\ell_{\perp}},\frac{y}{\ell_{\perp}}\right)}{\ell_{\perp}}\quad\mbox{and}\quad 
\psi_{\ell_{\|}}(z):=\frac{\psi\left(\frac{z}{\ell_{\|}}\right)}{\ell_{\|}^{\sfrac12}}
\end{align*}
so that $\phi_{\ell_{\perp}}=-\ell_{\perp}^2\Delta \Phi_{\ell_{\perp}}$,
where we will assume $\ell_{\perp},\ell_{\|}>0$  to be such that
\[\ell_{\perp}\ll \ell_{\|}\ll 1\,.\]
By an abuse of notation, let us periodize $\Phi_{\ell_{\perp}}$ and $\psi_{\ell_{\|}}$ so that the functions are treated as functions defined on $\mathbb T^2$ and $\mathbb T$ respectively. For a large \emph{real number} $\lambda$ such that $\lambda\ell_{\perp}\in \mathbb N$, we define $V_{\xi,\ell_{\perp},\ell_{\|},\lambda, \mu}:\mathbb T^3\times \mathbb R\rightarrow \mathbb R$ by
\begin{align*}
V_{(\xi)}
&:=V_{\xi,\ell_{\perp},\ell_{\|},\lambda,\mu}(x,t)  \\
&:= \frac{1}{\lambda^2 N_{\Lambda}^2}\psi_{\ell_{\|}}(N_{\Lambda}\ell_{\perp}\lambda (x\cdot \xi+\mu t))\Phi_{\ell_{\perp}}(N_{\Lambda}\ell_{\perp}\lambda (x-\alpha_{\xi})\cdot A_\xi,N_{\Lambda}\ell_{\perp}\lambda (x-\alpha_{\xi})\cdot (\xi\times A_{\xi}))\xi\,. 
\end{align*}
where here $\alpha_{\xi}\in\mathbb R^3$ are shifts that ensure that the set of functions $\{V_{\xi,\ell_{\perp},\ell_{\|},\lambda,\mu}\}_{\xi}$ have mutually disjoint support. In order for such shifts $\alpha_{\xi}$ to exist, we require that $\ell_{\perp}$ to be sufficiently small, depending on the finite sets $\Lambda_{\alpha}$. 

Our {\em intermittent jet} is then defined to be 
\begin{align}
W_{(\xi)}
&:=W_{\xi,\ell_{\perp},\ell_{\|},\lambda,\mu}(x,t) \notag\\
&:= \psi_{\ell_{\|}}(N_{\Lambda}\ell_{\perp}\lambda (x\cdot \xi+\mu t))\phi_{\ell_{\perp}}(N_{\Lambda}\ell_{\perp}\lambda (x-\alpha_{\xi})\cdot A_\xi,N_{\Lambda}\ell_{\perp}\lambda (x-\alpha_\xi)\cdot (\xi\times A_{\xi}))\xi\,.
\label{eq:jet}
\end{align}
From the definition, using \eqref{e:Mignolet} and $\ell_\perp \lambda \in \N$, we have that $W_{(\xi)}$ has zero mean, and $W_{(\xi)}$ is $\left(\sfrac{\T}{\ell_{\perp} \lambda}\right)^3$-periodic. Moreover, by our choice of $\alpha_\xi$, we have that the $W_{(\xi)}$ have mutually disjoint support, i.e.\
\begin{align}
W_{(\xi)} \otimes W_{(\xi')} \equiv 0 \qquad \mbox{whenever} \qquad \xi \neq \xi' \in \cup_{\alpha \in \{1,2\}} \Lambda_\alpha\,.
\label{eq:Messi}
\end{align}
Note that the intermittent jets $W_{(\xi)}$ are not divergence free, however assuming $\ell_{\perp}\ll \ell_{\|}$ then they be corrected by a small term, such that the sum with the corrector is divergence free. To see this, let us adopt the shorthand notation
\begin{align*}
\psi_{(\xi)}&:=\psi_{\xi,\ell_{\perp},\ell_{\|},\lambda,\mu}:=\psi_{\ell_{\|}}(N_{\Lambda}\ell_{\perp}\lambda (x\cdot \xi+\mu t)),\\
\Phi_{(\xi)}&:=\Phi_{\xi,\ell_{\perp},\lambda,\mu}:=\Phi_{\ell_{\perp}}(N_{\Lambda}\ell_{\perp}\lambda (x-\alpha_{\xi})\cdot A_\xi,N_{\Lambda}\ell_{\perp}\lambda (x-\alpha_\xi) \cdot (\xi\times A_{\xi}))\\
\phi_{(\xi)}&:=\phi_{\xi,\ell_{\perp},\lambda,\mu}:=\phi_{\ell_{\perp}}(N_{\Lambda}\ell_{\perp}\lambda (x-\alpha_{\xi})\cdot A_\xi,N_{\Lambda}\ell_{\perp}\lambda (x-\alpha_\xi)\cdot (\xi\times A_{\xi}))\,.
\end{align*}
and compute
\begin{align}\label{e:div_correct_id}
\curl \curl V_{(\xi)}= W_{(\xi)}+\underbrace{\frac{1}{\lambda^2 N_{\Lambda}^2}\curl\left(\Phi_{(\xi)}\curl\left(\psi_{(\xi)}\xi\right)\right)}_{\equiv 0}+\underbrace{\frac{1}{\lambda^2 N_{\Lambda}^2}
\nabla\psi_{(\xi)}\times \curl\left(\Phi_{(\xi)}\xi\right)}_{W^{(c)}_{(\xi)}}\,.
\end{align}
Thus
\[\div\left(W_{(\xi)}+W^{(c)}_{(\xi)}\right)\equiv 0\,.\]
Moreover, so long as $\ell_{\perp}\ll \ell_{\|}$ then $W^{(c)}_{(\xi)}$ is comparatively small compared to $W_{(\xi)}$. Observe that as a consequence of the normalizations \eqref{e:phi_normalize} and \eqref{e:psi_normalize} we have
\begin{align*}
\fint_{\T^3} W_{(\xi)}(x) \otimes W_{(\xi)}(x)\,dx=\xi\otimes \xi\,. 
\end{align*}
We also note that by definition $W_{(\xi)}$ is mean zero. As a consequence, using Lemma \ref{l:linear_algebra} we have
\begin{equation}\label{e:Reynolds_cancellation}
 \sum_{\xi\in \Lambda_\alpha}\gamma^2_{\xi}(R)\fint_{\T^3}  W_{(\xi)}(x)\otimes   W_{(\xi)}(x)\,dx=R\,,
\end{equation}
for every symmetric matrix $R$ satisfying $\abs{R-\Id}\leq \sfrac12$.
By scaling and Fubini, we have the estimates
\begin{align}
\norm{\nabla^N\partial_t^M\psi_{(\xi)}}_{L^{p}}&\les \ell_{\|}^{\sfrac{1}{p}-\sfrac12}\left(\frac{\ell_{\perp}\lambda}{\ell_{\|}}\right)^N\left(\frac{\ell_{\perp}\lambda\mu}{\ell_{\|}}\right)^M\label{e:phi_Lp_bnd}\\
\norm{\nabla^N\phi_{(\xi)}}_{L^{p}}+\norm{\nabla^N\Phi_{(\xi)}}_{L^{p}}&\les \ell_{\perp}^{\sfrac{2}{p}-1}\lambda^N\label{e:psi_Lp_bnd}\\
\norm{\nabla^N\partial_t^M W_{(\xi)}}_{L^{p}}+\lambda^2\norm{\nabla^N\partial_t^M V_{(\xi)}}_{L^{p}}&\les \ell_{\perp}^{\sfrac{2}{p}-1}\ell_{\|}^{\sfrac{1}{p}-\sfrac12}\lambda^N\left(\frac{\ell_{\perp}\lambda\mu}{\ell_{\|}}\right)^M\,,\label{e:W_Lp_bnd}
\end{align}
where again here we have assumed 
\[\ell_{\|}^{-1}\ll \ell_{\perp}^{-1}\ll \lambda\,.\]

Finally, we note the essential identity
\begin{align}\label{eq:useful:2}
\div\left(W_{(\xi)}\otimes W_{(\xi)}\right) =2(W_{(\xi)}\cdot \nabla \psi_{(\xi)})\phi_{(\xi)}\xi=\frac{1}{\mu}\phi^2_{(\xi)}\partial_t \psi^2_{(\xi)}\xi\,,
\end{align}
which follows from the fact that by construction we have that $W_{(\xi)}$ is a scalar multiple of $\xi$, 
\[
(\xi \cdot \nabla) \psi_{(\xi)} = \frac{1}{\mu} \partial_t \psi_{(\xi)}\,,
\]
and $\phi_{(\xi)}$ is time-independent. 

\subsection{The perturbation}\label{sec:perturbation}
In this section we will construct the perturbation $w_{q+1}$.

\subsubsection{Stress cutoffs}
Because the Reynolds stress $\RRO$ is not spatially homogenous, we  introduce stress cutoff functions.
We let $0 \leq \tilde \chi_0, \tilde \chi \leq 1$  be bump functions adapted to the intervals $[0,4]$ and $[1/4,4]$, such that together they form a partition of unity:
\begin{align}
\tilde \chi_{0}^2(y) + \sum_{i\geq 1}  \tilde \chi^2_{i}(y) \equiv 1, \quad \mbox{where} \quad \tilde \chi_i (y) = \tilde \chi(4^{-i} y),
\label{eq:tilde:partition}
\end{align}
for any $y>0$.
We then define 
\begin{align}
\chi_{(i)}(x,t)=\chi_{i,q+1}(x,t) = \tilde \chi_{i}\left( \left\langle \frac{\RRO(x,t)}{\lambda_q^{-\sfrac{\eps_R}{4}}\delta_{q+1}}\right\rangle\right) 
\label{eq:chi:i:def}
\end{align}
for all $i\geq 0$. 
Here and throughout the paper we use the notation $\langle A \rangle = (1+ |A|^2)^{1/2}$ where $|A|$ denotes the Euclidean norm of the matrix $A$.
By definition the cutoffs $\chi_{(i)}$ form a partition of unity
\begin{align}
\sum_{i\geq 0} \chi_{(i)}^2 \equiv 1
\label{eq:chi:partition} 
\end{align}
and we will show in Lemma~\ref{lem:max:i} below that there exists an index $i_{\rm max} = i_{\rm max}(q)$, such that $\chi_{(i)} \equiv 0$ for all $i > i_{\rm max}$, and moreover that $4^{i_{\rm max}} \les \tau_{q+1}^{-1}$.

\subsubsection{The definition of the velocity increment}

Recall that from~Lemma~\ref{l:linear_algebra}, the functions $\gamma_{(\xi)}$ are well-defined and smooth in the $\sfrac 12$ neighborhood of the identity matrix. In view of \eqref{eq:chi:i:def}, this motivates introducing the parameters $\rho_i$  by 
\begin{align}\label{e:rho_i_def}
\rho_i:= \lambda_q^{-\sfrac{\eps_R}{4}} \delta_{q+1} 4^{i+2}, \qquad \mbox{for all} \qquad i \geq 0 \, ,
\end{align}
which have the property that 
\begin{align}
\frac{|\RRO|}{\rho_i} \leq \frac 14
\quad \mbox{on the support of} \quad  \chi_{(i)}\, , \quad \mbox{for all} \quad i \geq 0\,.\notag
\end{align}
As such, for $i\geq 0$ we define the coefficient function $a_{\xi,i,q+1}$ by
\begin{equation}
a_{(\xi)}:= a_{\xi,i,q+1}(x,t):=\theta(t) \;\rho_i^{\sfrac 12} \; \chi_{i,q+1}(x,t) \; \gamma_{(\xi)}\left(\Id - \frac{\RRO(x,t)}{ \rho_i}\right) \, ,
\label{eq:a:oxi:def}
\end{equation}
where $\theta \colon [0,T] \to [0,1]$ is a  smooth temporal cut-off function with the following properties: 
\begin{enumerate}[label=(\roman*)]
\item $\theta(t)=1$ for all $t$ such that $ {\rm dist}(t,\GG^{(q+1)}) \geq 2 \tau_{q+1}$.
\item $\theta(t)=0$ for all $t$ such that $ {\rm dist}(t,\GG^{(q+1)}) \leq  \tau_{q+1}$.
\item $\norm{\theta}_{C^M}\les \tau_{q+1}^{-M}$, where the implicit constant depends only on $M$.
\end{enumerate}
To see that a choice for $\theta$ with property (iii) holding is possible, recall from \eqref{eq:B:q+1:def} that the bad set $\BB^{(q+1)}$ consists of a finite disjoint union of intervals of length $5 \tau_{q+1}$. 
From the first property above and  \eqref{eq:bar:Rq:supp}, we conclude that 
\begin{align}
t\in \supp(\mathring{\overline{R}}_q)   \quad\mbox{implies} \quad\theta(t)=1\,. 
\label{eq:R_support}
\end{align}
From the second property above we further obtain that
\begin{align}\label{e:a_support}
t\in \supp(\theta)\supset \supp(a_{(\xi)})   \quad\mbox{implies} \quad {\rm dist}(t, \GG^{(q+1)}) >  \tau_{q+1} \, .
\end{align}

We note that as a consequence of \eqref{e:Reynolds_cancellation},  \eqref{eq:chi:partition}, \eqref{eq:a:oxi:def}, and \eqref{eq:R_support} we have 
\begin{align}
\sum_{i\geq 0}\sum_{\xi\in \Lambda_{(i)}}a_{(\xi)}^2 \fint_{\mathbb T^3} \mathbb W_{(\xi)}\otimes \mathbb W_{(\xi)}  dx=
\theta^2  \sum_{i\geq 0} \rho_i \chi_{(i)}^2\Id - \RRO\,,
\label{e:WW_id}
\end{align}
which justifies the definition of the amplitude functions $a_{(\xi)}$. Note that $\theta = 1$ on the support of $\chi_{(i)}$ for any $i \geq 1$. 

By a slight abuse of notation, let us now fix $\lambda, \sigma, \ell_{\perp},\ell_{\|}$, and $\mu$ for the short hand notation $W_{(\xi)}$, $V_{(\xi)}$, $\Phi_{(\xi)}$, $\phi_{(\xi)}$ and $\psi_{(\xi)}$ introduced in Section~\ref{s:intermittent}:
\begin{gather*}
W_{(\xi)}:=W_{\xi,\ell_{\perp},\ell_{\|},\lambda_{q+1},\mu}, \quad V_{(\xi)}:=V_{\xi,\ell_{\perp},\ell_{\|},\lambda_{q+1},\mu},\\
\psi_{(\xi)}:=\psi_{\xi,\ell_{\perp},\ell_{\|},\lambda_{q+1},\mu},\quad
\Phi_{(\xi)}:=\Phi_{\xi,\ell_{\perp},\lambda_{q+1},\mu},\quad\mbox{and}\quad
\phi_{(\xi)}:=\phi_{\xi,\ell_{\perp},\lambda_{q+1},\mu}
\end{gather*}
where  $\ell_{\perp},\ell_{\|}$, and $\mu$ are defined in \eqref{e:param_def}.  Importantly, we have from \eqref{e:integer_constraint} that $\lambda_{q+1}\ell_{\perp}\in\mathbb N$ which ensures the periodicity of $W_{(\xi)}$, $V_{(\xi)}$, $\Phi_{(\xi)}$, $\phi_{(\xi)}$ and $\psi_{(\xi)}$. Observe that as a consequence of our parameter choices we have the useful inequality 
\begin{align}
\label{eq:useful}
\mu^{-1} \ell_\perp^{-1} \ell_{\|}^{-\sfrac 12}  =  \lambda_{q+1}^{-\frac{5-4\alpha}{8}}  \ll 1 \, ,
\end{align}
for all $\alpha < \sfrac 54$.  

The {\em principal part of $w_{q+1}$} is defined as
\begin{equation}
w^{(p)}_{q+1}:=\sum_{i}\sum_{\xi\in \Lambda_{(i)}}a_{(\xi)} \;  W_{(\xi)}\,,
\label{eq:w:q+1:p:def}
\end{equation}
where the sum is over $0 \leq i \leq i_{\rm max}(q)$. Here we write $\Lambda_{(i)}= \Lambda_{i {\rm \, mod\, } 2}$. Note that $\abs{i-j}\geq 2$ implies $\chi_i \chi_{j} \equiv 0$ and $\xi\neq \xi'$ implies $W_{(\xi)}\otimes W_{(\xi')}\equiv 0$. This implies that the summands in \eqref{eq:w:q+1:p:def} have mutually disjoint supports. In order to fix the fact that $w_{q+1}^{(p)}$ is not divergence free, we define an {\em incompressibility corrector} by
\begin{align}
w^{(c)}_{q+1}:=&\sum_{i}\sum_{\xi\in \Lambda_{(i)}} \curl\left(\nabla a_{(\xi)}\times V_{(\xi)}
\right)+\frac{1}{\lambda_{q+1}^2 N_{\Lambda}^2}\nabla\left(a_{(\xi)}\psi_{(\xi)}\right)\times \curl\left(\Phi_{(\xi)}\xi\right)\,,
\label{eq:w:q+1:c:def}
\end{align}
so that by a similar formula to \eqref{e:div_correct_id}
\begin{align}\label{e:curl_form}
w_{q+1}^{(p)}+w_{q+1}^{(c)}=\sum_{i}\sum_{\xi\in \Lambda_{(i)}} \curl\curl(a_{(\xi)}V_{(\xi)})\,,
\end{align}
and thus
\begin{align*}
\div\left(w_{q+1}^{(p)}+w_{q+1}^{(c)}\right)\equiv 0\,.
\end{align*}

In addition to the incompressibility corrector $w_{q+1}^{(c)}$, we introduce a  {\em temporal corrector} $w_{q+1}^{(t)}$, which is defined by
\begin{equation}\label{e:temporal_corrector}
w^{(t)}_{q+1}:=-\frac{1}{ \mu}\sum_{i}\sum_{\xi\in \Lambda_{(i)}}\mathbb P_{H}\mathbb P_{\neq 0}\left(a_{(\xi)}^2\phi^2_{(\xi)} \psi^2_{(\xi)}\xi \right)\, .
\end{equation}

Finally, we define the velocity increment $w_{q+1}$ by
\begin{equation}\label{eq:w:q+1:def}
w_{q+1}:=w_{q+1}^{(p)}+w_{q+1}^{(c)}+w_{q+1}^{(t)}\,,
\end{equation}
which is by construction mean zero and divergence-free. 
The new velocity field $v_{q+1}$ is then defined as
\begin{equation}\label{eq:v:q+1:def}
v_{q+1} = \vo + w_{q+1} \, .
\end{equation}

Observe that as a consequence of \eqref{e:a_support} we have 
\begin{equation}\label{e:w_support}
t\in \supp(w_{q+1}) \quad\mbox{implies} \quad {\rm dist}(t, \GG^{(q+1)}) >  \tau_{q+1}\,.
\end{equation}
Hence $v_{q+1}=\vo$ on $\GG^{(q+1)}$ which we recall was required in Section \ref{ss:proof_prop} to deduce property \ref{eq:cond:v} of Section \ref{sec:inductive} for $\GG^{(q+1)}$. Moreover, property \ref{eq:cond:vi} also follows as a consequence of \eqref{e:w_support} and \eqref{eq:bar:Rq:supp}.

\subsubsection{Estimates of the perturbation}

This section closely mirrors~\cite[Section 4.4]{BV}, and thus we omit most details where the estimates/proofs are mutatis-mutandis those from~\cite{BV}. There is an analogy between the mollification parameter $\ell$ in \cite{BV} and the time-scale $\tau_{q+1}$ in this paper, in view of parabolic smoothing.

First, similarly to~\cite[Lemma 4.1 and 4.2]{BV} we state a useful lemma concerning the cutoffs  $\chi_{(i)}$ defined in \eqref{eq:chi:i:def}, summarizing their size and regularity: 
\begin{lemma}
\label{lem:max:i}
For $q\geq 0$, there exists $i_{\rm max}(q) \geq 0$ such that
\begin{align}
\chi_{(i)} \equiv 0 \quad \mbox{for all} \quad i > i_{\rm max} \, .
\label{eq:chi:i:vanish}
\end{align}
Moreover, for all $0 \leq i\leq i_{\rm max}$ we have that  
\begin{align}
\rho_i\leq  \lambda_1^{\beta}  4^{i_{\rm max}} \leq \tau_{q+1}^{-2} \, ,
\label{eq:i:max:bound}
\end{align}
and the bound
\begin{equation}
\sum_{i = 0}^{i_{\rm max}} \rho_i^{\sfrac12}2^{-i}\leq  \lambda_{q}^{-\sfrac{\eps_R}{16}} \delta_{q+1}^{\sfrac 12}   \,  \label{e:rho_sum}
\end{equation}
holds.
Additionally, for $0 \leq i \leq i_{\rm max}$ we have
\begin{align} 
\norm{ \chi_{(i)} }_{L^2} &\les  2^{-i} \, ,\label{eq:psi:i:L2}  \\
 \norm{ \chi_{(i)}}_{C^{N}_{x,t}} & \les  \tau_{q+1}^{-3N}  \, ,\label{eq:psi:i:CN}
\end{align}
for all $N\geq 1$, where the implicit constant only depends on $N$.
\end{lemma}
\begin{proof}[Proof of Lemma~\ref{lem:max:i}]
The existence of $i_{\rm max}$ follows as a consequence of the bound
\begin{align}
\norm{\RRO}_{L^\infty}   \leq \lambda_q^{8} \, .
\label{eq:Rq:Linfinity}
\end{align}
The bound \eqref{eq:Rq:Linfinity} follows from \eqref{eq:Rq:1}--\eqref{eq:Rq:2} and the Gagliardo-Nirenberg inequality $\norm{f}_{L^\infty} \les   \norm{f}_{L^1}^{\sfrac 13} \norm{f}_{\dot{H}^3}^{\sfrac 23}$, which holds for any zero-mean periodic function $f \in H^3$, and  the definition of $\tau_{q+1}$ in \eqref{eq:tau:q:def}, and the fact that  $\eps_R \leq 1$. Indeed, we have
\begin{align*}
\norm{\RRO}_{L^\infty}   \les  \left(\lambda_q^{-\sfrac{\eps_R}{4}} \delta_{q+1}\right)^{\sfrac 13} \left( \tau_{q+1}^{-1} \lambda_q^{4}\right)^{\sfrac 23} = \left(\lambda_q^{-\sfrac{\eps_R}{4}} \delta_{q+1}\right)^{\sfrac 13} \left( \lambda_q^{7+ \sfrac{\eps_R}{4}} \delta_{q+1}^{-\sfrac 12} \lambda_q^{4}\right)^{\sfrac 23} 
 =   \lambda_q^{\sfrac{\eps_R}{12} + \sfrac{22}{3} }
\end{align*}
and the remaining power  of $\lambda_q^{-\sfrac 23}$ may be used to absorb $\lambda_q^{\sfrac{\eps_R}{12}} $ and the implicit universal constant.

The first bound expressed in \eqref{eq:i:max:bound} follows from the definition of $\rho_i$ in \eqref{e:rho_i_def}, the fact that by \eqref{eq:delta:q:def} we have $\delta_{q+1} \leq \lambda_1^\beta$, and the fact that $a$ may be chosen sufficiently large to ensure that $4^5 \lambda_q^{-\sfrac{\eps_R}{4}} \leq 1$.
 Next, we note that in view of the definition of $\chi_{(i)}$, for any $i \geq 1$,  if $(x,t)$ is such that 
\begin{align*}
\left\langle   \lambda_q^{\sfrac{\eps_R}{4}} \delta_{q+1}^{-1}  \RRO(x,t) \right \rangle < 4^{i-1}
\end{align*}
then $\chi_{(i)}(x,t) = 0$. Therefore, by the bound \eqref{eq:Rq:Linfinity} and the fact that $\beta b \leq \sfrac 14$, if $i\geq 1$ is such that 
\begin{align*}
\left\langle  \lambda_q^{\sfrac{\eps_R}{4}} \delta_{q+1}^{-1} \lambda_q^{8}  \right \rangle \leq  \lambda_q^{9} < 4^{i-1}
\end{align*}
then $\chi_{(i)} \equiv 0$.  Therefore, in view of the parameter inequality
\[
 \lambda_q^{9} \leq 4^{-2} \lambda_1^{-\beta} \tau_{q+1}^{-2} \, ,
\]
which holds in view of \eqref{eq:tau:q:def} and the fact that $\beta b \leq \sfrac 14$, upon taking $a$ to be sufficiently large, we may thus define 
\begin{align*}
i_{\rm max}(q) = \max \left\{i\geq 0 \colon  \lambda_1^\beta \; 4^{i}  \leq   \tau_{q+1}^{-2}\right\}. 
\end{align*}
With this choice of $i_{\rm max}$ it from the above argument it follows that \eqref{eq:chi:i:vanish} holds.
The bound on $i_{\rm max}$ claimed in the second inequality in \eqref{eq:i:max:bound} then follows from the above definition.

The bound \eqref{e:rho_sum} follows from the second estimate in \eqref{eq:i:max:bound} which gives an upper bound on $i_{\rm max}$, the definition \eqref{e:rho_i_def},   and using  that $\lambda_q^{-\sfrac{\eps_R}{16}} \log_4 (\tau_{q+1}^{-2}) \leq 8\lambda_q^{-\sfrac{\eps_R}{16}} \log_4 (\lambda_q)$ can be made arbitrarily small if $a$ is chosen sufficiently large, depending on $\eps_R$. 

For $i= 0,1$, the bound \eqref{eq:psi:i:L2} follows from the fact that $\tilde \chi_0, \tilde \chi \leq 1$.  For $i \geq 2$, we appeal to the definition of $\chi_{(i)}$, Chebyshev's inequality, and the $L^1$ estimate on $\RRO$ in \eqref{eq:Rq:1}, to obtain that $\norm{\chi_{(i)}}_{L^1} \les 4^{-i}$. The bound \eqref{eq:psi:i:L2} follows by interpolation.

 Estimate \eqref{eq:psi:i:CN} is a consequence of \eqref{eq:Rq:2}   and~\cite[Proposition C.1]{BDLISZ15}, applied to  the composition with the smooth functions $\gamma_\xi(\cdot)$ and $\langle \cdot \rangle = \sqrt{1+ (\cdot)^2}$.  Indeed, for any $i \geq 0$ we obtain
\begin{align*}
\norm{ \chi_{(i)}}_{C^N_{x,t}} 
&\les \norm{\langle  \lambda_q^{\sfrac{\eps_R}{4}} \delta_{q+1}^{-1} \RRO\rangle}_{C^N_{x,t}} + \norm{\langle  \lambda_q^{\sfrac{\eps_R}{4}} \delta_{q+1}^{-1} \RRO\rangle}_{C^1_{x,t}}^N\\
&\les 1 + \lambda_q^{\sfrac{\eps_R}{4}} \delta_{q+1}^{-1} \norm{  \RRO }_{C^N_{x,t}} + \left( \lambda_q^{\sfrac{\eps_R}{4}} \delta_{q+1}^{-1} \norm{  \RRO }_{C^1_{x,t}}\right)^N \\
&\les 1 + \lambda_q^{\sfrac{\eps_R}{4}} \delta_{q+1}^{-1} \tau_{q+1}^{-N-1} \lambda_q^{4} + \left( \lambda_q^{\sfrac{\eps_R}{4}} \delta_{q+1}^{-1}  \tau_{q+1}^{-2} \lambda_q^{4} \right)^N \\
&\les 1+ \tau_{q+1}^{-N-2} + \tau_{q+1}^{-3N} \les \tau_{q+1}^{-3N}.
\end{align*}
Here we have used \eqref{eq:tau:q:def} to show that $\tau_{q+1} \leq 1$, and that $\lambda_q^{\sfrac{\eps_R}{4}} \delta_{q+1}^{-1} \lambda_q^{4} \leq   \lambda_q^{7 + \sfrac{\eps_R}{4}} \delta_{q+1}^{-\sfrac 12} =  \tau_{q+1}^{-1}$.
\end{proof}

Next, we recall from \cite[Lemma 4.3]{BV} the following bounds on the coefficients $a_{(\xi)}$.
\begin{lemma}
\label{lem:many:bounds}
The bounds
\begin{align}
\norm{a_{(\xi)}}_{L^2} &\les \rho_i^{\sfrac 12} 2^{-i}\les \delta_{q+1}^{\sfrac 12}  \, ,\label{e:a_est_L2} \\
\norm{a_{(\xi)}}_{L^\infty} &\les \rho_i^{\sfrac 12}  \les \delta_{q+1}^{\sfrac 12} 2^i \, ,\label{e:a_est} \\
\norm{a_{(\xi)}}_{C^{N}_{x,t}} &\les  \tau_{q+1}^{-3N- 1}  \label{e:a_est_CN}
\end{align}
hold for all $0 \leq i \leq i_{\rm max}$ and $N \geq 1$. 
\end{lemma}
\begin{proof}[Proof of Lemma~\ref{lem:many:bounds}]
The bound \eqref{e:a_est} follows directly from the definitions \eqref{eq:a:oxi:def} and \eqref{e:rho_i_def},  the boundedness of the functions $\theta$, $\chi_{(i)}$, and $\gamma_{(\xi)}$. Using also \eqref{eq:psi:i:L2}, the estimate \eqref{e:a_est_L2} follows similarly.
In order to prove \eqref{e:a_est_CN}, we apply derivatives to \eqref{eq:a:oxi:def}, use the bounds previously established in Lemma~\ref{lem:max:i}, use~\cite[Proposition C.1]{BDLISZ15} and the bound \eqref{eq:Rq:2} for $\RRO$, combined with property $ \norm{\theta}_{C^M}\les \tau_{q+1}^{-M}$. The additional factor of $\tau_{q+1}^{-1}$ when compared to \eqref{eq:psi:i:CN} is to absorb the factor  of $\rho_i^{1/2}$ via \eqref{eq:i:max:bound}.   
\end{proof}

As a consequence of Lemma~\ref{lem:many:bounds} and the definitions \eqref{eq:w:q+1:p:def}, \eqref{eq:w:q+1:c:def}, and \eqref{e:temporal_corrector}, we obtain the following bounds:

\begin{proposition}
\label{prop:perturbation}
The principal part and of the velocity perturbation, the incompressibility, and the temporal correctors obey the bounds 
\begin{align}
\norm{w_{q+1}^{(p)}}_{L^2}  
&\leq \frac{1}{2} \delta_{q+1}^{\sfrac 12} \label{e:wp_est}\\
\norm{w_{q+1}^{(p)}}_{W^{N,p}}
&\les \tau_{q+1}^{-2} \ell_{\perp}^{\sfrac{2}{p}-1}\ell_{\|}^{\sfrac{1}{p}-\sfrac12} \lambda_{q+1}^N\label{e:wp_N_est}\\
\norm{w_{q+1}^{(c)}}_{W^{N,p}}+\norm{w_{q+1}^{(t)}}_{W^{N,p}}
&\les   \mu^{-1} \tau_{q+1}^{-3}     \ell_{\perp}^{\sfrac{2}{p}-2}\ell_{\|}^{\sfrac{1}{p}-1} \lambda_{q+1}^N \les \lambda_{q+1}^{- \frac{5-4\alpha}{16}}   \ell_{\perp}^{\sfrac{2}{p}-1}\ell_{\|}^{\sfrac{1}{p}-\sfrac12} \lambda_{q+1}^N
\label{e:wct_N_est}
\end{align}
for $N\in\mathbb N$, and $p>1$.  
\end{proposition}
From the second estimate in \eqref{e:wct_N_est} it is clear that the incompressibility and temporal correctors obey better estimates than the principal corrector. 

In order to establish the bound \eqref{e:wp_est}, it is essential to use the fact that $a_{(\xi)}$ oscillates at a frequency which is much smaller than that of $W_{(\xi)}$, which allows us to appeal to the $L^p$ de-correlation lemma~\cite[Lemma 3.6]{BV}, which we recall here for convenience:
\begin{lemma}
\label{lem:Lp:independence}
Fix integers $M,\kappa,\lambda \geq 1$ such that
\begin{align}
\label{eq:Lp:independence:assume}
\frac{2\pi\sqrt{3} \lambda}{\kappa} \leq \frac{1}{3} \qquad \mbox{and} \qquad \lambda^{4} \frac{(2\pi\sqrt{3} \lambda)^M}{\kappa^M}  \leq 1\,.
\end{align}
Let $p \in \{1,2\}$, and let $f$ be a $\T^3$-periodic function such that there exists a constants $C_f$ such that
\begin{align*}
\|D^j f\|_{L^p} \leq C_f\lambda^j \,,
\end{align*}
for  all $1 \leq j \leq M+4$. 
In addition, let $g$ be a $(\sfrac{\T}{\kappa})^{3}$-periodic function. Then we have that 
\begin{align*}
 \|f g \|_{L^p} \lesssim C_f \|g\|_{L^p} \,,
\end{align*}
holds, where the implicit constant is universal.
\end{lemma}
Using Lemma~\ref{lem:Lp:independence}, the proofs of  bounds \eqref{e:wp_est}--\eqref{e:wct_N_est} follow in the same spirit as that of \cite[Proposition 4.5]{BV}. 
\begin{proof}[Proof of Proposition~\ref{prop:perturbation}]
In order to prove \eqref{e:wp_est}, 
we use \eqref{e:a_est_L2} when $N=0$, and \eqref{e:a_est_CN} with $\lambda_q^{\sfrac{\eps_R}{8}} \delta_{q+1}^{-\sfrac 12} \leq \tau_{q+1}^{-1}$ for $N\geq 1$, to conclude that
\begin{align}
\norm{D^N a_{(\xi) }}_{L^2} \les\rho_i^{\sfrac 12} 2^{-i}   \tau_{q+1}^{-5N}\, ,
\label{eq:a:HN}
\end{align}
where the implicit constant depends only on $N$. 
Since $\mathbb W_{(\xi)}$ is $(\sfrac{\T}{\lambda_{q+1}\ell_{\perp}})^{3}$ periodic, in order to apply Lemma~\ref{lem:Lp:independence} with $\lambda = \tau_{q+1}^{-5}$ and $\kappa = \lambda_{q+1} \ell_\perp$, we first note that by \eqref{eq:tau:q:def} and \eqref{e:param_def}, we have 
\begin{align}
2 \pi \sqrt{3} \lambda \kappa^{-1} = 2 \pi \sqrt{3} \tau_{q+1}^{-5} \lambda_{q+1}^{-1} \ell_\perp^{-1} = 2 \pi \sqrt{3} \lambda_q^{35 +  \frac{5 \eps_R}{4}} \delta_{q+1}^{-\sfrac52} \lambda_{q+1}^{- \frac{5(5-4\alpha)}{24}} \leq \lambda_q^{36}  \lambda_{q+1}^{- \frac{5(5-4\alpha)}{24} + 5\beta} \leq \lambda_{q+1}^{-\frac{5-4\alpha}{6}}\, ,
\label{eq:separation}
\end{align}
by using that $\beta$ is sufficiently small and $b$ is sufficiently large, depending on $\alpha$. For instance, we may take 
\begin{align}
5 \beta \leq \frac{5-4\alpha}{50}  \qquad \mbox{and}  \qquad  \frac{36}{b}  \leq \frac{5-4\alpha}{50}\, .
\label{eq:new:beta:b}
\end{align}
In \eqref{eq:separation} we also used that $\lambda_1^{- \frac{15\beta}{2}} 2 \pi \sqrt{3} \leq 1$, once $a$ is chosen sufficiently large.
Therefore, after a short computation we see that the assumptions of Lemma~\ref{lem:Lp:independence} hold with the aforementioned $\kappa$ and $\lambda$, with {$M = 4$} in \eqref{eq:Lp:independence:assume}. Therefore, we only care about $N\leq 4$ in \eqref{eq:a:HN}, which also fixes the implicit constant in this inequality, and we may thus take $C_f$ to be proportional to $\rho_i^{\sfrac 12} 2^{-i}$. It thus follows from Lemma~\ref{lem:Lp:independence} and estimate \eqref{e:W_Lp_bnd}  with $M=N=0$ and $p=2$ that 
\begin{align}
\norm{a_{(\xi)} W_{(\xi)} }_{L^2}&
\lesssim \rho_i^{\sfrac 12} 2^{-i} \norm{ W_{(\xi)}}_{L^2} 
\lesssim \rho_i^{\sfrac 12} 2^{-i} \,.\notag
\end{align}
Upon summing over $i \in \{0,\ldots,i_{\rm max}\}$, and appealing to \eqref{e:rho_sum}, we obtain that 
\begin{align}
\norm{w_{q+1}^{(p)}}_{L^2} \les \delta_{q+1}^{\sfrac 12} \lambda_q^{-\sfrac{\eps_R}{16}} \leq \frac 12 \delta_{q+1}^{\sfrac 12}\,,\notag
\end{align}
by using the small negative power of $\lambda_q$ to absorb the implicit constants in the first inequality.

Consider the estimate \eqref{e:wp_N_est}. Observe that by definition~\eqref{eq:w:q+1:p:def}, estimate \eqref{e:W_Lp_bnd}  with $M=0$, and the bound \eqref{e:a_est_CN}, we have 
\begin{align}
\norm{w_{q+1}^{(p)}}_{W^{N,p}}
&\lesssim \sum_{i}\sum_{\xi\in \Lambda_{(i)}} \sum_{N'=0}^{N} \norm{a_{(\xi)}}_{C^{N-N'}}\norm{ W_{(\xi)}}_{W^{N',p}}\notag\\
&\lesssim \sum_{i=0}^{i_{\rm max}} \sum_{\xi\in \Lambda_{(i)}} \sum_{N'=0}^{N} \tau_{q+1}^{-3(N-N')-1}\ell_{\perp}^{\sfrac{2}{p}-1}\ell_{\|}^{\sfrac{1}{p}-\sfrac12}\lambda_{q+1}^{N'}
\notag\\
&\les \tau_{q+1}^{-2}\ell_{\perp}^{\sfrac{2}{p}-1}\ell_{\|}^{\sfrac{1}{p}-\sfrac12}\lambda_{q+1}^N \, .
\label{eq:tilde:R:3}
\end{align}
Here  we have again used \eqref{eq:i:max:bound} in order to sum over $i$, and have used the bound $\tau_{q+1}^{-3} \leq \lambda_{q+1}$ which holds since $\beta$ is small and $b$ is large. 

For the analogous bound on $w_{q+1}^{(c)}$, by  \eqref{e:phi_Lp_bnd}--\eqref{e:W_Lp_bnd}, estimate \eqref{e:a_est_CN}, the parameter estimates $\tau_{q+1}^{-3} \leq \lambda_{q+1}$ and $\ell_\perp \leq \ell_\|$, and Fubini (recall that $\psi_{(\xi)}$ and $\Phi_{(\xi)}$ are functions of one and respectively two variables which are orthogonal to each other), we have
\begin{align*}
&\norm{ \curl\left(\nabla a_{(\xi)}\times V_{(\xi)}
\right)+\frac{1}{\lambda_{q+1}^2 N_{\Lambda}^2}\nabla\left(a_{(\xi)}\psi_{(\xi)}\right)\times \curl\left(\Phi_{(\xi)}\xi\right)}_{W^{N,p}}\\
&\les \sum_{N'=0}^{N+1} \norm{a_{(\xi)}}_{C^{N+2-N'}} \norm{V_{(\xi)}}_{W^{N',p}}  + \frac{1}{\lambda_{q+1}^2} \sum_{N'=0}^N \sum_{N''= 0}^{N-N'+1} \norm{a_{(\xi)}}_{C^{N-N'+1-N''}}  \norm{\psi_{(\xi)}}_{W^{N'',p}} \norm{\Phi_{(\xi)}}_{W^{N'+1,p}}\\
&\les \sum_{N'=0}^{N+1} \tau_{q+1}^{-3(N+2-N') - 1}\lambda_{q+1}^{N'-2} + \sum_{N'=0}^N \sum_{N''= 0}^{N-N'+1} \tau_{q+1}^{-3(N+1-N'-N'') - 1} \ell_{\|}^{\sfrac 1p - \sfrac 12} \left(\frac{\ell_\perp \lambda_{q+1}}{\ell_{\|}}\right)^{N''} \ell_\perp^{\sfrac 2p - 1} \lambda_{q+1}^{N'-1}\\
&\les \tau_{q+1}^{-1} \ell_{\|}^{\sfrac 1p - \sfrac 12} \ell_\perp^{\sfrac 2p - 1} \lambda_{q+1}^{N} \left( \tau_{q+1}^{-3} \lambda_{q+1}^{-1}    \right)  \,.
\end{align*}
Summing over $0 \leq i \leq i_{\rm \max}$ loses an additional factor of $\tau_{q+1}^{-1}$, which yields the desired bound for the first term on the left of \eqref{e:wct_N_est}.  Similarly, to estimate the summands in the definition \eqref{e:temporal_corrector} of $w_{q+1}^{(t)}$ we use \eqref{e:phi_Lp_bnd}, \eqref{e:psi_Lp_bnd}, \eqref{e:a_est_CN}, the aforementioned parameter inequalities, and Fubini to obtain
\begin{align*}
\norm{\mu^{-1} \mathbb P_{H}\mathbb P_{\neq 0}\left(a_{(\xi)}^2\phi^2_{(\xi)} \psi^2_{(\xi)}\xi \right)}_{W^{N,p}}
&\les \mu^{-1} \sum_{N'=0}^{N} \sum_{N''=0}^{N'}  \norm{a_{(\xi)}^2}_{C^{N-N'}} \norm{\phi_{(\xi)}^2}_{W^{N'',p}} \norm{\psi_{(\xi)}^2}_{W^{N'-N'',p}} \\
&\les \mu^{-1} \tau_{q+1}^{-3(N-N')-2} \ell_\perp^{\sfrac{2}{p}-2} \lambda_{q+1}^{N''} \ell_{\|}^{\sfrac{1}{p} -1} \left( \frac{\ell_\perp \lambda_{q+1}}{\ell_{\|}}\right)^{N'-N''}\\
&\les \mu^{-1} \tau_{q+1}^{-2} \ell_\perp^{\sfrac{2}{p}-2}  \ell_{\|}^{\sfrac{1}{p} -1 }  \lambda_{q+1}^{N}\, .
\end{align*}
Summing in $i$ loses a factor of $\tau_{q+1}^{-1}$ cf.~\eqref{eq:i:max:bound}, and we obtain the bound for the second term on the left of  \eqref{e:wct_N_est}.

For the proof of \eqref{e:wct_N_est}, we additionally note that  \eqref{eq:useful}, and \eqref{eq:new:beta:b}  imply the parameter inequalities  
\begin{align}
\label{eq:useful:too}
\tau_{q+1}^{-1} \leq \lambda_q^8 \leq \lambda_{q+1}^{\frac{5-4\alpha}{100}}\, , \qquad \tau_{q+1}^{-3} \lambda_{q+1}^{-1} \leq \mu^{-1} \tau_{q+1}^{-1} \ell_\perp^{-1} \ell_{\|}^{-\sfrac 12} \, , \qquad \mbox{and} \qquad \mu^{-1} \tau_{q+1}^{-3} \ell_\perp^{-1} \ell_{\|}^{-\sfrac 12} \leq \lambda_{q+1}^{- \frac{5-4\alpha}{16}} \, ,
\end{align}
which concludes the proof of the proposition.
\end{proof}

The following bound shows that \eqref{e:interative_v} holds, and collects a number of useful bounds for the cumulative velocity increment $w_{q+1}$, which in turn imply that \eqref{eq:v:L2:x} and \eqref{eq:v:H3:x} hold at level $q+1$.

\begin{proposition}
\label{p:perturbation}
The bounds
\begin{align}
\norm{w_{q+1}}_{L^2}  &\leq \frac{3}{4} \delta_{q+1}^{\sfrac12}\label{eq:w:q+1:L2}\\
\norm{v_{q+1} - v_q}_{L^2} &\leq   \delta_{q+1}^{\sfrac 12}\label{eq:increment:L2} \\
\norm{w_{q+1}}_{W^{s,p}}
&\les\tau_{q+1}^{-2}\ell_{\perp}^{\sfrac{2}{p}-1}\ell_{\|}^{\sfrac{1}{p}-\sfrac12} \lambda_{q+1}^{s}\label{e:w_est_deriv}\\
\norm{\partial_t w_{q+1}}_{H^2} 
&\les \ \lambda_{q+1}^{5}\label{e:w_t_est_deriv}
\end{align}
hold for $1 < p\leq \infty$ and $s \geq 0$.
\end{proposition}
Before turning to the proof of the proposition, we note that estimate \eqref{eq:increment:L2} and the inductive assumption \eqref{eq:v:L2:x} at level $q$ imply that 
\begin{align}
\norm{v_{q+1}}_{L^2} \leq 2 \delta_0^{\sfrac 12} - \delta_q^{\sfrac 12} + \delta_{q+1}^{\sfrac 12} \leq 2 \delta_0^{\sfrac 12} - \delta_{q+1}^{\sfrac 12}\,,
\label{eq:T*}
\end{align}
which is a consequence of $2 \lambda_q^\beta \leq \lambda_{q+1}^{\beta}$. Thus, \eqref{eq:v:L2:x} holds at level $q+1$. 
Similarly, from \eqref{eq:vq:global:2} and \eqref{e:w_est_deriv} with $s=3$ and $p=2$, and the parameter inequality \eqref{eq:useful:too}, we conclude 
\begin{align}
\norm{v_{q+1}}_{H^3} \les \lambda_q^{4} +\tau_{q+1}^{-2} \lambda_{q+1}^3 \les \lambda_{q+1}^{\frac{4}{b} + 3 + \frac{5-4\alpha}{50}} \les \lambda_{q+1}^{\frac 72} \leq \lambda_{q+1}^{4} \, ,
\label{eq:T*T*}
\end{align}
where we have used that $b$ is large and that $\alpha \in [1,\sfrac 54)$. The remaining power of $\lambda_{q+1}^{-\sfrac{1}{2}}$ may be used to absorb the implicit constant, and thus \eqref{eq:v:H3:x} holds also at level $q+1$.

Similarly to \eqref{eq:T*T*}, we establish two bounds which will be useful  in Section \ref{sec:stress} for  the proof of Corollary~\ref{cor:stress:main}.
First, from \eqref{e:w_est_deriv} with $s= \sfrac 92$ and $p=2$, and \eqref{eq:vq:1} with $M=0$ and $N=2$,  it follows that
\begin{align}
\norm{v_{q+1}}_{L^\infty(\sfrac T3, \sfrac{2T}{3};H^{\sfrac 92})} \leq \norm{w_{q+1}}_{H^{\sfrac 92}} + \norm{\overline v_{q}}_{L^\infty(\sfrac T3, \sfrac{2T}{3};H^{5})} \les  \tau_{q+1}^{-2}  \lambda_{q+1}^{\sfrac 92} + \tau_{q+1}^{-2}\lambda_q^{4} \les \lambda_{q+1}^{5}\,.
\label{eq:T*T*T*}
\end{align}
Here we have also used the parameter inequality \eqref{eq:useful:too}. Similarly, by \eqref{e:w_t_est_deriv} and the bound \eqref{eq:vq:1} with $M=1$ and $N=0$ we obtain 
\begin{align}
\norm{\partial_t v_{q+1}}_{L^\infty(\sfrac T3, \sfrac{2T}{3};H^{2})}  \leq \norm{\partial_t w_{q+1}}_{H^2} +  \norm{\partial_t \overline v_q}_{L^\infty(\sfrac T3, \sfrac{2T}{3};H^{3})} \les \lambda_{q+1}^{5} +  \tau_{q+1}^{-1} \lambda_q^{4} \les \lambda_{q+1}^{5}
\, .
\label{eq:T*T*T*T*}
\end{align}

\begin{proof}[Proof of Proposition~\ref{p:perturbation}]
The estimates \eqref{eq:w:q+1:L2} and \eqref{eq:increment:L2} are a direct consequence of the already established bounds and the definitions \eqref{eq:w:q+1:def}  and \eqref{eq:v:q+1:def}. Indeed, combining \eqref{e:wp_est} with \eqref{e:wct_N_est} with $p=2$ and $N=0$, we conclude that 
\[
\norm{w_{q+1}}_{L^2} \delta_{q+1}^{-\sfrac 12} \leq \frac 12 + \lambda_{q+1}^{\beta - \frac{5-4\alpha}{16}} \leq \frac 34\,,
\] 
since $\beta$ is sufficiently small (see~\eqref{eq:new:beta:b}). From \eqref{eq:w:q+1:L2} and \eqref{eq:vq:vell} we obtain 
\[
\norm{v_{q+1} - v_q}_{L^2} \leq \norm{\overline v_q - v_q}_{L^2} + \norm{w_{q+1}}_{L^2} \leq \delta_{q+1}^{\sfrac 12} \, ,
\] 
as desired.
The estimate \eqref{e:w_est_deriv} with non-integer values of $s$ follows by interpolation from the case $s  \in \N$.   Comparing \eqref{e:wp_N_est} with the second inequality in \eqref{e:wct_N_est}, we see that the bound for the principal corrector is the worst, since $\lambda_{q+1}^{-\frac{5-4\alpha}{16}} \leq 1 \leq \tau_{q+1}^{-1}$, and thus \eqref{e:w_est_deriv} follows directly. 

Thus it is left to prove \eqref{e:w_t_est_deriv}. An estimate on $\partial_t w_{q+1}^{(p)}$ will clearly dominate an estimate on $\partial_t w_{q+1}^{(c)}$. Hence it suffices to estimate $\partial_t w_{q+1}^{(p)}$ and $\partial_t w_{q+1}^{(c)}$. First consider $\partial_t w_{q+1}^{(p)}$. From the bound \eqref{e:W_Lp_bnd}  with $N=2, M=1, p=2$, estimate \eqref{e:a_est_CN} with $N=3$, and the definition \eqref{e:param_def} of $\mu$, we obtain
\begin{align*}
\norm{\partial_t w_{q+1}^{(p)}}_{H^2}  
&\lesssim \sum_{i}\sum_{\xi\in \Lambda_{(i)}}\norm{a_{(\xi)}}_{C^{3}_{x,t}}\norm{\partial_t   W_{(\xi)}}_{H^2}  \\
&\lesssim \sum_{i=0}^{i_{\rm max}} \tau_{q+1}^{-10} \ell_{\perp}\ell_{\|}^{-1}\lambda_{q+1}^3\mu\\
&\les  \tau_{q+1}^{-11} \lambda_{q+1}^{2\alpha+2}\\
&\les  \lambda_{q+1}^{5}
\, .
\end{align*}
where in the last inequality we used that \eqref{eq:useful:too} provides an upper bound for $\tau_{q+1}^{-1}$, and that $\alpha < \sfrac 54$. In order to estimate $\partial_t w_{q+1}^{(t)}$ we use estimates \eqref{e:phi_Lp_bnd} and \eqref{e:psi_Lp_bnd}, Fubini, and the bound \eqref{e:a_est_CN} to obtain
\begin{align*}
\norm{\partial_t w_{q+1}^{(t)}}_{H^2} 
&\lesssim \mu^{-1} \sum_{i}\sum_{\xi\in \Lambda_{(i)}} \left( \norm{a_{(\xi)}^2}_{C^{2}_{x,t}}  \norm{\phi^2_{(\xi)} \partial_t \psi^2_{(\xi)}}_{H^2} + \norm{a_{(\xi)}^2}_{C^{3}_{x,t}} \norm{\phi^2_{(\xi)} \psi^2_{(\xi)}}_{H^2} \right) \\
&\lesssim \mu^{-1}  \sum_{i=0}^{i_{\rm max}}\sum_{\xi\in \Lambda_{(i)}} \sum_{N=0}^{2} \left( \tau_{q+1}^{-8}  \norm{\phi^2_{(\xi)}}_{H^N} \norm{\partial_t \psi^2_{(\xi)}}_{H^{2-N}} +  \tau_{q+1}^{-11} \norm{\phi^2_{(\xi)}}_{H^N} \norm{\psi^2_{(\xi)}}_{H^{N-2}} \right) \\
&\lesssim \mu^{-1} \sum_{i=0}^{i_{\rm max}} \left( \tau_{q+1}^{-8} \ell_{\perp}^{-1}\ell_{\|}^{-\sfrac12}\left(\frac{\ell_{\perp}\lambda_{q+1}\mu}{\ell_{\|}}\right)\lambda_{q+1}^2 +\tau_{q+1}^{-11} \ell_{\perp}^{-1}\ell_{\|}^{-\sfrac12} \lambda_{q+1}^2 \right) \\
&\les  \tau_{q+1}^{-9}  \ell_{\|}^{-\sfrac32} \lambda_{q+1}^3 +\tau_{q+1}^{-12} \mu^{-1} \ell_{\perp}^{-1}\ell_{\|}^{-\sfrac12} \lambda_{q+1}^2 \\
& \lesssim \lambda_{q+1}^{5}  \, .
\end{align*}
Here we have used explicitly the parameter choice~\eqref{e:param_def}, the parameter inequality~\eqref{eq:useful}, the first bound in \eqref{eq:useful:too}, the bound $\ell_{\|}^{-1}\leq \ell_{\perp}^{-1}\leq \lambda_{q+1}$, and the inequality $i_{\rm max} \les \tau_{q+1}^{-1}$.
\end{proof}

\section{Convex integration step: the Reynolds stress}\label{sec:stress}

The main result of this section may be summarized as:
\begin{proposition}
\label{prop:stress:main}
There exists an $\eps_R>0$ sufficiently small, and a  {parameter} $p>1$ sufficiently close to $1$, depending only on $\alpha$, $b$, and $\beta$, such that the following holds: There exists a traceless symmetric $2$ tensor $\widetilde R$ and a scalar pressure field $\widetilde p$, defined implicitly in \eqref{eq:tilde:R:1} below, satisfying
\begin{subequations}\label{eq:tilde:R:def}
\begin{align}
\partial_t v_{q+1}+\div\left(v_{q+1}\otimes v_{q+1}\right)+\nabla \widetilde p +(-\Delta )^{\alpha}v_{q+1} &=\div \widetilde R \, ,\\
\div v_{q+1} & = 0 \, .
\end{align}
\end{subequations}
Moreover $\tilde R$ obeys the bound 
\begin{align}
\norm{\widetilde R}_{L^p}\lesssim \lambda_{q+1}^{- 2 \eps_R}\delta_{q+2}
\label{eq:tilde:R:Lp}\,,
\end{align}
where the constant depends on the choice of $p$ and $\eps_R$, but is independent of $q$, and $\tilde R$ has the support property
\begin{align}
\supp \tilde R  \subset \T^3 \times \left \{ t\in [0,T] \colon {\rm dist}(t, \GG^{(q+1)}) > \tau_{q+1} \right\}.
\label{eq:tilde:R:supp}
\end{align}
\end{proposition}

An immediate consequence of Proposition~\ref{prop:stress:main} is that the desired inductive estimates \eqref{eq:Rq:L1} and \eqref{eq:Rq:H3:x} and the support property \eqref{eq:Rq:time:support} hold for the Reynolds stress $\RR_{q+1}$, which is defined as follows.
\begin{corollary}
\label{cor:stress:main}
There exists a traceless symmetric $2$ tensor $\mathring R_{q+1}$ and a scalar pressure field $p_{q+1}$ such that the triple
$(v_{q+1},p_{q+1},\RR_{q+1})$ solves the Navier-Stokes-Reynolds system \eqref{eq:NSE:Reynolds} at level $q+1$.
Moreover, the following bounds hold
\begin{subequations}
\begin{align}
\norm{\mathring R_{q+1}}_{L^1}&\leq \lambda_{q+1}^{- \eps_R}\delta_{q+2} \, ,
\label{e:R_final_L1}\\
\norm{\mathring R_{q+1}}_{H^{3}}&\leq \lambda_{q+1}^{7} \, ,
\label{e:R_final_H3}
\end{align}
\end{subequations}
and $\RR_{q+1} (t) = 0$ whenever ${\rm dist}(t,\GG^{(q+1)}) \leq \tau_{q+1}$.
\end{corollary}
\begin{proof}[Proof of Corollary~\ref{cor:stress:main}]
With $\widetilde R$ and $\widetilde p$ defined in Proposition~\ref{prop:stress:main}, we let 
\begin{align*}
\mathring R_{q+1}= {\mathcal R \Proj_H \div \widetilde R} \qquad \mbox{and} \qquad p_{q+1} = \tilde p - \Delta^{-1} \div \div \widetilde R\,.
\end{align*}
It follows from \eqref{eq:tilde:R:def} and the definitions of the inverse-divergence operator $\RSZ$ and of the Helmholtz projection $\Proj_H$ that the $(v_{q+1},p_{q+1},\RR_{q+1})$ solve the Navier-Stokes-Reynolds system \eqref{eq:NSE:Reynolds} at level $q+1$. Since the operator $\mathcal R \Proj_H \div$ is time-independent, the claimed support property for $\RR_{q+1}$, namely  \eqref{eq:Rq:time:support} at level $q+1$, follows directly from \eqref{eq:tilde:R:supp}.

With the parameter $p>1$ from Proposition~\ref{prop:stress:main}, using  that $\norm{{\mathcal R} \Proj_H \div }_{L^p\to L^p} \les 1$, we directly bound
\begin{align}\notag
\norm{\mathring R_{q+1}}_{L^1}&\lesssim \norm{\mathring R_{q+1}}_{L^p}\lesssim \norm{\widetilde R}_{L^p}\lesssim \lambda_{q+1}^{-2\eps_R}\delta_{q+2}\,.
\end{align}
The estimate \eqref{e:R_final_L1} then follows since the residual factor  $\lambda_{q+1}^{-\eps_R}$ can absorb any constant if we assume $a$ is sufficiently large.  In order to prove \eqref{e:R_final_H3}, we use equation~\eqref{eq:tilde:R:def}, the support property of $\RR_{q+1}$ which implies that $\supp \RR_{q+1} \subset \T^3 \times [\sfrac T3, \sfrac {2T}{3}]$, and the bounds \eqref{eq:T*}--\eqref{eq:T*T*T*T*}. Combining these, we obtain
\begin{align*}
\norm{\mathring R_{q+1}}_{H^3}&= \norm{\mathcal R \Proj_H (\div \widetilde R)}_{H^3}\\
&\les \norm{\partial_t v_{q+1}+\div(v_{q+1}\otimes v_{q+1}) +(-\Delta )^{\alpha}v_{q+1}}_{L^\infty(\sfrac T3,\sfrac{2T}{3};H^2)} \\
&\lesssim  \norm{\partial_t v_{q+1}}_{L^\infty(\sfrac T3,\sfrac{2T}{3};H^2)}+\norm{v_{q+1}\otimes v_{q+1}}_{H^{3}} + \norm{v_{q+1}}_{L^\infty(\sfrac T3,\sfrac{2T}{3};H^{\sfrac 92})} \\
&\lesssim   \norm{\partial_t v_{q+1}}_{L^\infty(\sfrac T3,\sfrac{2T}{3};H^2)} +\norm{v_{q+1}}_{H^{3}} \norm{v_{q+1}}_{L^\infty} + \norm{v_{q+1}}_{L^\infty(\sfrac T3,\sfrac{2T}{3};H^{\sfrac 92})} \\
&\lesssim  \norm{\partial_t v_{q+1}}_{L^\infty(\sfrac T3,\sfrac{2T}{3};H^2)}+\norm{v_{q+1}}_{H^{3}}^{\sfrac 32} \norm{v_{q+1}}_{L^2}^{\sfrac 12}+ \norm{v_{q+1}}_{L^\infty(\sfrac T3,\sfrac{2T}{3};H^{\sfrac 92})} \\
&\lesssim  \lambda_{q+1}^{5}+ \lambda_{q+1}^6 \delta_0^{\sfrac 14}+ \lambda_{q+1}^{5} \\
&\les \lambda_{q+1}^{\sfrac{13}{2}} \, .
\end{align*}
For the dissipative term we have used that $\alpha < \sfrac 54$, so that $2\alpha +2 <\sfrac 92$. Using the residual power of $\lambda_{q+1}^{- \sfrac 12}$ we may absorb any constants and thus \eqref{e:R_final_H3} follows.
\end{proof}

\subsection{Proof of Proposition~\ref{prop:stress:main}}
Recall that $v_{q+1} = w_{q+1} + \vo$, where $\vo$ is defined in Section~\ref{ss:proof_prop} and $(\overline v_q, {\mathring {\bar R}}_{q})$ solve \eqref{eq:NSE:Reynolds}. Using \eqref{eq:w:q+1:def} we obtain
\begin{align}
\div \tilde R - \nabla \tilde p
&=   (-\Delta)^{\alpha} w_{q+1} +  \partial_t (w^{(p)}_{q+1} +w^{(c)}_{q+1} )  + \div( \vo \otimes w_{q+1} + w_{q+1} \otimes \vo) \notag\\
&\qquad + \div\left((w_{q+1}^{(c)}+w_{q+1}^{(t)}) \otimes w_{q+1}+ w_{q+1}^{(p)} \otimes (w_{q+1}^{(c)}+w_{q+1}^{(t)}) \right) \notag\\
&\qquad + \div(w_{q+1}^{(p)} \otimes w_{q+1}^{(p)} + \RRO)+\partial_t w^{(t)}_{q+1} \notag\\
&=: \div \left( \tilde R_{\rm linear} + \tilde R_{\rm corrector} + \tilde R_{\rm oscillation}  \right) + \nabla q\,.
\label{eq:tilde:R:1}
\end{align}
Here, the linear error and corrector errors are defined by applying $\RSZ$ to the first and respectively second line of \eqref{eq:tilde:R:1}, while the oscillation error is defined in Section~\ref{sec:oscillation} below. The  zero mean pressure $q$ is defined implicitly in a unique  way.

Besides the already used inequalities between the parameters, $\ell_{\perp}$, $\ell_{\|}$ and $\lambda_{q+1}$, we shall use that if $p$ is sufficiently close to $1$ the following bounds hold:
\begin{align}
&\tau_{q+1}^{-5} \lambda_{q+1}^{2\alpha-1} \ell_{\perp}^{\sfrac{2}{p}-1}\ell_{\|}^{\sfrac{1}{p}-\sfrac12}+   \tau_{q+1}^{-5} \lambda_{q+1}^{1-2\alpha} \ell_{\perp}^{\sfrac{2}{p}-2}\ell_{\|}^{\sfrac{1}{p}-\sfrac 52} \notag\\
&\qquad + \tau_{q+1}^{-6}\lambda_{q+1}^{-1}\ell_{\perp}^{\sfrac{2}{p}-3}\ell_{\|}^{\sfrac{1}{p}-1}  + \tau_{q+1}^{-6} \lambda_{q+1}^{1-2\alpha} \ell_{\perp}^{\sfrac{2}{p}-1}\ell_{\|}^{\sfrac{1}{p}-2} \les \lambda_{q+1}^{-2\eps_R} \delta_{q+2}\,.
\label{eq:parameters}
\end{align}
To see this, we appeal to the bound \eqref{eq:useful:too} for $\tau_{q+1}^{-1}$ and the parameter choices~\eqref{e:param_def} to conclude that the left side of \eqref{eq:parameters} is bounded from above as
\begin{align}
& \ell_\perp^{\sfrac{2}{p}-2} \ell_{\|}^{\sfrac 1p - 1} \left( \tau_{q+1}^{-5} \lambda_{q+1}^{2\alpha-1} \ell_{\perp} \ell_{\|}^{\sfrac12}+   \tau_{q+1}^{- 5} \lambda_{q+1}^{1-2\alpha}  \ell_{\|}^{-\sfrac 32} + \tau_{q+1}^{-6}\lambda_{q+1}^{-1}\ell_{\perp}^{-1}   + \tau_{q+1}^{-6} \lambda_{q+1}^{1-2\alpha} \ell_{\perp} \ell_{\|}^{-1} \right) \notag\\
& \les \ell_\perp^{\sfrac{2}{p}-2} \ell_{\|}^{\sfrac 1p - 1} \tau_{q+1}^{-6}  \left( \lambda_{q+1}^{2\alpha-1} \ell_{\perp} \ell_{\|}^{\sfrac12}+    \lambda_{q+1}^{1-2\alpha}  \ell_{\|}^{-\sfrac 32} +  \lambda_{q+1}^{-1}\ell_{\perp}^{-1}   + \lambda_{q+1}^{1-2\alpha} \ell_{\perp} \ell_{\|}^{-1} \right) 
\notag\\
& \les \ell_\perp^{\sfrac{2}{p}-2} \ell_{\|}^{\sfrac 1p - 1} \lambda_{q+1}^{\frac{3(5-4\alpha)}{50}} \left(  \lambda_{q+1}^{-\frac{5-4\alpha}{12}}+    \lambda_{q+1}^{- \frac{5-4\alpha}{8}}   +  \lambda_{q+1}^{-\frac{5(5-4\alpha)}{24}}  + \lambda_{q+1}^{-\frac{28\alpha + 1}{24}} \right) \notag\\
&\les \ell_\perp^{\sfrac{2}{p}-2} \ell_{\|}^{\sfrac 1p - 1} \lambda_{q+1}^{- \frac{5-4\alpha}{50}} \notag\\
&\les \lambda_{q+1}^{- \frac{5-4\alpha}{100}}\notag
\end{align}
where in the last inequality we have chosen $p$ sufficiently close to $1$, depending only on $\alpha$. 
To conclude the proof of \eqref{eq:parameters}, note that
\begin{align}
\lambda_{q+1}^{-2\eps_R} \delta_{q+2} \geq  \lambda_{q+1}^{- 2 \eps_R} \lambda_{q+2}^{-2\beta} \geq \lambda_{q}^{- 2 \eps_R b -2\beta b^2}\,,\notag
\end{align}
and therefore if we ensure that $\eps_R$ and $\beta$ are sufficiently small, depending on $\alpha$ and $b$ only, such that 
\begin{align}
2 \eps_R b + 2 \beta b^2 \leq \frac{5-4\alpha}{100},
\label{eq:main:constraint:b}
\end{align}
then the three estimates above imply \eqref{eq:parameters}.

\subsubsection{The linear  error}
In order to prove \eqref{eq:tilde:R:Lp}, we first estimate the contributions to $\tilde R$ coming from $\tilde R_{\rm linear}$. Recalling~\eqref{e:curl_form}, the bounds \eqref{eq:vq:global:2}, \eqref{e:W_Lp_bnd}, \eqref{e:a_est_CN}, and \eqref{e:w_est_deriv}, we obtain
\begin{align}
\norm{ \tilde R_{\rm linear}}_{L^p}  
&\les \norm{{\mathcal R} (   (-\Delta)^{\alpha} w_{q+1} )}_{L^p} +\norm{\mathcal R(\partial_t (w^{(p)}_{q+1} +w^{(c)}_{q+1} ))}_{L^p}+ \norm{\mathcal R\div(\vo \otimes w_{q+1} + w_{q+1} \otimes \vo)}_{L^p}
\notag\\
& \les
\norm{w_{q+1}}_{W^{2\alpha-1,p}}+\sum_{i}\sum_{\xi\in \Lambda_{(i)}} \norm{\partial_t\mathcal R\curl\curl(a_{(\xi)}V_{(\xi)})}_{L^p}+\norm{\vo}_{L^{\infty}}\norm{w_{q+1}}_{L^{p}} 
\notag\\
&\les
(1+\norm{\vo}_{L^{\infty}}) \norm{w_{q+1}}_{W^{2\alpha-1,p}} +\sum_{i} \sum_{\xi\in \Lambda_{(i)}}\norm{\partial_t\curl(a_{(\xi)}V_{(\xi)})}_{L^p}\notag\\
&\les
(1+\norm{\vo}_{H^3})  \norm{w_{q+1}}_{W^{2\alpha-1,p}} + \sum_i \sum_{\xi\in \Lambda_{(i)}}  \left(\norm{a_{(\xi)}}_{C^1_{x,t}} \norm{\partial_t V_{(\xi)}}_{W^{1,p}}+ \norm{a_{(\xi)}}_{C^2_{x,t}}  \norm{V_{(\xi)}}_{W^{1,p}}\right) \notag\\
&\les \lambda_q^{4} \tau_{q+1}^{-2} \lambda_{q+1}^{2\alpha-1} \ell_{\perp}^{\sfrac{2}{p}-1}\ell_{\|}^{\sfrac{1}{p}-\sfrac12} +\tau_{q+1}^{-5}  \ell_{\perp}^{\sfrac{2}{p}-1}\ell_{\|}^{\sfrac{1}{p}-\sfrac12}\lambda_{q+1}^{-1}\left(\frac{\ell_{\perp}\lambda_{q+1}\mu}{\ell_{\|}}\right) +  \tau_{q+1}^{-8}  \ell_{\perp}^{\sfrac{2}{p}-1}\ell_{\|}^{\sfrac{1}{p}-\sfrac12}\lambda_{q+1}^{-1} \notag\\
&\les  \tau_{q+1}^{-5} \lambda_{q+1}^{2\alpha-1} \ell_{\perp}^{\sfrac{2}{p}-1}\ell_{\|}^{\sfrac{1}{p}-\sfrac12}\label{eq:tilde:R:2} \, .
\end{align}
Here we have used the definition of $\mu$ from \eqref{e:param_def}, and the parameter inequalities $\lambda_q^{4} \les \tau_{q+1}^{-1} \les \lambda_{q+1}^{\sfrac{\alpha}{2}}$. By \eqref{eq:parameters}, the above estimate is consistent with \eqref{eq:tilde:R:Lp}.

\subsubsection{Corrector error}
Next we turn to the  errors involving correctors. Appealing to estimates \eqref{e:wp_N_est}  {and} \eqref{e:wct_N_est} of Proposition~\ref{prop:perturbation}, we have
\begin{align*}
\norm{ \tilde R_{\rm corrector}}_{L^p} 
&\leq \norm{\mathcal R \div\left((w_{q+1}^{(c)}+w_{q+1}^{(t)}) \otimes w_{q+1}+ w_{q+1}^{(p)} \otimes (w_{q+1}^{(c)}+w_{q+1}^{(t)}) \right)}_{L^p}\\
&\lesssim   \norm{ w_{q+1}^{(c)}+w_{q+1}^{(t)} }_{L^{2p}}   \norm{ w_{q+1}}_{L^{2p}} + \norm{w_{q+1}^{(p)}  }_{L^{2p}}   \norm{ w_{q+1}^{(c)}+w_{q+1}^{(t)} }_{L^{2p}}
\\
&\lesssim \tau_{q+1}^{-5} \mu^{-1} \ell_{\perp}^{\sfrac{2}{p}-3}\ell_{\|}^{\sfrac{1}{p}-\sfrac32} 
\\
&\lesssim  \tau_{q+1}^{-5} \lambda_{q+1}^{1-2\alpha} \ell_{\perp}^{\sfrac{2}{p}-2}\ell_{\|}^{\sfrac{1}{p}-\sfrac 52}
\,.
\end{align*}
In the last inequality we have appealed to the definition~\eqref{e:param_def}.  Due to \eqref{eq:parameters} this estimate is sufficient for \eqref{eq:tilde:R:Lp}.

\subsubsection{Oscillation error}
\label{sec:oscillation}
In this section we estimate the remaining error, $\tilde R_{\rm oscillation}$, which obeys
\begin{align}
&\div \left( \tilde R_{\rm oscillation} \right) + \nabla P =\div\left(w^{(p)}_{q+1}\otimes w^{(p)}_{q+1} + \RRO\right)+\partial_t w^{(t)}_{q+1}  \,,
\label{eq:Ronaldo}
\end{align}
where $P$ is a suitable pressure.
From the definition of $w_{q+1}^{(p)}$ in \eqref{eq:w:q+1:p:def} and of the coefficients $a_{(\xi)}$ in \eqref{eq:a:oxi:def}, using the disjoint support property of the intermittent jets \eqref{eq:Messi}, the fact that $\Lambda_{(1)} \cap \Lambda_{(2)} = \emptyset$, and appealing to the identity \eqref{e:WW_id}, we have
\begin{align*}
&\div\left(w^{(p)}_{q+1}\otimes w^{(p)}_{q+1} \right) +   \div  {\RRO}  \\
&\qquad = \sum_{i,j}\sum_{\xi\in \Lambda_{(i)}, \xi'\in\Lambda_{(j)}}\div\left(a_{(\xi)} a_{(\xi')} W_{(\xi)}\otimes W_{(\xi')}\right)+ \div \RRO 
  \\
&\qquad = \sum_{i}\sum_{\xi\in \Lambda_{(i)}}\div\left(a_{(\xi)}^2 W_{(\xi)}\otimes W_{(\xi)}\right)+ \div \RRO 
 \\
&\qquad = \sum_{i,j}\sum_{\xi\in \Lambda_{(i)}}\div\left(a_{(\xi)}^2\left(W_{(\xi)}\otimes W_{(\xi)}-\fint_{\T^3}W_{(\xi)}\otimes W_{(\xi)} \, dx \right)\right)  +\nabla \left( \theta^2 \sum_{i\geq 0} \rho_i \chi_{(i)}^2  \right)
 \\
&\qquad =\sum_{i}\sum_{\xi\in \Lambda_{(i)}}\underbrace{\div\left(a_{(\xi)}^2\mathbb P_{{\geq} \sfrac{\lambda_{q+1}\ell_{\perp}}{2}}\left( W_{(\xi)}\otimes W_{(\xi)} \right)\right)}_{E_{(\xi)}} \; +\; \nabla \left( \theta^2 \sum_{i\geq 0} \rho_i \chi_{(i)}^2  \right) \, . 
\end{align*}
Here we use that since $W_{(\xi)}$ is $\left(\sfrac{\T}{\ell_{\perp} \lambda}\right)^3$-periodic, the minimal separation between active frequencies of ${W}_{(\xi)} \otimes { W}_{(\xi)}$ and the $0$ frequency is given by $\lambda_{q+1} \ell_{\perp}$. That is, $\Proj_{\neq 0}(W_{(\xi)} \otimes W_{(\xi)}) = \Proj_{\geq \lambda_{q+1} \ell_\perp/2}(W_{(\xi)} \otimes W_{(\xi)})$. We further split
\begin{align*}
E_{(\xi)}&=  \mathbb P_{\neq 0}\left(\mathbb P_{{\geq}\sfrac{\lambda_{q+1}\ell_{\perp}}{2}}\left( W_{(\xi)}\otimes  W_{(\xi)}\right)\nabla \left( a_{(\xi)}^2\right) \right) + \Proj_{\neq 0}\left(a_{(\xi)}^2\div \left( W_{(\xi)}\otimes   W_{(\xi)}\right) \right)\\
&=: E_{(\xi,1)}+E_{(\xi,2)} \, .
\end{align*}
The term $\RSZ E_{(\xi,1)}$, which is the first contribution to $\tilde R_{\rm oscillation}$,  is estimated by using that  the coefficient functions $a_{(\xi)}$ are essentially frequency localized inside of the ball of radius $\tau_{q+1}^{-3} \ll \lambda_{q+1} \ell_\perp$, in view of \eqref{e:a_est_CN}. More precisely, by Lemma~\ref{lem:many:bounds} we are justified to use~\cite[Lemma B.1]{BV}, with the parameter choices   $\lambda = \tau_{q+1}^{-3}$, $C_a = \tau_{q+1}^{-5}$, $\kappa = \sfrac{\lambda_{q+1} \ell_{\perp}}{2}$, and $L$ sufficiently large, to conclude
\begin{align*}
\norm{\mathcal R E_{(\xi,1)}}_{L^p}&\les \norm{\abs{\nabla}^{-1}  E_{(\xi,1)}}_{L^p}
 \\
& \les\norm{\abs{\nabla}^{-1} \mathbb P_{\neq 0}\left(\mathbb P_{{\geq}\sfrac{\lambda_{q+1}\ell_{\perp}}{2}}\left(  W_{(\xi)}\otimes   W_{(\xi)}\right)\nabla \left( a_{(\xi)}^2\right) \right)}_{L^p}
 \\
&\les \frac{\tau_{q+1}^{-5}}{\lambda_{q+1}\ell_{\perp}}\left(1+\frac{\tau_{q+1}^{-6}}{\left(\tau_{q+1}^3 \lambda_{q+1} \ell_{\perp}\right)^{L-2}}\right)  \norm{  W_{(\xi)}\otimes   W_{(\xi)}}_{L^p} \\
&\les   \frac{\tau_{q+1}^{-5}}{\lambda_{q+1}\ell_{\perp}}  \norm{  W_{(\xi)}}_{L^{2p}} \norm{  W_{(\xi)}}_{L^{2p}} \\
&\les \tau_{q+1}^{-5}\lambda_{q+1}^{-1}\ell_{\perp}^{\sfrac{2}{p}-3}\ell_{\|}^{\sfrac{1}{p}-1}\, .
\end{align*}
In the last inequality above we have used estimate~\eqref{e:W_Lp_bnd}, and in the second to last inequality we have used that  by taking $L$ sufficiently large, for instance {$L=4$ is sufficient} in view of the first inequality in \eqref{eq:useful:too} and the definition of $\ell_\perp$ in \eqref{e:param_def}, we have $\tau_{q+1}^{-6} (\tau_{q+1}^{3} \lambda_{q+1} \ell_{\perp})^{2-L} \les 1$. Summing these contributions over $0\leq i \leq i_{\rm max}$ costs an additional factor of $\tau_{q+1}^{-1}$, and from the third term in \eqref{eq:parameters} we obtain the the bound for $\RSZ E_{(\xi,1)}$ is consistent with \eqref{eq:tilde:R:Lp}.

We are left to estimate the contribution from the $E_{(\xi,2)}$ term. From identity \eqref{eq:useful:2} we have that 
\begin{align*}
E_{(\xi,2)}&=\frac{1}{ \mu}\mathbb P_{\neq 0}\left(a_{(\xi)}^2\phi^2_{(\xi)}\partial_t \psi^2_{(\xi)}\xi\right) = \frac{1}{\mu} \partial_t \Proj_{\neq 0}\left(a_{(\xi)}^2 \phi_{(\xi)}^2 \psi_{(\xi)}^2 \xi\right) - \frac{1}{\mu}  \Proj_{\neq 0}\left( (\partial_t a_{(\xi)}^2) \phi_{(\xi)}^2 \psi_{(\xi)}^2 \xi\right)\, .
\end{align*}
Hence, summing in $\xi$ and $i$, pairing with the $\partial_t w^{(t)}_{q+1}$ present in \eqref{eq:Ronaldo}, recalling the definition of $w^{(t)}_{q+1}$ in \eqref{e:temporal_corrector}, and noting that $\Id -\Proj_H = \nabla (\Delta^{-1} \div)$, we obtain
\begin{align}
\sum_{i}\sum_{\xi\in \Lambda_{(i)}}E_{(\xi,2)}+\partial_t w^{(t)}_{q+1}
&= \frac{1}{\mu} \sum_{i} \sum_{\xi \in \Lambda_{(i)}} (\Id - \Proj_H) \partial_t \Proj_{\neq 0}\left(a_{(\xi)}^2 \phi_{(\xi)}^2 \psi_{(\xi)}^2 \xi\right) - \frac{1}{\mu}  \sum_{i}\sum_{\xi\in \Lambda_{(i)}} \Proj_{\neq 0}\left(\partial_t (a_{(\xi)}^2)\phi^2_{(\xi)} \psi^2_{(\xi)}\xi\right) \notag\\
&=\nabla q - \frac{1}{\mu}  \sum_{i}\sum_{\xi\in \Lambda_{(i)}} \Proj_{\neq 0}\left(\partial_t (a_{(\xi)}^2)\phi^2_{(\xi)} \psi^2_{(\xi)}\xi\right) \, ,
\label{eq:peanuts}
\end{align}
where $q = \frac{1}{\mu} \sum_{i} \sum_{\xi \in \Lambda_{(i)}} \Delta^{-1} \div \partial_t \Proj_{\neq 0}\left(a_{(\xi)}^2 \phi_{(\xi)}^2 \psi_{(\xi)}^2 \xi\right) $ is a pressure term. 
At last, we estimate the second contribution to $\tilde R_{\rm oscillation}$ by  using \eqref{e:phi_Lp_bnd}, \eqref{e:psi_Lp_bnd}, Fubini, \eqref{eq:i:max:bound}, and \eqref{e:a_est_CN}, to obtain
\begin{align}
\norm{\mathcal R\left(\frac{1}{\mu} \sum_{i}\sum_{\xi\in \Lambda_{(i)}} \Proj_{\neq 0}\left(\partial_t (a_{(\xi)}^2)\phi^2_{(\xi)} \psi^2_{(\xi)}\xi\right) \right)}_{L^p}&\les \frac{1}{\mu}\sum_{i}\sum_{\xi\in \Lambda_{(i)}}\norm{\partial_t (a_{(\xi)}^2)\phi^2_{(\xi)} \psi^2_{(\xi)}\xi}_{L^p}\notag\\
&\les \frac{1}{\mu}\sum_{i}\sum_{\xi\in \Lambda_{(i)}}\norm{a_{(\xi)}}_{C^{1}_{t,x}}\norm{a_{(\xi)}}_{L^{\infty}}\norm{\phi_{(\xi)}}_{L^{2p}}^2\norm{\psi_{(\xi)}}_{L^{2p}}^2\notag\\
&\les \mu^{-1} \sum_{i=0}^{i_{\rm max}}\tau_{q+1}^{-5}  \ell_{\perp}^{\sfrac{2}{p}-2}\ell_{\|}^{\sfrac{1}{p}-1}\notag\\
&\les \tau_{q+1}^{-6}\mu^{-1}\ell_{\perp}^{\sfrac{2}{p}-2}\ell_{\|}^{\sfrac{1}{p}-1} \notag\\
&\les \tau_{q+1}^{-6} \lambda_{q+1}^{1-2\alpha} \ell_{\perp}^{\sfrac{2}{p}-1}\ell_{\|}^{\sfrac{1}{p}-2}
\, .
\label{e:sutter_home2}
\end{align}
In the last equality above we have used the definition of $\mu$. Using the bound for the last term in \eqref{eq:parameters}, we conclude that the above estimate is consistent with \eqref{eq:tilde:R:Lp}, which shows that $\tilde R_{\rm oscillation}$ also obeys this inequality.

\subsubsection{The temporal support of $\tilde R$}
In order to conclude the proof of Proposition~\ref{prop:stress:main}, we need to show that \eqref{eq:tilde:R:supp} holds.
From \eqref{eq:tilde:R:1} it follows that 
\[
\supp \tilde R \subset \supp w_{q+1}^{(p)} \cup \supp w_{q+1}^{(c)} \cup \supp w_{q+1}^{(t)} \cup \supp \RRO.
\]
By \eqref{eq:bar:Rq:supp} we know that $\RRO(t) = 0$ whenever ${\rm dist}(t, \GG^{(q+1)})\leq 2 \tau_{q+1}$, while by property \eqref{e:a_support} we have that $a_{(\xi)}(t) = 0$ whenever  ${\rm dist}(t, \GG^{(q+1)})\leq  \tau_{q+1}$. By their definitions, the principal \eqref{eq:w:q+1:p:def}, incompressibility \eqref{eq:w:q+1:c:def}, and temporal correctors \eqref{e:temporal_corrector}, are composed only of terms which contain the coefficient functions $a_{(\xi)}$, and thus, similarly to \eqref{e:w_support} we have that $w_{q+1}^{(p)}(t) = w_{q+1}^{(c)} (t) = w_{q+1}^{(t)}(t) = 0$ whenever  ${\rm dist}(t, \GG^{(q+1)})\leq  \tau_{q+1}$. This proves \eqref{eq:tilde:R:supp}.

\section*{Acknowledgments} 
The work of T.B. has been partially supported by the NSF grant DMS-1600868.
V.V. was partially supported by the NSF grant DMS-1652134  and an Alfred P. Sloan Research Fellowship.

\bibliographystyle{abbrv}

\end{document}